\documentclass{amsart}
\usepackage[utf8]{inputenc}
\usepackage{amssymb}
\usepackage{hyperref}
\usepackage[final]{showkeys} 

\input xy
\xyoption{all}

\theoremstyle{definition}
\newtheorem{mydef}{Definition}[section]
\newtheorem{lem}[mydef]{Lemma}

\newtheorem{thm}[mydef]{Theorem}

\newtheorem{conjecture}[mydef]{Conjecture}
\newtheorem{cor}[mydef]{Corollary}

\newtheorem{question}[mydef]{Question}

\newtheorem{defin}[mydef]{Definition}
\newtheorem{example}[mydef]{Example}

\newtheorem{remark}[mydef]{Remark}

\newcommand{\fct}[2]{{}^{#1}#2}

\newcommand\Mod{\operatorname{\textbf{Mod}}}

\newcommand\Hom{\operatorname{Hom}}

\newcommand\id{\operatorname{id}}

\newcommand\Set{\operatorname{\bf Set}}

\newcommand\Ban{\operatorname{\bf Ban}}

\newcommand\Str{\operatorname{\bf Str}}

\newcommand\Elem{\operatorname{\bf Elem}}

\newcommand\Ab{\operatorname{\bf Ab}}
\newcommand\Gra{\operatorname{\bf Gra}}
\newcommand\Grp{\operatorname{\bf Grp}}
\newcommand\Bool{\operatorname{\bf Bool}}

\newcommand\Lin{\operatorname{\bf Lin}}

\newcommand{\Rmod}{R\text{-}\operatorname{\bf {Mod}}}

\newcommand\cof{\operatorname{cof}}

\newcommand\colim{\operatorname{colim}}
\newcommand\dom{\operatorname{dom}}

\newcommand{\Ff}{\mathcal{F}}

\newcommand\ck{\mathcal {K}}
\newcommand\cl{\mathcal {L}}

\newcommand{\ba}{\bar{a}}
\newcommand{\bb}{\bar{b}}
\newcommand{\bc}{\bar{c}}

\newcommand{\bx}{\bar{x}}
\newcommand{\by}{\bar{y}}
\newcommand{\bz}{\bar{z}}

\newcommand{\sea}{\mathfrak{C}}
\newcommand{\ran}{\operatorname{ran}}

\newcommand{\cf}[1]{\text{cf} (#1)}
\newcommand{\seq}[1]{\langle #1 \rangle}
\newcommand{\rest}{\upharpoonright}

\newcommand{\leap}[1]{\le_{#1}}
\newcommand{\ltap}[1]{<_{#1}}

\newcommand{\lta}{\ltap{\K}}
\newcommand{\lea}{\leap{\K}}

\newcommand{\K}{\mathbf{K}}

\newbox\noforkbox \newdimen\forklinewidth
\forklinewidth=0.3pt \setbox0\hbox{$\textstyle\smile$}
\setbox1\hbox to \wd0{\hfil\vrule width \forklinewidth depth-2pt
 height 10pt \hfil}
\wd1=0 cm \setbox\noforkbox\hbox{\lower 2pt\box1\lower
2pt\box0\relax}
\def\unionstick{\mathop{\copy\noforkbox}\limits}
\newcommand{\nf}{\unionstick}

\newcommand{\smallnf}{\downarrow}

\newcommand{\nfs}[4]{#2 \nf_{#1}^{#4} #3}

\def\1nf{\unionstick^{(1)}}

\def\2nf{\unionstick^{(2)}}
\def\3nf{\unionstick^{(3)}}

\newcommand{\tp}{\mathbf{tp}}

\newcommand{\Ss}{\gS}
\newcommand{\gS}{\textbf{S}}

\newcommand{\ehanf}[1]{\beth_{\left(2^{#1}\right)^+}}

\newcommand{\Ll}{\mathbb{L}}

\newcommand{\tlt}{\triangleleft}
\newcommand{\ntlt}{\ntriangleleft}

\newcommand{\ccl}{\operatorname{cl}}

\newcommand{\Mm}{\mathcal{M}}
\newcommand{\Nn}{\mathcal{N}}
\newcommand{\Hh}{\mathcal{H}}
\newcommand{\Xx}{\mathcal{X}}
\newcommand{\LS}{\text{LS}}

\newcommand{\Pp}{\mathbb{P}}

\setbox0\hbox{$\textstyle\smile$}
\setbox1\hbox to \wd0{\hfil{\sl /\/}\hfil} \setbox2\hbox to
\wd0{\hfil\vrule height 10pt depth -2pt width
               \forklinewidth\hfil}
\newbox\doesforkbox
\setbox\doesforkbox\hbox{\lower 0pt\box1 \lower
2pt\box2\lower2pt\box0\relax}
\def\nunionstick{\mathop{\copy\doesforkbox}\limits}
\newcommand{\nnf}{\nunionstick}

\title{Accessible categories, set theory, and model theory: an invitation}
\date{\today \\AMS 2010 Subject Classification: Primary: 18C35. Secondary: 03C45, 03C48, 03C52, 03C55, 03C75, 03E05, 03E55.}
\keywords{accessible categories, internal size, stable independence, abstract elementary class}

\author{Sebastien Vasey}
\email{sebv@math.harvard.edu}
\urladdr{http://math.harvard.edu/\textasciitilde sebv/}
\address{Department of Mathematics \\ Harvard University \\ Cambridge, Massachusetts, USA}

\begin{document}

\begin{abstract}
  We give a self-contained introduction to accessible categories and how they shed light on both model- and set-theoretic questions. We survey for example recent developments on the study of presentability ranks, a notion of cardinality localized to a given category, as well as stable independence, a generalization of pushouts and model-theoretic forking that may interest mathematicians at large. We give many examples, including recently discovered connections with homotopy theory and homological algebra. We also discuss concrete versions of accessible categories (such as abstract elementary classes), and how they allow nontrivial ``element  by element'' constructions. We conclude with a new proof of the equivalence between saturated and homogeneous which does not use the coherence axiom of abstract elementary classes.
\end{abstract}

\maketitle

\tableofcontents

\section{Introduction}

Category theory, model theory, and set theory are all foundational branches of mathematics. In this paper, I will attempt to give a taste of the food one obtains when mixing the three of them together. I do not really know how to name this dish but, for the present paper at least and at the risk of scaring researchers from all three fields, I will call it \emph{categorical model theory}. This terminology appears in the title of Makkai and Paré's book \cite{makkai-pare}, but unfortunately leaves out the set-theoretic aspect. Let me then reassure the set theorists: set theory is very much part of categorical model theory (see Section \ref{set-sec}).

So what is categorical model theory? To start with, a \emph{model} (or \emph{structure}) consists of a set (the ``universe'') together with relation and functions on that set. For example the set of integers together with addition (seen as a binary function) is a model. Part of model theory studies the \emph{definable subsets} of such structures: sets that can be expressed from the relations and functions using simple set operations such as union, intersection, complement, and projection. This approach has had quite a bit of success (one of the easiest outcomes to describe is the theory of o-minimality \cite{vandendries, pila-wilkie}). It \emph{is} moreover possible to study definable sets in a categorical setup (a field named categorical logic, see e.g.\ \cite{makkai-reyes}). Nevertheless, this is \emph{not at all} what we will do here.

To make a loose analogy, categorical model theory relates to studying definable sets of a fixed structure in roughly the same way that thermodynamics relates to quantum physics: category theory looks at \emph{an entire class of structures}, together with morphisms between them. Thus instead of only looking at $(\mathbb{Z}, +)$, we will look at the class of abelian groups. What should be the morphisms in that class? The classical answer is that it should be the group homomorphisms -- the maps preserving the additive structure -- yielding the well known category $\Ab$. There are however other possible choice of morphisms. For example, an abelian group is a structure $(A, +)$ satisfying certain axioms that can be expressed using first-order sentences (including for example $(\forall x \forall y) (x + y = y + x)$). In model-theoretic terminology, it is a model of the theory (i.e.\ set of sentences) $T_{ab}$ of abelian groups. When working abstractly with a class of models of a given first-order theory $T$, what should be the ``right'' notion of morphism? One answer that is well studied in model theory is the notion of an \emph{elementary embedding}: an embedding preserving all formulas, possibly with parameters. The elementary embeddings preserve a lot more than the homomorphisms (in particular, they preserve the definable sets) and yield another category, $\Elem (T)$, whose objects are the models of $T$ and whose morphisms are elementary embeddings. Tarski and Vaught showed that $\Elem (T)$ always has directed colimits (the class $\Elem (T)$ is closed under unions of chains ordered by elementary substructures), while this is not always true for the category $\Mod (T)$ of models of $T$ with homomorphisms.

An elementary embedding of abelian groups is quite difficult to describe, and for a category theorist the category $\Elem (T_{ab})$ of abelian groups with elementary embeddings is quite poorly behaved (all morphisms are monomorphisms, so it lacks a lot of limits and colimits; it does not have any notion of quotient for example). On the other hand, from the set-theoretic point of view the category $\Elem (T_{ab})$ is quite interesting and easy to work with, precisely because all morphisms are monomorphisms (for many purposes, one can think of the maps as inclusions, and think of this category as a kind of concrete poset, with certain isomorphisms between the elements). Still, there may be other interesting types of monomorphisms (for abelian groups, the injective homomorphisms, or the pure morphisms, are obvious choices). Another problem that comes up with the traditional model-theoretic approach is that interesting classes may (provably) not be models of a first-order theory: consider torsion abelian groups, Banach spaces, Zilber's pseudo-exponential fields \cite{zilber-pseudoexp}, etc. Moving to infinitary logics may partly fix this second problem but one is still somewhat tied to notions of elementary embeddings. There are other issues with the traditional ``Tarskian'' definition of a model as a set with functions and relations on it (see Macintyre's essay \cite{macintyre-bsl}). Several natural categories, such as arrow categories and more generally categories of functors, cannot easily be described within the Tarskian frame, for example. Let us then agree to take a categorical approach: categorical model theory will study categories that, in some sense, look or behave like the categories of models studied in classical model theory.

What kind of ``classical model-theoretic'' behavior are we looking for? Two basic results of model theory are the aforementioned Tarski-Vaught chain theorem (closure under chains of elementary substructures) and the downward Löwenheim-Skolem-Tarski theorem (every structure has a ``small'' elementary substructure). Category-theoretically, the first can be described by the existence of certain colimits (of chains, or more generally of directed diagrams). The second needs a notion of ``smallness'', i.e.\ really a notion of size. In a category, objects don't have a ``universe'', and even if they do, the cardinality of the universe may not tell us much. Nevertheless, it \emph{is} possible to define a notion of size, the presentability rank, by looking at how an object embeds into sufficiently directed colimits (as a simple example, a finite set contained inside an infinite union will be contained inside a component of the union; this property could serve as a definition of finiteness -- thus presentability ranks generalize cardinalities to other categories than the category of sets). After making this idea precise, we arrive at the definition of an \emph{accessible category}, one of the main frameworks of categorical model theory. Roughly, it is a category with all sufficiently directed colimits so that any object is a directed colimits of a fixed set of ``small'' subobjects (see Definition \ref{acc-def} here). For example, a set is a union of its finite subsets, an abelian group is a union of its finitely generated subgroups, etc. Accessible categories were first defined by Lair \cite{lair-accessible}, their theory was created by Makkai-Paré \cite{makkai-pare} and further developed in Adámek-Rosický \cite{adamek-rosicky}. I invite the reader to consult these references for more on the history. There are \emph{a lot} of examples of accessible categories, including all those discussed before, $\Ab$, Banach spaces with contractions (or isometries), $\Elem (T)$ for any first-order theory $T$, and more generally $\Elem (\phi)$ for $\phi$ an $\Ll_{\infty, \infty}$-formula\footnote{Roughly, $\Ll_{\infty, \infty}$ is the logic where we allow infinitary quantifications, conjunctions, and disjunctions. See Section \ref{struct-sec}.}  and a suitable notion of elementary embeddings. In passing, let us note that the fact that these examples encompass those studied in continuous model theory may well make accessible categories an interesting framework to reunify continuous with discrete model theory (see e.g.\ \cite{lrcaec-jsl}).

Before going further, however, let's address a question a logically-inclined reader may have: what about the (first-order) compactness theorem? Well, you cannot make an omelet without breaking eggs: in view of Lindström's theorem \cite[2.5.4]{chang-keisler}, going significantly beyond $\Elem (T)$ will imply losing the compactness theorem. In fact, the situation is worse than this: when passing to a category we ``forget'' a lot of the logical structures on the objects. For example, we cannot really study definable sets anymore: starting with a category of models with elementary embeddings, we can form a new category by ``Morleyizing'' -- adding a relations for each formula. The Morleyized category is (even concretely) isomorphic to the old one, but model theorists studying quantifier elimination would not want to identify the two categories\footnote{In fact I have seen categorical model theory described as ``model theory without logic''.}...

Why in the world, then, would one want to forget the logical structure? We have already argued that many mathematical categories of interest cannot be studied model-theoretically. Even in model theory, it can be helpful to distinguish between \emph{internal} properties (visible at the level of a single structure, such as quantifier elimination) and \emph{external} properties (visible by looking at the structure of the category). More precisely, let us define an external property to be a property that is invariant under equivalence of categories. One example of an interesting external property is the existence of a universal object in a given cardinality. Such properties show up in Shelah-style model theory \cite{shelahfobook}. In fact, Shelah has observed \cite[p.~23]{shelahaecbook} that what he calls \emph{dividing lines} (e.g.\ model-theoretic stability, simplicity, NIP, etc.) can be characterized by both internal and external properties\footnote{Shelah does not give a precise meaning to ``external'': the interpretation as ``invariant under equivalence of category'' is my own suggestion.}. Stability, for example, is equivalent to failure of the order property (internal), or to the existence of saturated models in certain cardinals (external). Thus, while dividing lines do have a logical characterization, they are also invariant under equivalence of categories. This is interesting insight suggests it may be possible to still do Shelah-style model theory categorically (and this is indeed the case, as exemplified in the large body of work on classification theory for AECs, see the references in Section \ref{reading-sec}).

At a broader level, working with accessible categories entails a higher level of generality. This has downsides but also benefits: more categories of interest are accessible, and accessible categories are closed under more operations. This can make the theory easier to develop in some cases. For example, starting from the class of models of a first-order $T$, one can form its category of $\lambda$-saturated models (for some $\lambda$). One can also form its class of models omitting some type. One could even just look at the class of models of $T$ of cardinality at least $\lambda$. \emph{None} of these examples are classes of models of a first-order but they are still accessible categories. As another example, there is a very general definition of stability, using forking independence, that can be given in any accessible category and makes no mention of logic. This definition specializes to (and is arguably much simpler than) the usual first-order one. See Section \ref{indep-sec}.

Coming back to the failure of the compactness theorem, another counterpoint is that there are many different types and levels of compactness\footnote{In fact, the condition in the definition of an accessible category that every object is a directed colimit of a fixed set of small subobjects can itself be thought of as a weak type of compactness.}. Some compactness can be recovered using large cardinals (e.g.\ compactness for $\Ll_{\kappa, \kappa}$), other types of compactness are implied by the ``complexity'' of the class or category under consideration (for example universal classes, see Definition \ref{univ-def}, do satisfy a weak version of the compactness theorem \cite[3.8]{abv-categ-multi-apal}; classes with low descriptive set-theoretic complexity also behave better than the general case \cite{almost-galois-stable}). Still other types of compactness are implied by the stability-theoretic properties of the class (e.g.\ in a first-order stable theory, any long-enough sequences contains an indiscernible subsequence, c.f.\ \cite[I.2.8]{shelahfobook} or Theorem \ref{indisc-extraction}, but this is not the case in general without large cardinals \cite[18.18]{jechbook}). This interplay between set-theoretic, model-theoretic, and stability-theoretic compactness is a fascinating aspect of categorical model theory, which is harder to see and study in setups where the compactness theorem applies.

At this point a still skeptical reader may say that, while all these philosophical points are interesting, the setup of accessible categories seems too general or difficult for an interesting theory. There are several answers to this objection. First, category theory itself is very general, but it still is an interesting framework in which to present and understand many different branches of mathematics. There \emph{are} interesting theorems in category theory (see the epilogue of \cite{ct-context}), but of course they often inform and supplement rather than completely supersede results in more specialized branches. This does not make the importance of category theory in doubt. Similarly, I think of accessible category as a framework rather than as an all-encompassing object of study. There are less general frameworks (abstract elementary classes, universal classes, first-order model theory, universal algebra, $\Ll_{\omega_1, \omega}$, locally presentable categories, Grothendieck abelian categories, etc.) in which one will be able to say more, but may also sometimes be more limited by the lower generality. The ``best'' framework depends on the question; finding it is a hard part of the mathematician's job.

Second, there \emph{are} nontrivial theorems about accessible categories. Let's just mention a basic one for now: \emph{any} accessible category is equivalent to the category of models of an $\Ll_{\infty, \infty}$ sentence, with morphisms the homomorphisms (Corollary \ref{acc-to-log}). For example, we have already seen that we can Morleyize the models of a first-order theory to obtain a category where the elementary embeddings are simply the injective homomorphisms. Since injective homomorphisms are simply homomorphisms that preserve the non-equality relation, we can add this non-equality relation to the models to get an instance of the theorem. This result gives a surprising correspondence between categorical and ``logical'' model theory (taken in the broad sense of studying classes of models of $\Ll_{\infty, \infty}$). Thus studying categories of models of $\Ll_{\infty, \infty}$ theories is just as hard as studying accessible categories generally. The correspondence also shows that, for the purpose of categorical model theory, the logics with generalized quantifiers described for example in \cite{model-theoretic-logics} are not necessary: if we study the category of models of a certain sentence expressed in a complicated logic with quantifiers such as ``there exists uncountably many'', together with a certain logical notion of morphism, then as long as this category is an accessible category, we can find an equivalent category that will be axiomatized simply in $\Ll_{\infty, \infty}$. See Example \ref{acc-to-log-ex}(\ref{toy-quasi-ex}), the ``toy quasiminimal class'', for an instance of this phenomenon. 

An intermediate type of setup is abstract model theory, which studies frameworks such as abstract elementary classes (AECs) \cite{sh88}. There we are still studying ``Tarskian'' classes of models, but the notion of embedding is axiomatically specified rather than defined to be elementarity. Any AEC can naturally be seen as an accessible category, and any class of models of an $\Ll_{\infty, \infty}$ theory can naturally be seen as an $\infty$-AEC (Definition \ref{infty-aec-def}). Thus there is a three way correspondence\footnote{Since any $\infty$-AEC has all morphisms monomorphisms, the precise correspondence should say that accessible categories with all morphisms monos are equivalent to $\infty$-AECs.} between categorical model theory (accessible categories), abstract model theory ($\infty$-abstract elementary classes), and logical model theory ($\Ll_{\infty, \infty}$).

Such a correspondence should be compared with Tarski's presentation theorem \cite{tarski-th-models-i} (see Theorem \ref{tarski} here): classes closed under substructure, isomorphisms, and unions of chains (abstract model theory) are the same as classes of models of universal theories (logical model theory). Another example of this phenomenon is given by the Birkhoff variety theorem \cite[10]{birkhoff-variety}: a class of algebras is a variety (logical model theory) if and only it is closed under products, subobjects, and quotients (categorical model theory), see \cite[3.9]{adamek-rosicky}. In fact it turns out that many types of accessible categories have natural logical and abstract classes characterizations \cite{adamek-rosicky, multipres-pams}. This is useful, because one can often take advantage of the concreteness of abstract model theory to use certain set-theoretic ideas. ``Element by element'' constructions in abstract elementary classes are a case in point (Section \ref{aec-sec}).

Third, accessible categories are already well connected to the rest of mathematics. While Makkai and Paré had a model-theoretic motivation, one of the first use of accessible categories was in algebraic topology: a \emph{model category} is a category endowed with three distinguished classes of morphisms, called fibrations, cofibrations and weak equivalences, satisfying axioms that are properties of the category of topological spaces (with the weak equivalences being the weak homotopy equivalences). In particular, weak equivalences should satisfy a ``two out of three'' property, and the morphisms should form certain weak factorization systems (e.g.\ any morphism should factor, in a somewhat canonical way, as a cofibration followed by a fibration that is also a weak equivalence). The reader can think of fibrations as ``nice surjections'' and cofibrations as ``nice inclusions''. See \cite{hess-model-survey} for a survey of the use of model categories in algebraic topology, and e.g.\ \cite{hoveybook} for the general theory. It turns out that the most convenient model categories to work with are the locally presentable ones (i.e.\ those that are bicomplete and accessible). This is because, roughly speaking, in such categories we can often find set-sized families of ``small'' cofibrations that generate the rest. Such model categories are called \emph{cofibrantly generated}. For example, the category of topological spaces is not accessible (Example \ref{pres-ex}(\ref{pres-ex-top})), but the category of \emph{simplicial sets} is. The model category induced on it will, it turns out, be cofibrantly generated, and in some homotopical sense equivalent to the one on topological spaces. An application of the theory of accessible categories to model categories is the existence, assuming Vopěnka's principle, of certain localizations \cite{localization-vopenka, left-det-rt}. See also \cite{beke-sheafifiable} for more on the connections between accessible categories, logic, and model categories.

Accessible categories have also been used in homological algebra. An only recently solved problem, the \emph{flat cover conjecture}, asks whether every module has a flat cover \cite{enochs-flat-cover-orig}. This was proven by Bican, El Bashir, and Enochs \cite{flat-cover}. Two proofs were given, one using set-theoretic tools of Eklof-Trlifaj \cite{ext-vanish}. This proof was recognized by Rosický \cite{flat-covers-factorizations} to really be an instance of a ``small object argument'': a way to build cofibrantly generated weak factorization systems in any locally presentable category.

I believe these connections are interesting, especially because logic and model theory have not historically threaded too much into algebraic topology and homological algebra. If accessible categories are really tied to model theory, then these results should have also model-theoretic meaning. Recently, joint work with Lieberman and Rosický \cite{more-indep-v2} showed that in fact the notion of a model category (or more generally a weak factorization system) being cofibrantly generated is closely tied to the theory of model-theoretic forking. Essentially, in a locally presentable model category, a weak factorization system is cofibrantly generated exactly when the category obtained by restricting to only to cofibrations is stable, in the sense of having a forking-like independence notion. See Section \ref{indep-sec} for an overview of this result. 

In conclusion, categorical model theory is a fascinating mix of set theory, category theory, and model theory that sheds light on all these topics and on some other parts of mathematics. This paper is meant to survey some basic results, as well as discuss some topics of current research. I will survey set-theoretic aspects (Section \ref{set-sec}), as well as questions on stable independence (Section \ref{indep-sec}), a recent research development that may be of interest to model theorists, category theorists, and mathematicians at large. Some methods applicable to concrete setups such as AECs, Section \ref{aec-sec} will also be discussed. I will end with some open problems (Section \ref{problem-sec}), as well as a short list of helpful resources to learn more (Section \ref{reading-sec}). A first appendix gives a streamlined method for handling the bookkeeping in point by point constructions, and the second appendix gives a very short introduction to first-order stability theory.  

I will assume that the reader is familiar with very basic logic and model theory \cite{chang-keisler}, set theory \cite{hrbacek-jech}, and category theory \cite{maclane, joy-of-cats}. Still, an effort has been made to repeat many standard definitions and provide examples and intuitions behind proofs for those that are not necessarily proficient in all three fields.

\subsection*{Acknowledgments}

I would like to thank Jiří Rosický and Michael Lieberman for introducing me to accessible categories and taking an early look at this survey. The work on category-theoretic sizes and stable independence presented here comes from our joint collaboration.

I thank Marcos Mazari-Armida and the referee for helpful comments. I also thank Justin Cavitt, Rebecca Coulson, and Rehana Patel, for encouraging me to write Appendix \ref{fo-sec}.

\section{Main definitions and examples}

\subsection{Set theory}

We assume familiarity with ordinals and cardinals. We think of ordinals as transitive sets ordered by membership, and identify cardinals with the corresponding ordinals. We denote by $\omega$ the first infinite ordinal, i.e.\ the set of natural numbers. For a cardinal $\lambda$, we write $\lambda^+$ for the successor of $\lambda$, the minimal cardinal strictly bigger than $\lambda$. Cardinals of the form $\lambda^+$ are called \emph{successor cardinals}, and cardinals that are not successors are called \emph{limit}. We write $\fct{Y}{X}$ for the set of functions from $Y$ to $X$ and, for an ordinal $\alpha$, $\fct{<\alpha}{X}$ for $\bigcup_{\beta < \alpha} \fct{\beta}{X}$. For $\lambda$ and $\mu$ cardinal, $\lambda^{<\mu} = |\fct{<\mu}{\lambda}|$ and $\lambda^\mu = |\fct{\mu}{\lambda}|$. A \emph{partially ordered set} (or poset for short) is a transitive, reflexive, and antisymmetric relation. A subset $I_0$ of a poset $I$ is \emph{cofinal} if for all $i \in I$ there is $i_0 \in I_0$ so that $i \le i_0$. The \emph{cofinality}, $\cf{I}$, of $I$ is the minimal cardinality of a cofinal subset of $I$. Of course, we identify an ordinal $\alpha$ with the corresponding linear order $(\alpha, \in)$. A cardinal $\lambda$ is \emph{regular} if it is infinite and $\cf{\lambda} = \lambda$.

\subsection{Category theory}

A category is called \emph{small} if it has only a set of objects, and it is called \emph{large} if it has a proper class of \emph{non-isomorphic} objects. All the categories in this paper will be locally small (i.e.\ the collection $\Hom (A, B)$ of morphisms from the object $A$ to the object $B$ is always a set, never a proper class). Our categorical notation and conventions will mostly be those of \cite{joy-of-cats} and \cite{adamek-rosicky}. In particular, a \emph{monomorphism} (or \emph{mono} for short) is a morphism $f$ such that $f g_1 = fg_2$ implies $g_1 = g_2$ for any two morphisms $g_1$ and $g_2$. The notion of an \emph{epimorphism} (\emph{epi} for short) is defined dually. A \emph{diagram} in a category $\ck$ is a functor $D: I \to \ck$, where $I$ is a (here always small) category, called the \emph{index} of the diagram. In this paper, the index $I$ will usually be a partially ordered set, which will be identified with the corresponding category. A \emph{cocone}\footnote{In this paper, we will be more interested in colimits than limits, so we also define cocones rather than cones.} for a diagram $D: I \to \ck$ consist in an object $A$ together with morphism $(D_i \xrightarrow{f_i} A)_{i \in I}$ such that whenever $i \xrightarrow{d} j$ is a morphism of $D$, we have that $f_i = f_j d$. The cocones for $D$ form a category, $\ck_D$, where a morphism from $(D_i \xrightarrow{f_i} A)_{i \in I}$ to $(D_i \xrightarrow{g_i} B)_{i \in I}$ is a $\ck$-morphism $A \xrightarrow{h} B$ so that $h f_i = g_i$ for all $i \in I$. A \emph{colimit} for $D$ is an initial object in the category $\ck_D$ (recall that an initial object in a category is one so that there is a unique morphism from it to any other object). 

\subsection{Presentability and accessible categories}

For $\lambda$ a cardinal, a partially ordered set is called \emph{$\lambda$-directed} if every subset of cardinality strictly less than $\lambda$ has an upper bound. Note that any non-empty poset is $\lambda$-directed for $\lambda \le 2$ (an empty poset is $0$-directed but not $1$-directed), $3$-directed is equivalent to $\aleph_0$-directed, and for a singular cardinal $\lambda$, $\lambda$-directed is equivalent to $\lambda^+$-directed. Thus we usually assume that $\lambda$ is a regular cardinal. We will say \emph{directed} instead of $\aleph_0$-directed. For $\lambda$ regular, an example of a $\lambda$-directed poset that the reader can keep in mind is $\lambda$ itself, seen as a chain (i.e.\ a linear order) of order type $\lambda$. Another typical example is the poset $[A]^{<\lambda}$ of all subsets of a fixed set $A$ which have cardinality strictly less than $\lambda$, ordered by containment.

A \emph{$\lambda$-directed diagram} is a diagram indexed by a $\lambda$-directed poset. A category has \emph{$\lambda$-directed colimits} if any $\lambda$-directed diagram has a colimit. Note that for $\lambda_1 < \lambda_2$, $\lambda_2$-directed implies $\lambda_1$-directed, hence having $\lambda_1$-directed colimits is \emph{stronger} than having $\lambda_2$-directed colimits. The following observation will be used without comments: 

\begin{lem}[Iwamura's lemma {\cite[1.7]{adamek-rosicky}}]
  If a category has colimits of all chains indexed by regular cardinals, then it has directed colimits.
\end{lem}
\begin{proof}[Proof sketch]
  It is immediate using the definition of cofinality that a category with all colimits of chains indexed by regular cardinals has colimits of all chains indexed by ordinals. Now any finite directed poset has a maximum, so colimits of such diagrams are trivial. Furthermore, any infinite directed partially ordered set $I$ can be written as $\bigcup_{\alpha < |I|} I_\alpha$, where each $I_\alpha$ has strictly smaller cardinality than $I$, is directed, and $I_\alpha \subseteq I_\beta$ for $\alpha < \beta$. We can therefore proceed by induction on $|I|$ to show that the category has all colimits indexed by $I$. 
\end{proof}
\begin{remark}
  A similar result no longer holds in the uncountable case \cite[Exercise 1.c(2)]{adamek-rosicky}: if $\lambda$ is uncountable, having colimits indexed by ordinals of cofinality at least $\lambda$ is \emph{strictly weaker} than having all $\lambda$-directed colimits.
\end{remark}

\begin{example}
  In many concrete algebraic cases, directed colimits exist and are essentially computed by taking unions. For example, the category $\Ab$ of abelian groups has directed colimits, and they can essentially be computed by taking unions of the resulting system of groups. However in the category of Banach spaces with contractions, colimits of increasing chains are given by taking the \emph{completion} of their union. Of course, some categories don't have directed colimits at all. For example, the category of well-orderings, with morphisms the order-preserving maps, does not have directed colimits. It does however have $\aleph_1$-directed colimits. The category of complete Boolean algebras (with homomorphisms), on the other hand, does not have $\lambda$-directed colimits for any $\lambda$.
\end{example}

The notion of a $\lambda$-presentable object is key to the definition of an accessible category.

\begin{defin}\label{pres-def}
  For $\lambda$ a regular cardinal, an object $A$ of a category $\ck$ is \emph{$\lambda$-presentable} if the functor $\Hom (A, -)$ preserves $\lambda$-directed colimits. Said more transparently, $A$ is \emph{$\lambda$-presentable} if for any $\lambda$-directed diagram $D: I \to \ck$ with colimit $(D_i \xrightarrow{g_i} \colim D)_{i \in I}$, any map $A \xrightarrow{f} \colim D$ factors essentially uniquely through $D$: there exists $i \in I$ and $A \xrightarrow{f_i} D_i$ such that $f_i g_i = f$. Essentially uniqueness means that, moreover, if $j \in I$ and $A \xrightarrow{f_j} D_j$ are such that $f_j g_j = f$, then there exists $k \ge i,j$ such that $g_{i,k} f_i = g_{j, k} f_j$ (here, $g_{i, j}$ and $g_{i, k}$ are the diagram maps from $i$ to $j$ and $i$ to $k$ respectively).

  $$  
    \xymatrix@=3pc{
      A \ar[r]^f \ar@{.>}[dr]_{f_i} & \colim D \\
      & D_i \ar[u]_{g_i} \\
    }
    $$

    When $\lambda = \aleph_0$, we will say that $A$ is \emph{finitely presentable}. We say that $A$ is \emph{presentable} if it is $\lambda$-presentable for some $\lambda$. The \emph{presentability rank} of a presentable object $A$ is the least regular cardinal $\lambda$ such that $A$ is $\lambda$-presentable.
\end{defin}

Note that a $\lambda$-presentable object is also $\mu$-presentable for all regular $\mu > \lambda$. 

\begin{example}\label{pres-ex} \
  \begin{enumerate}
  \item\label{pres-ex-set} In the category of sets (with morphisms all functions), a set is $\lambda$-presentable if and only if it has cardinality strictly less than $\lambda$. Indeed, directed colimits are essentially unions, and a set $A$ has cardinality strictly less than $\lambda$ if and only if whenever $A \subseteq \bigcup_{i < \lambda} A_i$, there exists $i < \lambda$  such that $A \subseteq A_i$. Thus the presentability rank of an infinite set is the successor of its cardinality. Similarly, in many algebraic categories (in fact in any abstract elementary classes -- see Definition \ref{infty-aec-def} and Example \ref{unif-ex}(\ref{unif-ex-aec})), the presentability rank of a big-enough object will be the successor of the cardinality of its universe.
    
  \item In very algebraic categories, such as $\Ab$, an object is $\lambda$-presentable if and only if it can be presented using strictly less than $\lambda$-many objects and equations \cite[3.12]{adamek-rosicky}. This is the motivation for the term ``presentable''.
  \item\label{pres-ex-wo} \cite[6.2]{internal-sizes-jpaa} Consider the category of well-orderings, with morphisms the initial segment embeddings. Note that (as opposed to what would happen if the maps were just order-preserving maps) this category has all directed colimits. An infinite well-order is $\lambda$-presentable precisely when it has cofinality strictly less than $\lambda$.    
  \item Consider the category of all algebras with one $\omega$-ary operation, and morphisms the homomorphisms. Then an object is $\lambda$-presentable if and only if it is generated by a set of cardinality strictly less than $\lambda$. Here, the free algebra on $\lambda$-many generators is $\lambda^+$-presentable but has cardinality $\lambda^{\aleph_0}$ (which could be strictly bigger than $\lambda$).    
  \item\label{pres-ex-hilb} Consider the category of complete metric spaces, with isometries as morphisms. For $\lambda$ regular \emph{uncountable}, a space $A$ is $\lambda$-presentable if and only if its it has \emph{density character} strictly less than $\lambda$ (the density character of a topological space is defined as the least cardinality of a dense subset). Thus the presentability rank is the successor of the density character. This holds more generally in all the ``continuous'' examples, including Banach and Hilbert spaces \cite[3.1]{lrcaec-jsl}. Note that in general the density character is different from the cardinality. For example, there are no Hilbert spaces of cardinality $\lambda$ when $\lambda < \lambda^{\aleph_0}$, but there are no problems finding a Hilbert space generated by an orthonormal basis of cardinality $\lambda$. Finite sums from such a basis will give the desired dense subset.
  \item\label{pres-ex-met} \cite[6.1]{internal-sizes-jpaa}\footnote{This example was pointed out to me by Ivan Di Liberti.} Consider the one point metric space $\{0\}$ in the category of complete metric spaces with isometries. This space is $\aleph_1$-presentable (by the preceding discussion) but it is \emph{not} $\aleph_0$-presentable. Indeed, the inclusion of $\{0\}$ into $\{0\} \cup \{\frac{1}{n} \mid 0 < n < \omega\}$ does not factor through $\{\frac{1}{n} \mid 0 < n \le m\}$, for any $m < \omega$. In fact, the empty metric space is the only $\aleph_0$-presentable object of the category.

  \item\label{pres-ex-top} \cite[1.2(10)]{adamek-rosicky} In the category of topological spaces (with morphisms the continuous functions), a discrete spaces (i.e.\ where every set is open) is $\lambda$-presentable exactly when it has cardinality strictly less than $\lambda$. On the other hand, a non-discrete space $A$ is \emph{never} presentable. To see this, fix a regular $\lambda$ and a non-open set $X \subseteq A$. Without loss of generality, $A \cap \lambda = \emptyset$. For $\alpha <\lambda$, let $D_\alpha$ be the space $A \cup \lambda$, where the non-trivial open sets in $D_\alpha$ are of the form $X \cup [\beta, \lambda)$ for $\beta \in [\alpha, \lambda)$. For $\alpha \le \beta$, the inclusion of $D_\alpha$ into $D_\beta$ is continuous ($D_\alpha$ has more open sets than $D_\beta$), and these inclusions give a $\lambda$-directed diagram. The colimit $D$ of this diagram is the indiscrete space on the set $A \cup \lambda$. Because $D$ is indiscrete, the inclusion of $A$ into $D$ is continuous, but it cannot factor through any $D_\alpha$ (because the inverse image of $X \cup [\alpha, \lambda)$, an open set of $D_\alpha$, is $X$, which is not open in $A$).
  \end{enumerate}

\end{example}

In view of the many examples above where the presentability rank was the \emph{successor} of some natural notion of size (see Section \ref{set-size-sec} on what is true in general), it is convenient to have a name for this predecessor. We will call it the internal, or category-theoretic, \emph{size}:

\begin{defin}\label{size-def}
  The \emph{(internal -- or category-theoretic) size} of an object $A$ is the predecessor (if it exists) of its presentability rank. 
\end{defin}

We can now precisely say what an accessible category is:

\begin{defin}\label{acc-def}
  For a regular cardinal $\lambda$, a category $\ck$ is \emph{$\lambda$-accessible} if:

  \begin{enumerate}
  \item $\ck$ has $\lambda$-directed colimits.
  \item (Smallness condition) There is a set $S$ of $\lambda$-presentable objects such that every object of $\ck$ is a $\lambda$-directed colimits of elements of $S$.
  \end{enumerate}

  If $\ck$ is $\lambda$-accessible and bicomplete (i.e.\ complete and cocomplete: it has all limits and colimits), we say that it is \emph{locally $\lambda$-presentable}. When $\lambda = \aleph_0$, we will talk about a finitely accessible or locally finitely presentable category. We say that $\ck$ is \emph{accessible} [\emph{locally presentable}] if it is $\lambda$-accessible [locally $\lambda$-presentable] for some regular cardinal $\lambda$. 
\end{defin}
\begin{remark}
  There are three occurrences of the parameter $\lambda$ in Definition \ref{acc-def}, and there are no real reasons why these occurrences should all be the same cardinal. Thus we could parameterize the definition further into three cardinals, see \cite[3.6]{internal-sizes-jpaa}. In fact, we can be even more precise and allow singular cardinals in the definition. For example, an $(\aleph_0, \aleph_1, <\aleph_{\omega_1})$-accessible category would be a category with directed colimits where every object is an $\aleph_1$-directed colimits of a fixed set of $(<\aleph_{\omega_1})$-presentable objects, where $(<\aleph_{\omega_1})$-presentable means $\lambda_0$-presentable for some regular $\lambda_0 < \aleph_{\omega_1}$. Thus we want each object in the diagram to have presentability rank some $\aleph_\alpha$, $\alpha < \omega_1$, but there may not be a single $\aleph_\alpha$ that bounds the rank of each object. For simplicity, we will not take this approach here but will often discuss, for example, accessible categories with all directed colimits.
\end{remark}

\begin{example}\label{acc-ex} \
  \begin{enumerate}
  \item The category $\Set$ of sets with functions as morphisms is locally finitely presentable: it is bicomplete and every set is a directed colimit of finite sets.
  \item Any complete lattice is (when seen as a category) a locally presentable category. 
  \item The category $\Set_{mono}$ of sets with injective functions as morphisms, is finitely accessible but not bicomplete (for example, it does not have coequalizers; intuitively it is impossible to quotient). There are two general phenomenons at play here: first \cite[6.2]{indep-categ-advances}, given any accessible category $\ck$, the category $\ck_{mono}$ obtained by restricting to its monomorphisms will be accessible as well. Second, if $\ck$ is an accessible category where all morphisms are monomorphisms (this is typically the case in model-theoretic setups), then $\ck$ cannot be locally presentable, unless it is a complete lattice (see the previous example), so in particular small.
  \item The category $\Ab$ of a abelian groups is locally finitely presentable. The finitely generated abelian groups are the finitely presentable objects and any other abelian group is a directed colimit of those. More generally, for a fixed ring $R$, the category $\Rmod$ of $R$-modules is locally finitely presentable.
  \item The category $\Gra$ of graphs (symmetric reflexive binary relations) with morphisms the graph homomorphisms is locally finitely presentable (finitely presentable objects are finite graphs).
  \item The category $\Ban$ of Banach spaces (with morphisms the contractions) is locally $\aleph_1$-presentable, but not locally $\aleph_0$-presentable (for reasons connected to Example \ref{pres-ex}(\ref{pres-ex-met})).
  \item The category $\Bool$ of boolean algebras (with morphisms the boolean algebra homomorphisms) is locally finitely presentable. On the other hand, the category of \emph{complete} boolean algebras is not even accessible: it does not have $\lambda$-directed colimits for any $\lambda$. For a fixed regular cardinal $\lambda$, the category of $\lambda$-complete Boolean algebras is however $\lambda$-accessible.
  \item The category of well-orderings with morphisms the order-preserving maps is $\aleph_1$-accessible but not finitely accessible (it does not have directed colimits). On the other hand the category of well-orderings with morphisms the initial segment maps is \emph{not} accessible, even though it has directed colimits: for any regular cardinal $\lambda$, there is a proper class of non-isomorphic well-orderings of cofinality $\lambda$, see Example \ref{pres-ex}(\ref{pres-ex-wo}). We could still look at the full subcategory of well-orderings of order type $\lambda^+$ or less. It turns out that this will be a $\lambda^+$-accessible category.
  \item The category of all well-founded models of (a sufficiently-big fragment of) ZFC, or of all well-founded models of ZFC + V = L, with morphisms the elementary embeddings, is $\aleph_1$-accessible. A variation of this example is studied in \cite[\S 6.1]{internal-sizes-jpaa}.
  \item Any Fraïssé class can be seen as generating a finitely accessible category. In fact, there exists a general categorical theory of Fraïssé constructions \cite{fraisse-kubis}.
  \item\label{free-ex} The category of all free abelian groups (with group homomorphisms) is not accessible if $V = L$, but is $\kappa$-accessible for $\kappa$ a strongly compact cardinal \cite[\S5.5]{makkai-pare}. Thus whether a category is accessible can, in certain cases, be a set-theoretic question. See Section \ref{set-func-sec} for more on this phenomenon.
  \end{enumerate}
\end{example}

\subsection{Categories of structures}\label{struct-sec}
In order to get more examples of accessible categories and introduce related frameworks, let us move toward logic. For completeness and because we will work with infinitary languages, we start by repeating the basic definitions. The reader will not lose much by skipping them. More details can be found in \cite{dickmann-book} or \cite[Chapter 5]{adamek-rosicky}. For $\kappa$ an infinite cardinal, a \emph{$\kappa$-ary vocabulary} (or signature) is a set $\tau$ containing\footnote{We will work with a single sort for simplicity.} relations and functions symbols of arity strictly less than $\kappa$. When $\kappa = \aleph_0$, we call $\tau$ a finitary vocabulary, and when $\kappa \ge \aleph_1$ an infinitary vocabulary. By default, a \emph{vocabulary} means a $\kappa$-ary vocabulary for some $\kappa$ (so possibly infinitary). More precisely, a $\kappa$-ary vocabulary $\tau$ is a a set $\seq{n_i : i \in I_R, m_j: j \in I_F}$, where $I_R$ and $I_F$ are disjoint index sets and $n_i, m_j$ are cardinals strictly less than $\kappa$. For such a vocabulary, a \emph{$\tau$-structure} $M$ consists of a set $A = U M$ (the \emph{universe}), for each $n_i \in \tau$, an $n_i$-ary relation $R_i$ on $A$, and for each $m_j \in \tau$ an $m_j$-ary function $f_j$ on $A$. A \emph{term} in the vocabulary $\tau$, over a given set of variables $V$ (disjoint from any other object that we consider, and of cardinality $\kappa$ -- we will never mention $V$ again), is defined inductively as either a variable $x$, or as $f_j (\rho_\alpha)_{\alpha < m_j}$, for each $\rho_\alpha$ a term. An \emph{atomic formula} is an expression of the form $\top$, $\bot$, $\rho = \rho'$ or $R_i (\rho_\alpha)_{\alpha < n_i}$ for terms $\rho$ ,$\rho'$, $\rho_\alpha$. For $\lambda \ge \kappa$ an infinite cardinal, an \emph{$\Ll_{\lambda, \kappa}$ formula} is defined inductively as either an atomic formula, $\bigwedge_{\alpha \in S} \phi_\alpha$, $\bigvee_{\alpha \in S} \phi_\alpha$, $\phi \rightarrow \psi$, $\neg \phi$, $(\forall \bx) \phi$, $(\exists \bx) \phi$, where $\phi_\alpha$, $\phi$, $\psi$, are formulas, $|S| < \lambda$, $\bx$ is a sequence of variables of length strictly less than $\kappa$, and we require that the formulas obtained from conjunctions and disjunctions still have fewer than $\kappa$-many variables. We also define $\Ll_{\infty, \kappa} = \bigcup_{\lambda} \Ll_{\lambda, \kappa}$ and $\Ll_{\infty, \infty} = \bigcup_{\kappa} \Ll_{\infty, \kappa}$ as expected. The \emph{free variables} of a formula are the ones that appear in it and are not bound by any quantifier. We write $\phi (\bx)$ for a formula $\phi$ with free variables among $\bx$. A \emph{sentence} is a formula without free variables, and a \emph{theory} is a set of sentences. For a $\tau$-structure $M$, $\ba$ a sequence of elements from $M$, and $\phi (\bx)$ a formula ($\bx$ and $\ba$ of the same length), we define what it means for $M$ to satisfy (or be a model of) $\phi$, with $\ba$ standing for $\bx$, $M \models \phi (\ba)$ for short, as expected. A $\tau$-structure satisfies (or is a model of) a theory if it satisfies all sentences of the theory.

For a vocabulary $\tau$, a \emph{homomorphism} from a $\tau$-structure $M$ to a $\tau$-structure $N$ is a function $f$ from $U M$ to $U N$ that preserves all atomic formulas: if $\phi (\bx)$ is atomic, $\ba$ is a sequence in $M$, and $M \models \phi (\ba)$, then $N \models \phi (f (\ba))$. We let $\Str (\tau)$ denote the category of all $\tau$-structures with homomorphisms. For $T$ a theory in $\Ll_{\infty, \infty}$, we let $\Mod (T)$ denote the full subcategory of $\Str (\tau)$ consisting of all models of $T$. When $\phi$ is a sentence, $\Mod (\phi)$ will denote $\Mod (\{\phi\})$. Note that $\Mod (\phi)$ is \emph{not} always an accessible category, essentially because the homomorphisms are not the right notion of embedding when $\phi$ is too complex. Thus we more generally define, for $\Phi \subseteq \Ll_{\infty, \infty}$, a \emph{$\Phi$-elementary map from $M$ to $N$} to be a function $f$ from $U M$ to $U N$ that preserves all formulas in $\Phi$. We let $\Elem_\Phi (\tau)$ denote the category of all $\tau$-structures with homomorphism, and $\Elem_\Phi (T)$ denote the full subcategory of $\Elem_\Phi (\tau)$ consisting of models of $T$. One can check that for $\lambda$ a regular cardinal and $T$ an $\Ll_{\lambda, \lambda}$ theory, $\Elem_{\Ll_{\lambda, \lambda}} (T)$ is $\lambda$-accessible \cite[\S 3.4]{makkai-pare}.

We will focus on the following simple type of formulas:

\begin{defin}
  A formula of $\Ll_{\infty, \infty}$ is \emph{positive existential} if it can be built from atomic formulas using only conjunctions, disjunctions and existential quantifications. A formula is \emph{basic} if it is of the form $(\forall \bx)(\phi \rightarrow \psi)$, where $\phi$ and $\psi$ are positive existential. A \emph{basic} theory is a set of basic sentences.
\end{defin}

Note that any theory in $\Ll_{\infty, \infty}$ can be ``Morleyized'' to a basic theory, by adding a relation symbol for each formula \cite[3.2.8]{makkai-pare}. In fact, for $\kappa \le \lambda$ and $T$ a theory in $\Ll_{\lambda, \kappa}$, $\Elem_{\Ll_{\lambda, \kappa}} (T)$ is isomorphic to $\Mod (T')$, for some basic $\Ll_{\lambda, \kappa}$-theory $T'$. This fact is sometimes called \emph{Chang's presentation theorem} by model theorists and is the reason why we can restrict ourselves to basic formulas without losing generality.

\begin{example}\label{logical-ex} Let $\kappa$ and $\lambda$ be regular cardinals.
  \begin{enumerate}
  \item \cite[\S 5.1]{adamek-rosicky} If $\tau$ is a $\kappa$-ary vocabulary then $\Str (\tau)$ is a locally $\kappa$-presentable category.
  \item \cite[5.35]{adamek-rosicky} For any basic theory $T$ in $\Ll_{\infty, \lambda}$, $\Mod (T)$ is accessible with $\lambda$-directed colimits.
  \end{enumerate}

  In the second result, existence of $\lambda$-directed colimits can be proven directly. The smallness condition in the definition of an accessible category follows, for example, from the infinitary downward Löwenheim-Skolem theorem. 
\end{example}

It turns out that Example \ref{logical-ex} is sharp, in the sense that any accessible category is equivalent to a category of the form $\Mod (T)$, for $T$ a basic theory (Corollary \ref{acc-to-log}). Similarly, it is possible to characterize the locally presentable categories as the categories of models of certain theories (called limit). See \cite[5.30]{adamek-rosicky}.

\subsection{$\infty$-abstract elementary classes}

We now introduce $\infty$-abstract elementary classes ($\infty$-AECs). They are in some sense a compromise: less abstract than accessible categories, but still more abstract than categories of models of basic theories. The definition of an AEC (the case $\mu = \aleph_0$ below) is due to Shelah \cite{sh88}. It was generalized to the case of a $\mu$-AEC in \cite{mu-aec-jpaa}. We introduce essentially the same notion, but to be precise we make $\mu$ (and the vocabulary) part of the data. The reader should first read it with the case $\mu = \aleph_0$ in mind. 

\begin{defin}\label{infty-aec-def}
  An \emph{$\infty$-abstract elementary class} (or $\infty$-AEC for short) $\K$ consists of:

  \begin{enumerate}
  \item A regular cardinal $\mu = \mu (\K)$.
  \item A $\mu$-ary vocabulary $\tau = \tau (\K)$.
  \item A class $K$ of $\tau$-structures.
  \item A partial order $\lea$ on $K$.
  \end{enumerate}

  satisfying the following four axioms:

  \begin{itemize}
  \item \underline{Abstract class axiom}: $K$ is closed under isomorphisms, $M \lea N$ implies that $M$ is a $\tau$-substructure of $N$ (i.e.\ $U M \subseteq U N$ and the inclusion $M \to N$ is a homomorphism; we write $M \subseteq N$), and $\lea$ respects isomorphisms in the sense that if $M, N \in K$, $M \lea N$, and $f: N \cong N'$, then $f[M] \lea N'$.
  \item \underline{Coherence axiom}: if $M_0, M_1, M_2 \in K$, $U M_0 \subseteq UM_1$, $M_1 \lea M_2$, and $M_0 \lea M_2$, then $M_0 \lea M_1$.
  \item \underline{Tarski-Vaught (TV) chain axiom}: if $\seq{M_i : i \in I}$ is a $\mu$-directed system in $\K$ (i.e.\ $I$ is a $\mu$-directed poset, all the $M_i$'s are in $K$, and $i \le j$ implies $M_i \lea M_j$), then letting $M := \bigcup_{i \in I} M_i$ (defined as expected), we have:

    \begin{itemize}
    \item $M \in K$.
    \item $M_i \lea M$ for all $i \in I$.
    \item If $N \in K$ and $M_i \lea N$ for all $i \in I$, then $M \lea N$.
    \end{itemize}
  \item \underline{Löwenheim-Skolem-Tarski (LST) smallness axiom}: there exists a cardinal $\lambda = \lambda^{<\mu} \ge |\tau| + \mu$ such that for any $M \in K$ and any $A \subseteq U M$, there exists $M_0 \in K$ with $M_0 \lea M$, $A \subseteq U M_0$, and $|U M_0| \le |A|^{<\mu} + \lambda$. We write $\LS (\K)$ (the \emph{Löwenheim-Skolem-Tarski number} of $\K$) for the minimal such $\lambda$.
  \end{itemize}

  Unless $K$ is empty, the vocabulary $\tau (\K)$ can always be recovered from $K$. Thus we usually just write $\K = (K, \lea)$ and say that $\K$ is a $\mu$-AEC to make the $\mu$ associated with it clear. When $\mu = \aleph_0$, we omit it and just say that $\K$ is an AEC. We also will not distinguish between $K$ and $\K$, writing for example $M \in \K$ instead of $M \in K$. Another convention: if we write $M \lea N$, we will automatically mean that also $M, N \in K$.
\end{defin}

In any $\infty$-AEC, there is a natural notion of morphism.

\begin{defin}
  Let $\K$ be an $\infty$-AEC. For $M, N \in \K$, a \emph{$\K$-embedding} from $M$ to $N$ is an injective $\tau (\K)$-homomorphism $f$ from $M$ to $N$ such that $f[M] \lea N$. When $\K$ is clear from context, we write $f: M \to N$ to mean that $f$ is a $\K$-embedding from $M$ to $N$. We will often identify $\K$ with the category whose objects are the structures in $\K$ and morphisms the $\K$-embeddings.
\end{defin}
\begin{remark}\label{aec-categ-prop}
  If $\K$ is a $\mu$-AEC, then it has the following properties as a category:

  \begin{itemize}
  \item It is a subcategory of $\Str (\tau (\K))$. Further, it is \emph{isomorphism-closed} in $\Str (\tau (\K))$ (a subcategory $\cl$ of a category $\ck$ is \emph{isomorphism-closed} -- or \emph{replete} -- if whenever $A$ is an object of $\cl$ and $f: A \to B$ is an isomorphism of $\ck$, then both $f$ and $B$ are in $\cl$). This is the essential content of the abstract class axiom.
  \item It is a \emph{concrete} category, as witnessed by the universe functor $U : \K \to \Set$.
  \item All its morphisms are monomorphisms, and in fact concrete monomorphisms (i.e.\ they are also monomorphisms in the category of sets -- injective functions). More is true: it is noticed in \cite[\S 8]{joy-of-cats} that the ``right'' notion of monomorphism in many examples ends up being the notion of a concrete embedding \cite[8.6]{joy-of-cats} whose definition mirrors the coherence axiom of AECs. In fact, what the coherence axiom says is exactly that the morphisms of $\K$ are concrete embeddings in the sense of \cite[8.6]{joy-of-cats}.
  \item It has $\mu$-directed colimits (this is the essential content of the chain axiom). In fact these $\mu$-directed colimits are \emph{concrete} in the sense that the functor $U$ preserves them: they are computed the same way as in $\Set$, by taking unions.
  \item It is an $\LS (\K)^+$-accessible category. Indeed, the objects of cardinality at most $\LS (\K)$ are all $\LS (\K)^+$-presentable, and there is only a set of them up to isomorphism. Moreover, the coherence and smallness axioms together imply that any other object $M \in \K$ can be written as the $\LS (\K)^+$-directed union of $\{M_0 \in \K \mid M_0 \lea M, |U M_0| \le \LS (\K)\}$ (we think of it as a $\lea$-system indexed by itself). Note however that $\K$ need \emph{not} be $\mu$-accessible. In fact it is easy to see that for any regular cardinal $\lambda$, the ($\aleph_0$-)AEC of all sets of cardinality at least $\lambda$ (ordered with subset) is \emph{not} $\lambda$-accessible. We will see later (Theorem \ref{aec-acc}) that the smallness axiom is, modulo the other axioms, \emph{equivalent} to the smallness condition in the definition of an accessible category.
  \end{itemize}

  In fact, one will not loose much by forgetting about logic and simply thinking of an $\mu$-AEC as a concrete category $(\ck, U)$ with concrete $\mu$-directed colimits and where all morphisms are concrete embeddings.
\end{remark}

Examples of AECs can be found in \cite[\S 3.2]{bv-survey-bfo}, and some more examples of $\infty$-AECs can be found in \cite[\S2]{mu-aec-jpaa}. Some other recently studied examples of interests include the category of flat modules with flat monomorphisms (see Example \ref{indep-ex}(\ref{indep-ex-mod})), a special case of AECs on class of modules of the form $\fct{\perp}{N}$ \cite{bet} and an example of what module theorists call a \emph{Kaplansky classes}, see e.g.\ \cite{kaplansky-finite}. For now, we note the following general fact: 

\begin{example}
  The class of models of a basic $\Ll_{\infty, \lambda}$-theory is a $\lambda$-AEC, when ordered with substructure (see Example \ref{logical-ex}). In fact, given any basic $\Ll_{\infty, \lambda}$-theory $T$, if the vocabulary contains a binary relation $R$ and $T$ contains the two basic sentences

  \begin{itemize}
  \item $(\forall x \forall y)((R (x, y) \land x = y) \rightarrow \bot)$
  \item $(\forall x \forall y)(x = y \lor R (x, y))$
  \end{itemize}

  then all the morphisms of $\Mod (T)$ will be monomorphisms (the two sentences above say that $R$ should be interpreted as the non-equality relation in any model of $T$), and the category $\Mod (T)$ will essentially be a $\lambda$-AEC.
\end{example}

Closing the loop, we will see later that \emph{any} accessible category where all morphisms are monomorphisms is equivalent to an $\infty$-AEC (Theorem \ref{acc-to-aec}). This also follows from the already mentioned equivalence between accessible categories and categories of models of $\Ll_{\infty, \infty}$ sentences.

Two remarks are in order. First, the reader might wonder about the hypothesis that all morphisms are monomorphisms. There is a generalization of the notion of an AEC, called an \emph{abstract elementary category} which removes this restriction \cite[5.3]{beke-rosicky}. In fact, let us define a \emph{$\mu$-abstract elementary category} $\K$ to be a subcategory of $\Str (\tau)$, $\tau = \tau (\K)$ a $\mu$-ary vocabulary, satisfying all the properties listed in Remark \ref{aec-categ-prop}, except that the morphisms are only required to be initial (in the sense of \cite[8.6]{joy-of-cats}), but no longer required to be concrete monomorphisms. We will not discuss $\mu$-abstract elementary categories further after this section.

Second, and more importantly, the correspondences between $\Ll_{\infty, \infty}$, accessible categories, and $\infty$-AECs do not hold cardinal by cardinal: as we have seen in Remark \ref{aec-categ-prop} it is \emph{not} true that any $\lambda$-AEC is $\lambda$-accessible (and it is similarly not true that a category of models of a basic $\Ll_{\infty, \lambda}$ theory is $\lambda$-accessible). It \emph{is} however true that a $\lambda$-accessible category is a $\lambda$-abstract elementary category as well as a category of models of a basic $\Ll_{\infty, \lambda}$ theory. In fact, we have, for a fixed regular $\lambda$, the following hierarchy, where each level has more categories than the next:

\begin{enumerate}
\item $\lambda$-accessible categories.
\item Categories of models of a basic $\Ll_{\infty, \lambda}$ theory.
\item $\lambda$-abstract elementary categories.
\item Accessible categories with $\lambda$-directed colimits.
\item $\infty$-abstract elementary categories = accessible categories = categories of models of a basic $\Ll_{\infty, \infty}$ sentence.
\end{enumerate}

This is relevant, especially because of the big differences between an $\aleph_1$-AEC and an $(\aleph_0$-)AEC: a lot more can be done in the latter setup (see Section \ref{aec-sec}).

\begin{example}\label{ban-ex}
  Consider the category $\Ban_{mono}$ of Banach spaces with isometries. This is an $\aleph_1$-accessible category with all directed colimits and all morphisms monos. However, these directed colimits are \emph{not} concrete (they are not unions, but completions of unions), so $\Ban_{mono}$ does not seem to obviously be an AEC. One can show \cite{hilb-mono-v3} that in fact \emph{no} faithful functor from $\Ban_{mono}$ into $\Set$ preserves directed colimits. Thus $\Ban_{mono}$ is indeed \emph{not} equivalent to an AEC. It will however be an $\aleph_1$-AEC.
\end{example}

\section{Fundamental results}

We start developing the theory of accessible categories from scratch, proving some basic results and ending by sketching the equivalence of the three frameworks described above. Most of the results of this section are well known and appear in \cite{adamek-rosicky} or \cite{mu-aec-jpaa}.

First, we recall some more category-theoretic terminology. Monomorphisms are in a sense a very weak generalization of the notion of an injection\footnote{For example, in the category of divisible abelian groups, the quotient map $\mathbb{Q} \to \mathbb{Q} / \mathbb{Z}$ turns out to be a monomorphism \cite[7.33(5)]{joy-of-cats}.}. A much stronger one is given by the following definition: suppose we have $A \xrightarrow{i} B \xrightarrow{r} A$ which compose to the identity ($ri = \id_A$). Then we call $i$ a \emph{section} (or \emph{split monomorphism}) and $r$ a \emph{retraction} (or \emph{split epimorphism}). In such a situation, we say that $A$ is a \emph{retract} of $B$. It is easily checked that sections and retractions are indeed monomorphisms and epimorphisms respectively. The canonical inclusions $A \to A \oplus B$ and projections $A \oplus B \to A$ in the category of $R$-modules are good examples of sections and retractions. In case we also have that $ir = \id_B$, then we write $r = i^{-1}$, $i = r^{-1}$, and call them \emph{isomorphisms}. Note that a retraction which is also a mono is an isomorphism \cite[7.36]{joy-of-cats}. Thus when all morphisms are monos (e.g.\ in an $\infty$-AEC), any retraction (and thus any section) is an isomorphism.

Our first goal will be to show that any object in an accessible category is presentable. This follows from the following result, which essentially says that a small union (colimit) of small objects is small (a diagram is called \emph{$\lambda$-small} if its indexing category has strictly less than $\lambda$-many objects):

\begin{thm}\label{colimit-size}
  For $\lambda$ a regular cardinal, a colimit of a $\lambda$-small diagram of $\lambda$-presentable objects is $\lambda$-presentable. In particular all the objects of an accessible category are presentable.
\end{thm}
\begin{proof}[Proof sketch]
  Let $D: I \to \ck$ be a $\lambda$-small diagram consisting of $\lambda$-presentable objects, with colimit $(D_i \xrightarrow{d_i} A)_{i \in I}$. Let $B$ be a $\lambda$-directed colimit of another diagram $E: J \to \ck$, and let $A \xrightarrow{f} B$. For each $i \in I$, the map $f d_i$ factors through some $E_{j_i}$, $j_i \in J$, by $\lambda$-presentability of $D_i$. Since $J$ is $\lambda$-directed and $|I| < \lambda$, there exists $j \in J$ so that $j \ge j_i$ for all $i \in I$. Then the universal property of the colimit implies that $f$ must factor through $E_j$. The ``in particular'' part follows from the rest because in a $\lambda$-accessible category, any object is (by the smallness condition) a colimit of a (set-sized) diagram consisting of $\lambda$-presentable objects. 
\end{proof}

We now work toward proving that an accessible category will, for each $\lambda$, have only a set (up to isomorphism) of $\lambda$-presentable objects. We start with the following technical observations:

\begin{remark}\label{retract-rmk} Work in a category $\ck$.
  \begin{enumerate}
  \item If $A_1 \xrightarrow{i_1} B \xrightarrow{r_1} A_1$ and $A_2 \xrightarrow{i_2} B \xrightarrow{r_2} A_2$  are such that $r_\ell i_\ell = \id_{A_\ell}$ (i.e.\ they are section/retraction pairs) and $i_1 r_1 = i_2 r_2$, then $A_1$ and $A_2$ are isomorphic (as witnessed by $r_2 i_1$ and $r_1 i_2$). Since $B$ has only a set of endomorphisms (i.e.\ morphisms from and to itself), there is only a set (up to isomorphism) of retracts of any given object.
  \item\label{retract-rmk-2} A retract of a $\lambda$-presentable object is $\lambda$-presentable (straightforward diagram chase from the definition of $\lambda$-presentability).
  \item\label{retract-rmk-3} If a $\lambda$-presentable object $A$ is a $\lambda$-directed colimit of a diagram $D: I \to \ck$, with colimit cocone $(D_i \xrightarrow{d_i} A)_{i \in I}$, then $d_i$ is a retraction for some $i \in I$, so $A$ is a retract of $D_i$ (by $\lambda$-presentability, the identity map on $A$ must factor through one of the components of the diagram: $\id_A = d_i f$ for some $i \in I$).
  \end{enumerate}
\end{remark}

We obtain:

\begin{lem}\label{pres-set}
  A $\lambda$-accessible category has, up to isomorphism, only a set of $\lambda$-presentable objects.
\end{lem}
\begin{proof}
  Let $S$ be the set of $\lambda$-presentable objects given by the smallness condition. By Remark \ref{retract-rmk}, any $\lambda$-presentable objects will be a retract of members of $S$, and there is only a set of such retracts.
\end{proof}

In order to say more, we try to understand when a $\lambda$-accessible will be $\mu$-accessible for $\mu > \lambda$. This is \emph{not} a trivial consequence of the definition because a $\lambda$-directed poset may not be $\mu$-directed. In fact, as we will see, it is \emph{not} true in general that a $\lambda$-accessible category is $\mu$-accessible for all $\mu > \lambda$ (it turns out that the \emph{accessibility spectrum} -- the class of cardinals $\lambda$ such that a category is $\lambda$-accessible -- is an interesting measure of the complexity of the category, see Section \ref{set-size-sec}). Before looking at counterexamples, let us state a positive result. For a regular cardinal $\mu$, an infinite cardinal $\lambda$ is called \emph{$\mu$-closed}\footnote{When $\lambda$ is regular, this is written $\mu \ll \lambda$ in \cite[A.2.6.3]{htt-lurie}. The notation can be a misleading though, because it is \emph{not} true that $\mu \ll \lambda < \lambda'$ implies $\mu \ll \lambda'$ (take for example $\mu = \aleph_1$, $\lambda = \aleph_2$, $\lambda' = \aleph_{\omega + 1}$).} if $\theta^{<\mu} < \lambda$ for all $\theta < \lambda$. Note that any uncountable cardinal is $\aleph_0$-closed and in general for any fixed $\mu$ there is a proper class of regular $\mu$-closed cardinal (given any infinite cardinal $\lambda_0$, the cardinal $\left(\lambda_0^{<\mu}\right)^+$ is always $\mu$-closed).

\begin{thm}[Raising the index of accessibility]\label{raise-index}
  Let $\mu \le \theta \le \lambda$ be regular cardinals and let $\ck$ be a $\theta$-accessible category with $\mu$-directed colimits. If $\lambda$ is $\mu$-closed, then $\ck$ is $\lambda$-accessible. 
\end{thm}
\begin{proof}
  Given an object $A$ of $\ck$, we have to write $A$ as a $\lambda$-directed colimit of $\lambda$-presentable objects. First, by $\theta$-accessibility we know we can write $A$ as a $\theta$-directed colimit of a diagram $D: I \to \ck$ of $\theta$-presentable objects. Now let $J$ be the poset of all $\mu$-directed subsets of $I$ of cardinality strictly less than $\lambda$, ordered by containment. For each $I_0 \in J$, we can use that $\ck$ has $\mu$-directed colimits to look at the colimit $\colim (D \rest I_0)$ of the diagram $D$ restricted to $I_0$. This process induces a new diagram $E: J \to \ck$, where $E_{I_0} = \colim (D \rest I_0)$. Notice that by Theorem \ref{colimit-size}, $E$ consists of $\lambda$-presentable objects. Further, because $\lambda$ is $\mu$-closed, any subset of $I$ of cardinality strictly less than $\lambda$ will be contained in some member of $J$. In particular, $J$ is $\lambda$-directed and $\colim E = \colim D = A$. Thus any object is a $\lambda$-directed colimits of $\lambda$-presentable objects. The argument also shows that each of these $\lambda$-presentable object is a colimit of $\mu$-presentable objects indexed by a poset of cardinality strictly less than $\lambda$. By Lemma \ref{pres-set}, there is only a set of $\mu$-presentable objects, hence only a set of such diagrams up to isomorphism, so there is only a set of $\lambda$-presentable objects.
\end{proof}
\begin{cor}
  Any accessible category is $\lambda$-accessible for a proper class of cardinals $\lambda$. Moreover, any $\theta$-accessible category with directed colimits (in particular any finitely accessible or locally $\theta$-presentable category) is $\lambda$-accessible for \emph{all} regular cardinals $\lambda > \theta$.
\end{cor}
\begin{cor}
  For any regular cardinal $\lambda$, an accessible category has only a set, up to isomorphism, of $\lambda$-presentable objects.
\end{cor}
\begin{proof}
  Let $\ck$ be an accessible category and fix $\theta \ge \lambda$ regular such that $\ck$ is $\theta$-accessible. By Lemma \ref{pres-set}, $\ck$ has only a set of $\theta$-presentable objects, hence (because $\lambda$-presentable implies $\theta$-presentable) a set of $\lambda$-presentable objects.
\end{proof}
\begin{remark}
In the proof of Theorem \ref{raise-index}, the hypothesis that $\lambda$ is $\mu$-closed was used to show that for any $\mu$-directed poset $I$, any subset of $I$ of cardinality strictly less than $\lambda$ can be completed to a \emph{$\mu$-directed} subset of cardinality strictly less than $\lambda$. It more generally suffices to assume that for any $\theta < \lambda$, $\cf{[\theta]^{<\mu}} < \lambda$ (recall that $[\theta]^{<\mu}$ is the set of all subsets of $\theta$ of cardinality strictly less than $\mu$, ordered by containment). Following \cite[2.3.1]{makkai-pare}, we will write $\mu \tlt \lambda$ when this holds (one can check the relation $\tlt$ indeed gives a partial order on the regular cardinals). Since we always have that $\lambda \tlt \lambda^+$, it follows, for example, that any $\lambda$-accessible category is also $\lambda^+$-accessible \cite[2.13(2)]{adamek-rosicky}. However, when $\lambda > 2^{<\mu}$, $\lambda$ is $\mu$-closed if and only if $\mu \tlt \lambda$ (because of the equation $\lambda^{<\mu} = 2^{<\mu} \cdot \cf{[\lambda]^{<\mu}}$, well known to set theorists, see \cite[2.5]{internal-sizes-jpaa}). Below $2^{<\mu}$, the behavior of $\cf{[\lambda]^{<\mu}}$ can be somewhat understood through the lens of Shelah's PCF theory \cite{shg, cardarithm}.
\end{remark}
\begin{example}[{\cite[2.11]{adamek-rosicky}}]
  For $\mu$ a regular uncountable cardinal, let $\ck$ be the category of $\mu$-directed posets, with morphisms the order-preserving maps. One can check that $\ck$ is $\mu$-accessible. Let $\lambda > \mu$ be a regular cardinal such that there exists a $\theta < \lambda$ with $\cf{[\theta]^{<\mu}} \ge \lambda$ (so $\theta$ witnesses that $\mu \ntlt \lambda$; take for example $\theta = \beth_{\omega} (\mu)$, $\lambda = \theta^+$, see discussion above). Then $\ck$ is \emph{not} $\lambda$-accessible because the poset $[\theta]^{<\mu}$ is $\mu$-directed, hence an object of $\ck$, but cannot be written as a $\lambda$-directed colimit of $\lambda$-presentable objects. Indeed, suppose it can, and let $D: I \to \ck$ be the corresponding diagram. The images of the colimit maps form a collection $(Y_i)_{i \in I}$ of $\lambda$-presentable subsets of $[\theta]^{<\mu}$, and the $\lambda$-presentable objects are those of cardinality strictly less than $\lambda$, so each $Y_i$ has cardinality strictly less than $\lambda$. Each $\{\alpha\}$, for $\alpha < \theta$, lies in some $Y_i$, and because $I$ is $\lambda$-directed, there must exist $i^\ast \in I$ such that each $\{\alpha\}$ lies in $Y_{i^\ast}$. Because $Y_{i^\ast}$ is $\mu$-directed, it must in fact be cofinal in $[\theta]^{<\mu}$, but we know that $Y_{i^\ast}$ has cardinality strictly less than $\lambda$, contradiction.
\end{example}

Different types of categories have different appropriate notions of functors. For accessible categories, an accessible functors play an important role:

\begin{defin}
  A functor $F: \ck \to \cl$ is \emph{$\lambda$-accessible} if both $\ck$ and $\cl$ are $\lambda$-accessible categories and $F$ preserves $\lambda$-directed colimits. We say that $F$ is \emph{accessible} if it is $\lambda$-accessible for some $\lambda$. 
\end{defin}

One reason accessible functors are important is the \emph{adjoint functor theorem}. In general category theory, Freyd's adjoint functor theorem \cite[18.12]{joy-of-cats} tells us that a functor between complete categories is (left) adjoint if and only if it preserves limits and satisfies a technical ``solution set condition'' that can be quite difficult to check. The statement simplifies when looking at accessible categories (recall that a bicomplete accessible category is just a locally presentable category). We will not look into this direction much further, so we omit the proof.

\begin{thm}[The adjoint functor theorem, {\cite[1.66]{adamek-rosicky}}]
  A functor between two locally presentable categories is adjoint if and only if it preserves limits and is accessible.
\end{thm}
\begin{remark}
  It is also true that any left or right adjoint functor between accessible categories is accessible \cite[2.23]{adamek-rosicky}.
\end{remark}

A question we will be interested in is how a given functor interacts with sizes: for a regular cardinal $\lambda$, we say that a functor $F$ \emph{preserve $\lambda$-presentable objects} if $F A$ is $\lambda$-presentable whenever $A$ is $\lambda$-presentable. Accessible functors are useful because they preserve certain sizes:

\begin{thm}[The uniformization theorem; {\cite[2.19]{adamek-rosicky}}]\label{unif-thm}
  For any accessible functor $F$, there exists a proper class of regular cardinals $\lambda$ such that $F$ is $\lambda$-accessible and preserves $\lambda$-presentable objects.
\end{thm}
\begin{proof}
  Let $F: \ck \to \cl$ be a $\mu$-accessible functor. Up to isomorphism, there is only a set of $\mu$-presentable objects in $\ck$, so there exists a regular cardinal $\mu' \ge \mu$ such that $F A$ is $\mu'$-presentable for every $\mu$-presentable objects $A$. Now let $\lambda \ge \mu'$ be a $\mu$-closed cardinal. Let $A$ be a $\lambda$-presentable object in $\ck$. By the proof of Theorem \ref{raise-index}, $A$ can be written as a $\lambda$-directed colimit of objects that are each $\lambda$-small $\mu$-directed colimits of $\mu$-presentable objects. Since $A$ is $\lambda$-presentable, it must be a retract of an object of this diagram (Remark \ref{retract-rmk}(\ref{retract-rmk-3})): hence $A$ is a retract of a $\lambda$-small $\mu$-directed colimits of $\mu$-presentable objects. Since by hypothesis $F$ preserves such colimits and any functor preserves retractions, $F A$ is a retract of a $\lambda$-small $\mu$-directed colimit of $\mu'$-presentable objects in $\cl$. By Theorem \ref{colimit-size}, $F A$ is $\lambda$-presentable.
\end{proof}
\begin{remark}\label{unif-thm-more}
  The proof gives more: if $\mu \le \lambda_0 \le \lambda_1 \le \lambda$ are all regular, $F: \ck \to \cl$, $\ck$ and $\cl$ are both $\lambda_0$-accessible, $F$ preserves $\mu$-directed colimits, $\lambda$ is $\mu$-closed, and $F A$ is $\lambda_1$-presentable whenever $A$ is $\lambda_0$-presentable, then $F$ preserves $\lambda$-presentable objects. In particular, an accessible functor preserving directed colimits will preserve $\lambda$-presentable objects for all high-enough regular $\lambda$.
\end{remark}

Dually, it is natural to ask when a functor \emph{reflects $\lambda$-presentable objects}, i.e.\ when $F A$ $\lambda$-presentable implies $A$ $\lambda$-presentable. A sufficient condition is for $F$ to reflect split epimorphisms (if $F f$ is a split epi -- i.e.\ a retraction -- then $f$ is a split epi), see \cite[3.6]{beke-rosicky}. 

\begin{thm}\label{reflect-pres}
  If $F: \ck \to \cl$ is a $\lambda$-accessible functor reflecting split epimorphisms, then $F$ reflects $\lambda$-presentable objects.
\end{thm}
\begin{proof}
  Assume that $F A$ is $\lambda$-presentable. Since $\ck$ is $\lambda$-accessible, $A$ is a the colimit of a $\lambda$-directed diagram $D: I \to \ck$ consisting of $\lambda$-presentable objects, with colimit cocone $(D_i \xrightarrow{d_i} A)_{i \in I}$. Since $F$ preserves $\lambda$-directed colimits, $F A$ is the colimit of $F D$, so as $F A$ is $\lambda$-presentable, Remark \ref{retract-rmk}(\ref{retract-rmk-3}) implies $F d_i$ is a retraction for some $i \in I$. Because $F$ reflects split epis, $d_i$ is a retraction, so $A$ is a retract of the $\lambda$-presentable object $D_i$, hence is $\lambda$-presentable (Remark \ref{retract-rmk}(\ref{retract-rmk-2})).
\end{proof}
\begin{remark}\label{reflect-pres-rmk}
  A functor that reflects isomorphism and whose image contains only monomorphisms will automatically reflect split epimorphisms.
\end{remark}

\begin{example}\label{unif-ex} \
  \begin{enumerate}
  \item Let $\ck$ be a locally presentable category. The binary product functor $F: \ck \to \ck$ sending $A$ to the category-theoretic product $A \times A$ is accessible, because it is adjoint to the diagonal functor. Thus by the uniformization theorem, $F$ preserves $\lambda$-presentable objects for a proper class of $\lambda$. That is, $F$ does not make the product ``too much bigger''.
  \item \cite[3.2(4)]{beke-rosicky} For $\mu$ a regular cardinal, let $F: \Set \to \Set$ send the set $X$ to the set $[X]^{<\mu}$ of subsets of $X$ of cardinality strictly less than $\mu$. This is an accessible functor but, if $\mu$ is uncountable and $\lambda < \lambda^{<\mu}$, $F$ will not preserve $\lambda^+$-presentable objects.
  \item \cite[3.3]{beke-rosicky} Let $F: \Grp \to \Ab$ be the abelianization functor from the category of group to the category of abelian groups. This is a right adjoint functor, so it preserves colimits, and hence by Remark \ref{unif-thm-more} preserves $\lambda$-presentable objects for all regular cardinals $\lambda$. However, if $G$ is a simple group then its abelianization $F (G)$ is the zero group. Thus (since there exists simple groups in all infinite cardinalities) $F$ can make sizes drop. Indeed, one can check that $F$ does not reflect split epimorphisms.
  \item If $\ck$ is locally presentable and $U: \ck \to \Set$ is an accessible functor preserving limits, then by the adjoint functor theorem, $U$ is left adjoint. If we think of $U$ as a forgetful functor, the right adjoint will be the free functor. Since $\ck$ has directed colimits, $U$ will preserve $\lambda$-presentable objects for all high-enough regular cardinals $\lambda$.
  \item\label{unif-ex-aec} Let $\K$ be a $\mu$-AEC, and let $U: \K \to \Set$ be the forgetful universe functor. By definition of a $\mu$-AEC, $U$ preserves $\mu$-directed colimits, hence is accessible. The abstract class axiom ensures that $U$ reflects isomorphisms. Moreover, any morphism in the image of $U$ must be a monomorphism so by Remark \ref{reflect-pres-rmk}, $U$ reflects split epis. Finally, it is easy to check that any $\LS (\K)^+$-presentable object in $\K$ will have cardinality at most $\LS (\K)$ (write the object as the $\LS (\K)^+$-directed colimit of its subobject of cardinality at most $\LS (\K)$). By the uniformization theorem, if $\lambda > \LS (\K)$ is a regular $\mu$-closed cardinal, then $F$ preserves and reflects $\lambda$-presentable objects. This means that $A$ is $\lambda$-presentable in $\K$ if and only if it its universe has cardinality strictly less than $\lambda$. In particular (taking $\mu = \aleph_0$) in an AEC, category-theoretic sizes correspond exactly to cardinalities (above $\LS (\K)$). See \cite[\S 4]{internal-sizes-jpaa} for more on such results.
  \item In the $\aleph_1$-AEC $\K$ of Banach spaces (with subspace inclusions), the universe functor $U$ does \emph{not} preserve $\aleph_1$-presentable objects: an $\aleph_1$-presentable Banach space will not have countable cardinality. This is because $U$ does not preserve $\aleph_0$-directed colimits (even though $\K$ \emph{does} have those colimits, they are not concrete: one \emph{cannot} compute them by taking unions).
  \end{enumerate}
\end{example}

Let's use the uniformization theorem to better understand the relationship between the smallness axiom in $\infty$-AEC (Definition \ref{infty-aec-def}) and the smallness condition in the definition of an accessible category (Definition \ref{acc-def}).

\begin{thm}[{\cite[5.5]{beke-rosicky}},{\cite[7.2]{indep-categ-advances}}]\label{aec-acc}
  Assume that $\K$ satisfies all the axioms of a $\mu$-AEC (Definition \ref{infty-aec-def}), except perhaps for the LST smallness axiom. The following are equivalent:

  \begin{enumerate}
  \item\label{aec-acc-1} $\K$ is accessible.
  \item\label{aec-acc-2} $\K$ satisfies the LST smallness axiom, and hence is a $\mu$-AEC.
  \end{enumerate}
\end{thm}
\begin{proof}
  We have seen already (Remark \ref{aec-categ-prop}) that (\ref{aec-acc-2}) implies (\ref{aec-acc-1}). Assume now that (\ref{aec-acc-1}) holds: $\K$ is accessible. Then the universe functor $U: \K \to \Set$ is accessible (preservation of $\mu$-directed colimits is just because they are computed the same way in $\K$ and $\Set$ -- this is what the chain axiom says). By the uniformization theorem, we can pick a regular cardinal $\theta \ge \mu + |\tau (\K)|$ such that $U$ is $\theta$-accessible and preserves $\theta$-presentable objects. We will show that $\LS (\K) \le \theta^{<\mu}$. Let $M \in \K$, and let $A \subseteq U M$. Set $\lambda := \left(\left(\theta + |A|\right)^{<\mu}\right)^+$. By Remark \ref{unif-thm-more}, $U$ is $\lambda$-accessible and preserves $\lambda$-presentable objects. In particular, $\K$ is $\lambda$-accessible so we can write $M$ as a $\lambda$-directed colimit (union) of $\lambda$-presentable objects: $M = \bigcup_{i \in I} M_i$. Now $|A| < \lambda$, so by $\lambda$-directedness there must exist $i \in I$ such that $A \subseteq U M_i$. Since $M_i$ is $\lambda$-presentable in $\K$, $U M_i = M_i$ must be $\lambda$-presentable in $\Set$, hence it must have cardinality strictly less than $\lambda$ (see Example \ref{pres-ex}(\ref{pres-ex-set})), as desired.
\end{proof}

We can now sketch part of the proof that any accessible category can be presented as the class of models of an $\Ll_{\infty, \infty}$ sentence. Let us make such a statement precise first: recall that two categories $\ck$ and $\cl$ are \emph{equivalent} if there exists a functor $F: \ck \to \cl$ that is full (surjective on morphisms), faithful (injective on morphisms), and essentially surjective on objects (any object of $\cl$ is isomorphic to an object in the image of $F$). Essential surjectivity (rather than surjectivity) on objects makes equivalence of categories weaker than \emph{isomorphism} of categories, but it is typically the former notion of ``being the same'' that is used in category theory\footnote{For example, the categories of finite dimensional vector spaces over $\mathbb{R}$ and of real matrices (where the objects are natural numbers and the morphisms matrices with the right dimension, with composition defined by matrix multiplication) are equivalent but not isomorphic, see \cite[3.35(2)]{joy-of-cats}.}. For $\infty \ge \lambda \ge \mu$, let us then define a category to be \emph{$(\lambda, \mu)$-elementary} if it is equivalent to a category of the form $\Mod (T)$, for $T$ a theory in $\Ll_{\lambda, \mu}$. We will see that \emph{any} $\lambda$-accessible category is $(\infty, \lambda)$-elementary. The proof proceeds in two steps. First, the category is embedded into a category of structures, and second this category of structures is axiomatized. We will only look at the first step (which is sufficient if one only cares about $\lambda$-AECs). The main idea is to represent each object using its Hom functor. The second step will be proven in the very special case of universal classes.

\begin{thm}[{\cite[4.8]{mu-aec-jpaa}}]\label{acc-to-aec}
  If $\ck$ is a $\lambda$-accessible category, then there is a (finitary) vocabulary $\tau$ and an embedding\footnote{That is, a functor which is faithful and injective on objects.} $E: \ck \rightarrow \Str (\tau)$ which is full and preserves $\lambda$-directed colimits. In particular, if in addition all the morphisms of $\ck$ are monomorphisms, $\ck$ is equivalent to a $\lambda$-AEC.
\end{thm}
\begin{proof}[Proof sketch]
  Let $\ck_\lambda$ be a small full subcategory of $\ck$ containing representatives of each $\lambda$-presentable object. For each $M \in \ck_\lambda$, let $S_M$ be a unary relation symbol and for each morphism $f$ in $\ck_\lambda$, let $\underline{f}$ be a binary function symbol. The vocabulary $\tau$ will consist of all such $S_M$ and $\underline{f}$. Now map each $M \in \ck$ to the following $\tau$-structure $E M$:

  \begin{enumerate}
  \item Its universe are the morphisms $g: M_0 \rightarrow M$, where $M_0 \in \ck_\lambda$.
  \item For each $M_0 \in \ck_\lambda$, $S_{M_0}^{EM}$ is the set of morphisms $g: M_0 \rightarrow M$.
  \item For each morphism $f: M_0 \rightarrow M_1$ of $\ck_\lambda$, and each $g: M_1 \rightarrow M$, $\underline{f}^{E M} (g) = g f$. When $g \notin S_{M_1}^{EM}$, just let $\underline{f}^{EM} (g) = g$.
  \end{enumerate}

  Map each morphism $f: M \rightarrow N$ to the function $\bar{f} : E M \rightarrow E N$ given by $\bar{f} (g) = fg$. Essentially, $E$ is the functor $E (M) = \Hom (-, M)$ from $\ck$ to $\fct{\ck_\lambda^{op}}{\Set}$. The Yoneda embedding lemma tells us that $E \rest \ck_\lambda$ is full and faithful \cite[6.20]{joy-of-cats}. The definition of a $\lambda$-presentable object also ensures that $E$ preserves $\lambda$-directed colimits. By writing any object as a $\lambda$-directed colimit of $\lambda$-presentable objects, we get that $E$ is full and faithful. To see the ``in particular'' part, consider the smallest isomorphism-closed subcategory $\cl$ of $\Str (\tau)$ that contains $E[\ck]$. This category is equivalent to $\cl$, satisfies the abstract class axiom, has concrete $\lambda$-directed colimits, and (trivially, because the morphisms are homomorphisms) satisfies the coherence axiom. Since $\ck$ is accessible and $E$ is full and faithful, $E[\ck]$ also is accessible, and hence $\cl$ is accessible. Now apply Theorem \ref{aec-acc}.
\end{proof}
\begin{cor}[{\cite[5.35]{adamek-rosicky}}]\label{acc-to-log}
  Any $\lambda$-accessible category is $(\infty, \lambda)$-elementary. In particular, a category is accessible if and only if it is $(\infty, \infty)$-elementary.
\end{cor}
\begin{proof}
  We first apply Theorem \ref{acc-to-aec} to reduce the problem to axiomatizing a class of structures, and then axiomatize this class (we will not explain how here).
\end{proof}

\begin{example}\label{acc-to-log-ex} \
  \begin{enumerate}
  \item For $\lambda$ a regular cardinal, we have seen that the AEC of all sets of cardinality at least $\lambda$ is not $\lambda$-accessible. In fact, this AEC is not even $(\infty, \lambda)$-elementary but the proof is not trivial, see \cite{henry-aec-uncountable-jsl}.
  \item\label{toy-quasi-ex} \cite[\S4]{logic-intersection-bpas} Consider the following AEC, $\K$, sometimes called the \emph{toy quasiminimal class}: the vocabulary contains a single binary relation. The objects of $\K$ are the equivalence relations with countably infinite classes. The ordering says that equivalence classes do not grow. It is not difficult to see that $\K$ is finitely accessible: an object is $\lambda$-presentable if and only if the relation has fewer than $\lambda$-many classes, and every equivalence relation is the directed colimits of its restrictions to finitely-many classes. By Corollary \ref{acc-to-log}, $\K$ is $(\infty, \omega)$-elementary, hence is equivalent to the category of models of an $\Ll_{\infty, \omega}$-theory. This latter category is obtained, roughly speaking, by collapsing each equivalence class to a point. However it is well known that $\K$ itself is \emph{not} the category of models of a basic $\Ll_{\infty, \omega}$ theory. In fact, it is not \emph{finitary} in the sense of Hyttinen-Kesälä \cite{finitary-aec}: roughly, one \emph{cannot} figure out whether a map is a morphism by checking finitely-many points at a time (like we would be able to do for homomorphisms in a finitary vocabulary). This shows that the concept of being a finitary AEC is \emph{not} invariant under equivalence of categories.
  \item\label{inter-ex} An $\infty$-AEC $\K$ has \emph{intersections} if for any $M \in \K$, and any non-empty collection $\{M_i : i \in I\}$ of $\lea$-substructures of $M$, $\bigcap_{i \in I} M_i$ induces a $\lea$-substructure of $M$. Generalizing the previous example, one can show that any $\lambda$-AEC with intersections is a $\lambda$-accessible category \cite[3.3]{multipres-pams}, hence $(\infty, \lambda)$-elementary. In fact, as the definition makes apparent, $\lambda$-AECs with intersections correspond exactly to the $\lambda$-accessible categories with wide pullbacks (and all morphisms monos), see \cite[5.7]{multipres-pams}. This gives a clear sense in which this class of AECs is natural, and less complex than general AECs. See \cite[\S2]{ap-universal-apal}, \cite{logic-intersection-bpas}, or \cite[\S5]{internal-sizes-jpaa} for more on AECs with intersections.
  \end{enumerate} 
\end{example}

As a consolation prize for not proving the axiomatizability part of Corollary \ref{acc-to-log}, let us prove it for a simpler framework: that of universal classes.

\begin{defin}\label{univ-def}
  For $\mu$ a regular cardinal, a \emph{$\mu$-universal class} is a $\mu$-AEC $\K$ such that $\lea$ is just the $\tau (\K)$-substructure relation, and moreover if $M \in \K$ and $M_0 \subseteq M$, then $M_0 \in \K$. Said another way, a $\mu$-universal class is simply a class of structures in a fixed $\mu$-ary vocabulary that is closed under isomorphisms, substructures, and $\mu$-directed unions (we identify the class with the $\mu$-AEC). When $\mu = \aleph_0$, we omit it.
\end{defin}

\begin{thm}[Tarski's presentation theorem; \cite{tarski-th-models-i}, {\cite[2.2]{multipres-pams}}]\label{tarski}
  Let $\mu$ be a regular cardinal and let $K$ be a class of structures in a fixed $\mu$-ary vocabulary $\tau$. The following are equivalent:

  \begin{enumerate}
  \item\label{tarski-1} $K$ is the class of models of a universal $\Ll_{\infty, \mu}$ theory (i.e.\ a theory where each sentence is of the form $(\forall \bx) \psi$, with $\psi$ quantifier-free).
  \item\label{tarski-2} $K$ is a $\mu$-universal class.
  \end{enumerate}
\end{thm}
\begin{proof}[Proof sketch]
  The implication (\ref{tarski-1}) implies (\ref{tarski-2}) is easy to check. Assume now that (\ref{tarski-2}) holds. Call $M \in K$ \emph{$\mu$-generated} if $M = \ccl^M (A)$, for some $A \in [U M]^{<\mu}$ (here, $\ccl^M$ denotes the closure of $A$ under the functions of $M$ -- note that such a closure is always a substructure of $M$, hence in $K$ by definition of a $\mu$-universal class). Note that $\mu$-generated is equivalent to $\mu$-presentable \cite[5.7]{internal-sizes-jpaa}, but this will not be needed. Now by listing all the isomorphism types of the $\mu$-generated models in $K$, form a quantifier-free formula $\phi (\bx)$ of $\Ll_{\infty, \mu}$ such that $N \models \phi[\ba]$ if and only if $\ccl^N (\ran(\ba))$ is in $K$. The formula $\psi := (\forall \bx) \phi (\bx)$ axiomatizes $K$. Indeed, if $N \in K$ then $N \models \psi$. Conversely, if $N \models \psi$ then $\{\ccl^N (\ran(\ba)) \mid \ba \in \fct{<\mu}{U N}\}$ is a $\mu$-directed system in $K$ whose union is $N$, hence $N$ is in $K$. This proves (\ref{tarski-1}).
\end{proof}
\begin{remark}
  There is also a category-theoretic characterization of $\mu$-universal classes: they are the $\mu$-accessible categories with all morphisms monos and all connected limits \cite[5.9]{multipres-pams}.
\end{remark}

\begin{example}\label{univ-ex} \
  \begin{enumerate}
  \item The class of all vector spaces over a fixed field $F$ is a universal class.
  \item The class of locally finite groups is a universal class (it is a good exercise to try to write down the universal sentence of \ref{tarski}(\ref{tarski-1}) for this case).
  \item\label{univ-skol-ex} If $T$ is an $\Ll_{\infty, \infty}$ theory, we can ``Skolemize it'' (add functions to pick witnesses of existential sentences) to get a universal class in an expanded vocabulary, whose restriction to the original vocabulary is the class of models of $T$. 
  \item Any $\mu$-universal class is a $\mu$-AEC with intersections (see Example \ref{acc-to-log-ex}(\ref{inter-ex})).    
  \item The AEC of algebraically closed fields (with subfield) is not universal: the rationals form a non-algebraically closed subfield of the complex numbers. This AEC does have intersections, however.
  \item The class of all linear orders, or the class of all graphs, is a universal class.
  \end{enumerate}
\end{example}

Universal classes are in a sense the simplest type of AECs. In fact, a key result is Shelah's presentation theorem \cite[I.1.9]{shelahaecbook}, which says that, given an AEC $\K$, we can find a universal class\footnote{Shelah stated his result in terms of a class of models omitting types, but the proof of Theorem \ref{tarski} shows that any universal class is a type-omitting class: omit the types of $\mu$-generated models outside of $K$.} $\K^\ast$ in an expanded vocabulary whose reduct (i.e.\ restriction to the original vocabulary) to $\K$ is $\K$. In fact, the reduct gives a functor from $\K^\ast$ to $\K$ which is faithful (injective on morphisms), surjective, and preserves directed colimits. This is the motivation for the following generalization of Shelah's presentation theorem to accessible categories with all morphisms monos.

\begin{thm}[Shelah's presentation theorem, categorical version {\cite[4.4]{internal-improved-v3-toappear}}]\label{shelah-pres}
  If $\cl$ is an accessible category with $\mu$-directed colimits and all morphisms monos, then there exists a $\mu$-universal class $\K$ and a faithful essentially surjective functor $H: \K \to \cl$ preserving $\mu$-directed colimits.
\end{thm}
\begin{proof}[Proof sketch]
  First, let $\cl^\ast$ be the category obtained by taking free $\mu$-directed colimits of the $\mu$-presentable objects of $\cl$. This is a $\mu$-accessible category \cite[2.26]{adamek-rosicky}, and the natural ``projection'' functor $F: \cl^\ast \to \cl$ is essentially surjective and preserves directed colimits. Moreover, $F$ is also faithful: given two distinct morphisms $f, g: A \to B$ in $\cl^\ast$, separate them on $\lambda$-presentable objects by finding $f_0, g_0: A_0 \to B_0$ distinct with $A_0$ and $B_0$  $\lambda$-presentable and maps $u: A_0 \to A$, $v: B_0 \to B$ such that $fu = v f_0$, $g u = v g_0$:

  $$
  \xymatrix@=3pc{
    A \ar@<-.5ex>[r]_f \ar@<.5ex>[r]^g & B    \\
    A_0 \ar[u]_{u} \ar@<-.5ex>[r]_{f_0} \ar@<.5ex>[r]^{g_0} & B_0 \ar[u]_v    \\    
    }
    $$

  Since $F$ is faithful on $\lambda$-presentable objects, $F f_0 \neq F g_0$. Moreover, $F v$ is a mono, so it follows that $F f \neq F g$. By Corollary \ref{acc-to-log}, $\cl^\ast$ is $(\infty, \mu)$-elementary, hence without loss of generality is the category of models of a basic $\Ll_{\infty, \mu}$ theory $T$. Skolemizing $T$ (see Example \ref{univ-ex}(\ref{univ-skol-ex})), one obtains a $\mu$-universal class $\K$ so that the reduct functor $G: \K \to \cl^\ast$ is faithful, essentially surjective, and preserves $\mu$-directed colimits. Let $H := F \circ G$.
\end{proof}

\begin{remark}\label{pres-rmk} \
  \begin{enumerate}
  \item Any $\mu$-AEC is an accessible category with $\mu$-directed colimits and all morphisms monos (Remark \ref{aec-categ-prop}), so when $\mu = \aleph_0$ we in particular recover a version of Shelah's original presentation theorem. Note that, when $\mu > \aleph_0$, it is not so easy to imitate Shelah's proof, see \cite[\S3]{internal-improved-v3-toappear} for a discussion.
  \item\label{pres-morley-rmk} In the case $\mu = \aleph_0$, a classical result of model theory, \emph{Morley's omitting type theorem}  tells us in particular (in categorical language, see \cite[3.4.1]{makkai-pare}) that for any $(\infty, \omega)$-elementary category $\ck$ (hence by Tarski's presentation theorem in particular for a universal class), there is a faithful functor $\Lin \to \K$ preserving directed colimits (where $\Lin$ is the category of linear orders and order-preserving maps). The functor constructed by the model-theoretic process is called the \emph{Ehrenfeucht-Mostowski (EM) functor}, and its images are called \emph{Ehrenfeucht-Mostowski models}. They can be described as the models generated by order-indiscernibles satisfying a fixed collection of quantifier-free types (an \emph{EM blueprint}). In fact, from any faithful functor $\Lin \to \ck$ preserving directed colimits, one can recover such a blueprint \cite[5.5]{boney-er-v2}. Combining Morley's theorem with Tarski's presentation theorem and  the categorical version of Shelah's presentation theorem, we obtain that for any large accessible category $\cl$ with directed colimits and all morphisms monos there is a faithful functor $\Lin \to \cl$ preserving directed colimits. Notice that this works not only for AECs but also for continuous classes (such as Banach spaces). 
  \end{enumerate}
\end{remark}

\section{Set-theoretic topics}\label{set-sec}

\subsection{Cardinality vs presentability}\label{set-size-sec}

Is the behavior of presentability, the notion of size defined in Definition \ref{pres-def}, similar to the behavior of cardinality in concrete classes? There are at least two questions one can consider in this direction. First, in the category of sets, a set is $\lambda$-presentable if and only if it has cardinality strictly less than $\lambda$. In particular, the presentability rank of a set is a successor. This happens in all the examples listed in \ref{pres-ex}. Thus one can ask: \emph{in an arbitrary accessible category, are high-enough presentability ranks always successors}?

Second, the Löwenheim-Skolem theorem for first-order logic says that any theory with an infinite model has models of all high-enough cardinalities. Is there a similar version for accessible categories? Since we do not have the compactness theorem, let us restrict ourselves to \emph{large} categories, and let us consider only eventual behavior. Cardinality here is again not the right notion (consider Example \ref{pres-ex}(\ref{pres-ex-hilb}): there are no Hilbert spaces in certain cardinalities), but we can still ask: \emph{does every large accessible category have objects of each high-enough size?} (see Definition \ref{size-def}). Following \cite[2.4]{beke-rosicky}, let us call an accessible category $\ck$ \emph{LS-accessible} if there exists a cardinal $\theta$ such that $\ck$ has an object of size $\lambda$ for all cardinals $\lambda \ge \theta$. We are then asking whether every large accessible category is LS-accessible.

Both questions turn out to be connected to the \emph{accessibility spectrum}: the set of regular cardinals $\lambda$ such that a given category is $\lambda$-accessible. Regarding the first question, we have:

\begin{thm}[{\cite[5.3]{internal-improved-v3-toappear}}]\label{succ-rank-thm}
  Let $\ck$ be a category and let $\lambda$ be a weakly inaccessible cardinal\footnote{A cardinal is \emph{weakly inaccessible} if it is regular and limit.}. If $\ck$ is $\mu$-accessible for unboundedly-many regular $\mu < \lambda$, then no object of $\ck$ has presentability rank $\lambda$.
\end{thm}
\begin{proof}[Proof idea]
  We show more generally that every object of $\ck$ is a $\lambda$-directed colimits of objects of presentability rank strictly less than $\lambda$. This will imply the result by Remark \ref{retract-rmk}. Let $S$ be the set of regular cardinals $\mu < \lambda$ such that $\ck$ is $\mu$-accessible. Given an object $A$, we know that for each $\mu \in S$, we can write $A$ as a $\mu$-directed colimit of a diagram $D^\mu : I^\mu \to \ck$ consisting of $\mu$-presentable objects. One can then put together the $D^\mu$'s to obtain a new $\lambda$-directed diagram with colimit $A$ and consisting of objects of presentability rank strictly less than $\lambda$.
\end{proof}
\begin{cor}\label{succ-cor} \
  \begin{enumerate}
  \item\label{succ-cor-1} \cite[4.2]{beke-rosicky} In an accessible category with directed colimits, all high-enough presentability ranks are successors.
  \item\label{suc-cor-2} \cite[3.11]{internal-sizes-jpaa} If the singular cardinal hypothesis (SCH) holds\footnote{The singular cardinal hypothesis is the statement that $\lambda^{\cf{\lambda}} = \lambda^+ + 2^{\cf{\lambda}}$ for every infinite cardinal $\lambda$.}, then in every accessible category, all high-enough presentability ranks are successors.
  \end{enumerate}
\end{cor}
\begin{proof} \
  \begin{enumerate}
  \item By Theorem \ref{raise-index}, $\ck$ is $\mu$-accessible for \emph{all} high-enough regular $\mu$, so the result is immediate from Theorem \ref{succ-rank-thm}.
  \item Assume that $\ck$ is $\mu$-accessible, and let $\lambda > 2^{<\mu}$ be a weakly inaccessible cardinal. Let $\theta_0 < \lambda$ be infinite and let $\theta := \left(\theta_0^{<\mu}\right)^+$. By the SCH hypothesis, $\theta < \lambda$ (see \cite[5.22]{jechbook}) and by Theorem \ref{raise-index}, $\ck$ is $\theta$-accessible. We have shown that $\ck$ is accessible in unboundedly-many regular cardinals below $\lambda$, hence  by Theorem \ref{succ-rank-thm} $\ck$ has no objects of presentability rank $\lambda$.
  \end{enumerate}
\end{proof}
\begin{remark}
  Since SCH holds above a strongly compact cardinal \cite[20.8]{jechbook}, we can replace the SCH assumption by a large cardinal axiom.
\end{remark}

The rough idea here is that if an accessible category is of ``low-enough'' complexity (as measured by its accessibility spectrum), then automatically its presentability ranks will have good behavior. Often additional cardinal arithmetic hypotheses can lower the complexity of the accessibility spectrum, to the point that we can prove results about general accessible categories. Let us now give an example (without proof) of this behavior for the second question above, whether every large accessible category has objects of all high-enough sizes (see Definition \ref{size-def}):

\begin{thm}[{\cite[6.11]{internal-improved-v3-toappear}}]\label{ls-acc-thm}
  Assume SCH. If $\ck$ is a large accessible category with all morphisms monos, and $\lambda$ is a high-enough successor cardinal such that $\ck$ is $\lambda$-accessible, then $\ck$ has an object of presentability rank $\lambda$.
\end{thm}
\begin{cor}\label{ls-sch} \
  Assume SCH.
  \begin{enumerate}
  \item\label{ls-sch-1} Any large accessible category with directed colimits and all morphisms monos is LS-accessible.
  \item\label{ls-sch-2} \cite[7.12]{internal-improved-v3-toappear} Any large accessible category has objects in all sizes of high-enough cofinality.
  \end{enumerate}
\end{cor}
\begin{proof}[Proof idea] \
  \begin{enumerate}
  \item Immediate from Theorems \ref{raise-index} and \ref{ls-acc-thm}.
  \item This can be obtained from Theorem \ref{ls-acc-thm} by combining careful use of the proofs of the uniformization theorem and Theorem \ref{reflect-pres}, together with the (nontrivial) fact that the inclusion $\ck_{mono} \to \ck$ is an accessible functor \cite[6.2]{indep-categ-advances}.
  \end{enumerate}
\end{proof}
\begin{remark}
  Corollary \ref{ls-sch}(\ref{ls-sch-1}) is due to Lieberman and Rosický, and can be proven without SCH \cite[2.7]{ct-accessible-jsl}: let $E: \Lin \to \ck$ be faithful and preserving directed colimits (see Remark \ref{pres-rmk}(\ref{pres-morley-rmk})). By the uniformization theorem, $E$ preserves $\lambda$-presentable objects for all high-enough regular $\lambda$, and by Theorem \ref{reflect-pres}, $E$ also reflects them (faithful functors reflect epimorphisms, and epimorphisms in $\Lin$ are isomorphisms). Thus for $\lambda$ a big-enough cardinal and $I$ a linear order of cardinality $\lambda$, $E (I)$ will have size exactly $\lambda$.
\end{remark}
    
\subsection{Large cardinals and images of accessible functors}\label{set-func-sec}

Consider the functor $F: \Set \to \Ab$ that associates to each set the free abelian group on that set. It is easily checked that $F$ is an accessible functor but, as noticed before (Example \ref{acc-ex}(\ref{free-ex})) the question of whether the image of $F$ (i.e.\ the full subcategory of $\Ab$ consisting of free abelian groups) is accessible is set-theoretic. One can ask this question generally: when is the image of an accessible functor accessible? The problem of course lies in proving existence of sufficiently directed colimits for this image. For technical reasons, we will close the image under subobjects: in the example of the free abelian group functor, subgroup of free groups are free, so the image is already closed under subobjects the only challenge is to check that a sufficiently directed diagram consisting of free groups has a cocone.

More precisely, define the \emph{powerful image} of an accessible functor $F: \ck \to \cl$ to be the smallest full subcategory $P$ of $\cl$ that contains $F[\ck]$ and is closed under subobjects (i.e.\ if $A \to B$ is a monomorphism and $B \in P$, then $A \in P$). The question becomes: when is the powerful image of an accessible functor accessible? The following result is due to Makkai and Paré:

\begin{thm}[{\cite[5.5.1]{makkai-pare}}]\label{mp-thm}
  If there is a proper class of strongly compact cardinals\footnote{A regular cardinal $\kappa$ is \emph{strongly compact} if every theory in $\Ll_{\kappa, \kappa}$ with all its subsets of size strictly less than $\kappa$ consistent (i.e.\ with a model) is consistent. See for example \cite[20.2]{jechbook}.}, then the powerful image of any accessible functor is accessible.
\end{thm}
\begin{proof}[Proof idea]
  Fix a $\lambda$-accessible functor $F$ with powerful image $P$. By the uniformization theorem, we can assume without loss of generality that $F$ preserves $\lambda$-presentable objects. Let $\kappa > \lambda$ be strongly compact. We show that $P$ has $\kappa$-directed colimits. Using ideas around Corollary \ref{acc-to-log}, we can reduce the problem of finding a cocone to the consistency of a certain $\Ll_{\kappa, \kappa}$-theory. This theory can be shown to have all subsets of size strictly less than $\kappa$ consistent hence, by the compactness theorem for $\Ll_{\kappa, \kappa}$, to be consistent. 
\end{proof}
\begin{remark}
  The large cardinal assumption can be slightly weakened to a proper class of \emph{almost} strongly compact cardinals \cite{btr-almost-compact-tams} but this is best possible \cite{lc-tame-pams}: the powerful image of every accessible functor is accessible if and only if there is a proper class of almost strongly compact cardinals. However, weaker statements can be proven from weaker large cardinal axioms (e.g.\ measurable or weakly compacts), see \cite{lieberman-almost-measurable-v4,bl-powerful-images-v4}. We will even mention a ZFC theorem about images of accessible functors (Theorem \ref{sing-compact}).
\end{remark}

Questions about the image of accessible functors can be used to study various kinds of compactness. One example is tameness in AECs (\cite[5.2]{ct-accessible-jsl}). As a simpler example, we consider the following property:

\begin{defin}\label{ap-def}
  An object $A$ in a category is an \emph{amalgamation base} if any span $C \leftarrow A \rightarrow B$ can be completed (not necessarily canonically) to a commutative square:

  $$
  \xymatrix@=3pc{
    C \ar@{.>}[r] & D \\
    A \ar[r] \ar[u] & B \ar@{.>}[u]
    }
  $$
  
  A category has the \emph{amalgamation property} (or \emph{has amalgamation}) if every object is an amalgamation base.
\end{defin}

A question one might ask is whether amalgamation up to a certain level (e.g.\ for all $\lambda$-presentable objects for some big $\lambda$) implies amalgamation the rest of the way. Large cardinals imply a simple answer (earlier results used model-theoretic techniques, see e.g.\ \cite{baldwin-boney}):

\begin{thm}[{\cite[3.6]{bootstrapping-accessible}}]
  If $\ck$ is a $\lambda$-accessible category, $\kappa > \lambda$ is strongly compact, and the full subcategory of $\ck$ consisting of $\kappa$-presentable objects has amalgamation, then $\ck$ has amalgamation.
\end{thm}
\begin{proof}[Proof sketch]
  Let $\ck^{\square}$ be the category of commutative squares in $\ck$, and let $\ck^{sp}$ be the category of spans $B \leftarrow A \rightarrow C$, with the morphisms in each category defined as expected. Consider the functor $F: \ck^\square \to \ck^{sp}$ that ``forgets'' the top corner of each square. One can check that this is a $\lambda$-accessible functor preserving $\lambda$-presentable objects, and moreover its image $P$ is closed under subobjects, hence is equal to its powerful image. This image is, by definition of $F$, the category of spans that \emph{can} be amalgamated. By Theorem \ref{mp-thm}, $P$ is $\kappa$-accessible. Now any span $S$ of $\ck^{sp}$ is a $\kappa$-directed colimit of $\kappa$-presentable objects, and each of these objects can by assumption be amalgamated, hence are in $P$. Because $P$ has $\kappa$-directed colimits, $S$ is also in $P$. This shows that any span can be amalgamated, hence that $\ck$ has amalgamation.
\end{proof}

As a final application, we mention how Shelah's singular compactness theorem \cite{shelah-singular} can be restated as a theorem about image of accessible functors. One of the most well known statement of the singular compactness theorem is that an abelian group of singular cardinality all of whose subgroups of lower cardinality are free is itself free. The proof can be axiomatized to apply to other kinds of objects than groups: modules, well-coloring in graphs, transversals, etc. In \cite{cellular-singular-jpaa}, Beke and Rosický state the following general form:

\begin{thm}\label{sing-compact}
  Let $F: \ck \to \cl$ be an $\aleph_0$-accessible functor. Assume that $F$-structures extend along morphisms. Let $A \in \cl$ be an object whose size is a singular cardinal. If all subobjects of $A$ of lower size are in the image of $F$, then $A$ itself is in the image of $F$.
\end{thm}

Here, we say that \emph{$F$-structure extend along morphisms} if for any $\ck$-morphism $g: A \to B$, any object $A'$ of $\ck$, and any isomorphism $i: F A' \cong F A$, there exists $f: A' \to B'$ and an isomorphism $j: F B' \cong F B$ such that the following diagram commutes:

  $$
  \xymatrix@=3pc{
    F A' \ar[d]_i\ar@{.>}[r]_{F f} & F B' \ar@{.>}[d]_{j} \\
    F A  \ar[r]_{F g} & F B 
    }
  $$

  This can be thought of as a generalization of the Steinitz exchange property in vector spaces and fields.

  \begin{example}
    Let $F: \Set_{mono} \to \Ab$ be the restriction of the free abelian group functor to $\Set_{mono}$. This is an $\aleph_0$-accessible functor and $F$-structures extend along morphisms (we can rename to the case where $g$ is the inclusion of $A$ into a superset $B$; then if $i: F A' \cong F A$, we know that both free groups have the same number of generators, and one can add $|B \backslash A|$-many elements to $A'$ to obtain a superset $B'$ so that $j$ extends $i$ to an isomorphism of $FB'$ with $FB$). Thus we recover Shelah's original application of the singular compactness theorem: if $A$ is an abelian group of singular cardinality, all of whose subobjects of lower cardinality are free, then all these subobjects lie in the image of $F$, hence by Theorem \ref{sing-compact} $A$ must also be in this image, i.e.\ be free.
  \end{example}

  \subsection{Vopěnka's principle} is a large cardinal axiom, whose consistency strength is between huge and extendible. A thorough introduction to the category-theoretic implications of Vopěnka's principle is in \cite[\S6]{adamek-rosicky}. Students of set theory may be familiar with Vopěnka's principle as the statement that in any proper class of structures in the same vocabulary, there exists an elementary embedding between two distinct members of the class (see \cite[p.~380]{jechbook}). Another logical characterization of Vopěnka's principle, due to Stavi, is that every logic has a Löwenheim-Skolem-Tarski number (see \cite[Theorem 6]{magidor-van2011}). A purely combinatorial characterization of Vopěnka's principle --- and the one that Vopěnka first stated --- is that there are no rigid proper classes of graphs. Stated in category-theoretic terms, no large full subcategory of the category of graphs is rigid, where a category is \emph{rigid} if all of its morphisms are identities.

  In fact, there is a more general category-theoretic formulation of Vopěnka's principle: any locally presentable category can be fully embedded into the category of graphs \cite[2.65]{adamek-rosicky}. Thus Vopěnka's principle is equivalent to the statement that no locally presentable category has a large rigid full subcategory. Even more strongly (because by Theorem \ref{acc-to-aec} any accessible category can be fully embedded into $\Str (\tau)$, a locally presentable category), no large full subcategory of an \emph{accessible} category can be rigid. Thus if $C$ is a proper class of objects from an accessible category, Vopěnka's principle tells us there must be a morphism between two distinct objects of $C$ (the first logical version mentioned above is the special case of the accessible category of $\tau$-structures with elementary embeddings).

  The following criteria makes it easy to check that a category is accessible \cite[6.9, 6.17]{adamek-rosicky}. It can be seen as a category-theoretic version of the fact that every logic has a Löwenheim-Skolem-Tarski number. 

  \begin{thm}
    Assume Vopěnka's principle. If $\cl$ is a full subcategory of an accessible category $\ck$ and $\cl$ has $\mu$-directed colimits for some $\mu$, then the inclusion $\cl \to \ck$ preserves $\lambda$-directed colimits for some $\lambda$ and $\cl$ is accessible.
  \end{thm}

  This result has been used to prove existence of certain homotopy localizations \cite{localization-vopenka, left-det-rt}. There is also an accessible functor characterization of Vopěnka's principle: every subfunctor of an accessible functor is accessible \cite[6.31]{adamek-rosicky}. 

\subsection{Weak diamond and amalgamation bases}
In complicated categories where pushouts are not available (e.g.\ when all morphisms are monos), the amalgamation property can play a key role. In this subsection, we look at a set-theoretic way to obtain it in a general class of concrete categories. More generally, we will look at \emph{amalgamation bases}: objects $A$ such that any span with base $A$ can be completed to a commutative square (see Definition \ref{ap-def}). We will study them in AECs (Definition \ref{infty-aec-def}), although several of the concepts are category-theoretic and the result can be generalized to $\mu$-AECs \cite[6.12]{mu-aec-jpaa}, or other kinds of concrete categories \cite[5.8]{multidim-v2}.

The key set-theoretic component is the weak diamond, a combinatorial principle introduced by Devlin and Shelah \cite{dvsh65}. We will use it in the following form (see \cite[6.1,7.1]{dvsh65}):

\begin{thm}[Devlin-Shelah]\label{wd-thm}
  Let $\lambda$ be an infinite cardinal and let $\seq{f_\eta : \lambda \to \lambda \mid \eta \in \fct{\lambda}{2}}$ be a sequence of functions. If there exists $\theta < \lambda$ such that $2^{\theta} = 2^{<\lambda} < 2^\lambda$, then ($\lambda$ is regular uncountable and) there exists $\eta \in \fct{\lambda}{2}$ such that the set $S_\eta$ defined below is stationary\footnote{A subset of a regular uncountable cardinal $\lambda$ is called \emph{stationary} if it intersects every closed unbounded subset of $\lambda$. If we think of closed unbounded subsets as having full measure, being stationary means having positive measure.}.

  $$
  S_\eta := \{\delta < \lambda \mid \exists \nu \in \fct{\lambda}{2} : \eta \rest \delta = \nu \rest \delta, \eta \rest (\delta + 1) \neq \nu \rest (\delta + 1), \text{ and }f_\eta \rest \delta = f_\nu \rest \delta\}
  $$

\end{thm}
\begin{remark}\label{wd-rmk}
  Given any fixed cardinal $\theta$, there is a unique cardinal $\lambda$ such that $2^{\theta} = 2^{<\lambda} < 2^\lambda$, which can equivalently be described as the minimal cardinal $\lambda$ such that $2^{\theta} < 2^{\lambda}$. Note that $\theta < \lambda \le 2^{\theta}$. If $\theta$ is finite, $\lambda = \theta^+ = \theta + 1$ but if $\theta$ is infinite, $\lambda$ is uncountable and moreover regular (because of the formula $2^\lambda = \left(2^{<\lambda}\right)^{\cf{\lambda}}$, see \cite[5.16(iii)]{jechbook}). If the generalized continuum hypothesis\footnote{The statement that $2^\lambda = \lambda^+$ for any infinite cardinal $\lambda$.} (GCH) holds, then $\lambda = \theta^+$, but in general it could be that $\lambda = \theta^+$ even if GCH fails (e.g.\ if $\theta = \aleph_0$, $2^{\aleph_0} = \aleph_2$, and $2^{\aleph_1} = \aleph_3$). Nevertheless, it is also consistent with the axioms of set theory that $\lambda > \theta^+$. We will say that the \emph{weak generalized continuum hypothesis (WGCH)} holds if $2^\theta < 2^{\theta^+}$ for all infinite cardinals $\theta$. In this case, the hypothesis of the previous theorem holds exactly when $\lambda$ is an infinite successor cardinal. Note also that, when $\lambda$ is regular uncountable, the conclusion of the Devlin-Shelah theorem implies that $2^{\lambda_0} < 2^{\lambda}$ for all $\lambda_0 < \lambda$. Indeed, given $F: \fct{\lambda}{2} \to \fct{\lambda_0}{2}$, we can let $f_\eta (\alpha)$ be $F (\eta) (\alpha)$ if $\alpha < \lambda_0$, or $0$ otherwise. Fixing the given $\eta \in \fct{\lambda}{2}$, and $\delta$ in $S_\eta$ bigger than $\lambda_0$, we obtain $\nu \neq \eta$ with $F (\nu) = F (\eta)$, so $F$ is not injective \cite[Appendix, 1.B(3)]{proper-and-imp}.
\end{remark}

The essence of the conclusion of the Devlin-Shelah Theorem (\ref{wd-thm}) is that if we think of the $f_\eta$'s as being indexed by branches of a binary splitting tree of height $\lambda$, then there exists two branches (i.e.\ some $\eta$ and some $\nu$) that split at some big height $\delta < \lambda$ and where moreover the corresponding functions are equal up to $\delta$. If we think of the $f_\eta$'s as embedding structures into a common codomain, them being equal up to $\delta$ will mean that a certain diagram commutes.

To state the promised application to amalgamation bases in AECs, we first give some terminology: for $\K$ an AEC and $\lambda$ a cardinal, we write $\K_\lambda$ (resp.\ $\K_{<\lambda}$) for the class of objects in $\K$ of cardinality $\lambda$ (resp.\ strictly less than $\lambda$). Of course, we identify it with the corresponding category. An object $N \in \K_\lambda$ is called \emph{universal} if any $M \in \K_\lambda$ embeds into $N$. We will show that, if $\lambda$ satisfies the hypotheses of the Devlin-Shelah theorem and $\K$ has a universal model in $\K_\lambda$, then amalgamation bases are cofinal in $\K_{<\lambda}$. The result is due to Shelah \cite{sh88}. The proof proceeds by contradiction: if amalgamation bases are not cofinal, we can build a tree of failure, and embed each branch of this tree into the universal model. Applying the weak diamond to this tree will yield enough commutativity to get that the tree has a lot of amalgamation bases.

One reason Theorem \ref{ap-thm} is interesting is that it turns a one-dimensional property in $\lambda$ (existence of a universal model) into a two-dimensional property below $\lambda$ (existence of amalgamation bases). There is in fact a higher-dimensional generalization \cite[11.16]{multidim-v2}, which is much harder to state (but see Section \ref{higher-dim-sec}).

\begin{thm}[{\cite[I.3.8]{shelahaecbook}}]\label{ap-thm}
  Let $\K$ be an AEC and let $\lambda > \LS (\K)$ be such that there exists $\theta < \lambda$ with $2^{\theta} = 2^{<\lambda} < 2^\lambda$. If there exists a universal model in $\K_\lambda$, then for any $M \in \K_{<\lambda}$, there exists $N \in \K_{<\lambda}$ such that $M \lea N$ and $N$ is an amalgamation base in $\K_{<\lambda}$.
\end{thm}
\begin{proof}
  Suppose not. First recall that $\lambda$ is regular uncountable (Remark \ref{wd-rmk}). Fix $M \in \K_{<\lambda}$ with no amalgamation base extending it in $\K_{<\lambda}$. Because $\K$ is isomorphism-closed in $\Str (\tau (\K))$, we may do some renaming to assume without loss of generality that $U M \subseteq \lambda$. We build $\seq{M_\eta \mid \eta \in \fct{<\lambda}{2}}$ such that for all $\beta < \lambda$ and all $\eta \in \fct{\beta}{2}$:

  \begin{enumerate}
  \item $M_{\seq{}} = M$.    
  \item\label{req-2} $M_\eta \in \K_{<\lambda}$.
  \item\label{req-3} $U M_\eta \subseteq \lambda$.
  \item $M_{\eta \rest \alpha} \lea M_\eta$ for all $\alpha < \beta$.
  \item\label{req-5} If $\beta$ is limit, $M_\eta = \bigcup_{\alpha < \beta} M_\alpha$.
  \item\label{req-6} The span $M_{\eta \smallfrown 0} \leftarrow M_\eta \rightarrow M_{\eta \smallfrown 1}$ \emph{cannot} be completed to a commutative square (where the maps are inclusion embeddings).
  \end{enumerate}

  This is possible: the construction proceeds by transfinite induction on the length of $\eta$. The base case is given, at successors we use that $M_\eta$ cannot be an amalgamation base in $\K_{<\lambda}$ by assumption (and do some renaming to implement (\ref{req-3})). At limit stages, we take unions (and use the axioms of AECs).

  This is enough: for each $\eta \in \fct{\lambda}{2}$, define $M_\eta := \bigcup_{\alpha < \lambda} M_{\eta \rest \alpha}$. The axioms of AECs imply, of course, that $M_{\eta \rest \alpha} \lea M_\eta$ for all $\alpha < \lambda$. Moreover, $M_\eta \in \K_\lambda$. Indeed, $M_{\eta \rest \alpha} \neq M_{\eta \rest (\alpha + 1)}$, as otherwise we would trivially have been able to amalgamate the span $M_{\eta \rest \alpha \smallfrown 0} \leftarrow M_{\eta \rest \alpha} \rightarrow M_{\eta \rest \alpha \smallfrown 1}$. Finally, requirement (\ref{req-3}) ensures that $UM_\eta \subseteq \lambda$. Fix a universal model $N \in \K_\lambda$. By doing some renaming again, we can assume without loss of generality that $U N \subseteq \lambda$. For each $\eta \in \fct{\lambda}{2}$, fix a $\K$-embedding $g_\eta: M_\eta \rightarrow N$. Our goal will be to get a contradiction to requirement (\ref{req-6}). We are not there yet, because even if $\eta \rest \delta = \nu \rest \delta$, the maps $g_\eta$ and $g_{\nu}$ will not necessarily agree on $M_{\eta \rest \delta}$. This is where the weak diamond will come in.

  Define $f_\eta : \lambda \to \lambda$ by $f_\eta (\alpha) = g_\eta (\alpha)$ if $\alpha \in U M_\eta$, and $f_\eta (\alpha) = 0$ if $\alpha \notin U M_\eta$. We are now in the setup of the Devlin-Shelah theorem: fix $\eta \in \fct{\lambda}{2}$ such that the set $S_\eta$ defined there is stationary. Consider the set $C := \{\delta < \lambda \mid U M_{\eta \rest \delta} \subseteq \delta\}$. The set $C$ is closed (because of requirement (\ref{req-5})) and unbounded. To see the latter, we run a standard ``catching your tail'' argument\footnote{This is in fact an instance of Theorem \ref{refl-thm} in the appendix, see Example \ref{forcing-ex}.}: fix $\alpha < \lambda$. We inductively build an increasing sequence of ordinals $\seq{\alpha_n : n < \omega}$ as follows: take $\alpha_0 = \alpha$, and given $\alpha_n$, we know $U M_{\eta \rest \alpha_n} \subseteq \lambda$ (requirement (\ref{req-3})), and $|U M_{\eta \rest \alpha_n}| < \lambda$ (requirement (\ref{req-2})), so use regularity of $\lambda$ to pick $\alpha_{n + 1} \in [\alpha, \lambda)$ with $U M_{\eta \rest \alpha_n} \subseteq \alpha_{n + 1}$. At the end, we let $\delta := \sup_{n < \omega} \alpha_n$. The construction, together with requirement (\ref{req-5}), implies that $\delta \ge \alpha$ and $\delta \in C$.

  Because $C$ is closed unbounded and $S_\eta$ is stationary, we can pick $\delta \in C \cap S_\eta$. By definition of $S_\eta$, there exists $\nu \in \fct{\lambda}{2}$ such that $\eta \rest \delta = \nu \rest \delta$, $\eta (\delta) \neq \nu (\delta)$, and $f_\eta \rest \delta = f_\nu \rest \delta$. Let $\rho := \eta \rest \delta$. By definition of $C$, $f_\eta \rest M_\rho = f_\nu \rest M_\rho$. This implies that $f_\eta$, $f_\nu$, and $N$ witness the span $M_{\eta \rest (\delta + 1)} \leftarrow M_{\rho} \rightarrow M_{\nu \rest (\delta + 1)}$ can be completed to a commutative square, contradicting requirement (\ref{req-6}).
\end{proof}

If we know weak GCH holds, and moreover we know that $\K$ has a single object (up to isomorphism) in two successive cardinalities\footnote{Model-theorists call this property \emph{categoricity}, see \ref{categ-def}.}, the theorem simplifies and we get the amalgamation property locally: 

\begin{cor}
  Let $\K$ be an AEC, and let $\mu \ge \LS (\K)$ be such that $2^{\mu} < 2^{\mu^+}$. If $\K$ has a single model (up to isomorphism) in cardinalities $\mu$ and $\mu^+$, then $\K_\mu$ has the amalgamation property.
\end{cor}
\begin{proof}
  By its uniqueness, the object of cardinality $\mu^+$ must be universal. Applying Theorem \ref{ap-thm} with $\lambda = \mu^+$, we get that there exists an amalgamation base in $\K_\mu$, and this amalgamation base must be isomorphic to any other object of cardinality $\mu$, hence $\K_\mu$ has amalgamation.
\end{proof}

\section{Generalized pushouts and stable independence}\label{indep-sec}

Many of the ``classical'' categories of mainstream mathematics are bicomplete: they have all limits and colimits. However, problems will occur if we want to study such categories set-theoretically, and specifically if we want to restrict ourselves to certain classes of monomorphisms. It is clear quotients (coequalizers) will be lost, but one will also lose pushouts: any category with pushouts, all morphisms monos, and an initial object must be thin, hence essentially just a poset \cite[3.30(3)]{indep-categ-advances}. Still, it is natural to ask for approximations to pushouts. The results of this section are a survey of the work of Lieberman, Rosický, and the author on this question \cite{indep-categ-advances, more-indep-v2}. 

For the purpose of the discussion to follow, let's briefly repeat that for a given a diagram $D: I \to \ck$, a  \emph{cocone} for that diagram is an object $A$ together with maps $(D_i \xrightarrow{f_i} A)_{i \in I}$ commuting with the diagram. One can look at the category $\ck_D$ of cocones for $D$, where the morphisms are defined as expected. The colimit of a diagram $D$ is then simply an initial object (one that has a unique morphism to every other object) in the category of cocones. In case $D$ is a span $B \xleftarrow{f} A \xrightarrow{g} C$, a cocone is simply an amalgam of this span, and an initial object in the category $\ck_D$ of cocones would be a pushout. The amalgamation property (Definition \ref{ap-def}) simply says that $\ck_D$ is non-empty, i.e.\ that there is \emph{some} cocone, maybe satisfying no universal property whatsoever. A much stronger approximation is the existence of \emph{weak pushouts}: a weak pushout of a given span $D$ is a cocone that is weakly initial in the category of cocones for $D$: there is a morphism to every other cocone, but that morphism is not required to be unique. In categories where all morphisms are monos, weak pushouts are still too strong of a requirement:

\begin{example}
  Consider the category $\Set_{mono}$ of sets with injections. Consider the inclusion of $A = \emptyset$ into $B = \{0,1 \}$ and $C = \{0, 2\}$. Let $D = \{0,1,2\}$. Then $D$, together with the corresponding inclusions, is a cocone/amalgam for the span: $B \leftarrow A \rightarrow C$. On the other hand, consider $D' = \{1,2,3,4\}$ and $f_1: B \to D'$, $f_2: C \to D'$ defined by $f_1 (0) = 3$, $f_1 (1) = 1$, $f_2 (0) = 4$, $f_2 (2) = 2$. $(f_1, f_2)$ is also an amalgam of $B \leftarrow A \rightarrow C$, but $D$ has no morphisms to $D'$ (in the category of cocones for the span $B \leftarrow A \rightarrow C$), and by cardinality considerations $D'$ has no morphisms to $D$ either. Thus $\Set_{mono}$ does not have weak pushouts.
\end{example}

What is happening in the example is that we had two choices for amalgamating $B$ and $C$: either sending $0$ to the same place, or sending it two different elements. These two choices are then incompatible, in the sense that no amalgam of one type will ever have a morphism into an amalgam of the other type (in the appropriate category of amalgams of a fixed span). In other words, the category of amalgams is not connected. Let us make this explicit, first in complete generality,  then for the specific case of amalgams:

\begin{defin}
  Two objects $A$ and $B$ are called \emph{comparable}\footnote{The terminology comes from posets.} if either $\Hom (A, B) \neq \emptyset$ or $\Hom (B, A) \neq \emptyset$. We say that $A$ and $B$ are \emph{connected} if there exists $A = A_0, A_1, \ldots, A_n = B$ such that $A_{i}$ and $A_{i + 1}$ are comparable for all $i < n$. This is an equivalence relation, and the equivalence class is called a \emph{connected component} of the category. The category is called \emph{connected} if any two of its objects are connected. We say that $A$ and $B$ are \emph{jointly connected} if there exists an object $C$ with morphisms $A \to C$, $B \to C$.
\end{defin}

Of course, the connected components of a category are exactly the connected component of the undirected graph whose vertices are objects and where there is an edge between $A$ and $B$ exactly when $A$ and $B$ are comparable. Assuming amalgamation, joint connectedness coincides with connectedness:

\begin{lem}\label{compat-lem}
  In a category with the amalgamation property, two objects are connected if and only if they are jointly connected.
\end{lem}
\begin{proof}
  Let $\ck$ be a category with the amalgamation property. First note that if two objects $A$ and $B$ are comparable, then they are jointly connected (for example, if $A \xrightarrow{f} B$ then $A \xrightarrow{f} B = C \xleftarrow{\id_B} B$ witness the joint connectivity). Also, if $A$ and $B$ are jointly connected, as witnessed by $A \to C$, $B \to C$, then $A$ and $C$ are comparable, and $C$ and $B$ are comparable, so $A$ and $B$ are connected. We now show that being jointly connected is an equivalence relation. This will be enough because being connected is the smallest equivalence relation extending comparability. So assume that $A_0$, $A_1$ are jointly connected, as witnessed by $A_0 \rightarrow B \leftarrow A_1$, and $A_1$, $A_2$ are jointly connected, as witnessed by $A_1 \rightarrow C \leftarrow A_2$. Amalgamate $B$ and $C$ over $A_1$, forming the diagram below:

  $$
  
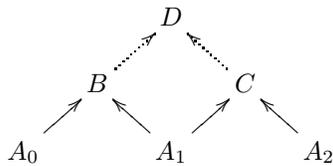
\xymatrix@=1pc{
    & & D & & \\
    & B \ar@{.>}[ru] & & C \ar@{.>}[lu] & \\
    A_0 \ar[ru] & & A_1 \ar[lu] \ar[ru] & & A_2 \ar[lu]
  }
  $$

  Then $D$ and the expected compositions witness that $A_0$ and $A_2$ are jointly connected. 
\end{proof}

\begin{defin}
  Two amalgams $(B \xrightarrow{f^a} D^a, C \xrightarrow{g^a} D^a)$ and $(B \xrightarrow{f^b} D^b, C \xrightarrow{g^b} D^b)$ of a span $B \leftarrow A \rightarrow C$ are \emph{jointly connected} if they are jointly connected in the category of cocones for the appropriate span. Explicitly, there exists $D$ and morphisms into $D$ making the following diagram commute:

          
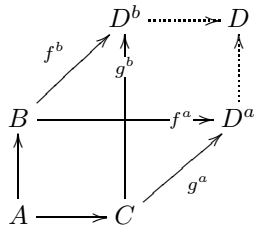
\[
        \xymatrix{ & D^b \ar@{.>}[r] & D \\
    B \ar[ru]^{f^b} \ar[rr]|>>>>>{f^a} & & D^a \ar@{.>}[u] \\
    A \ar[u] \ar[r] & C \ar[uu]|>>>>>{g^b}  \ar[ur]_{g^a} & \\
  }
        \]

        We say that two amalgams $D^a$ and $D^b$ are \emph{connected} if they are connected in the appropriate category of cocones, i.e.\ there exists a chain of amalgams $D^0, D^1, \ldots, D^n$ (along with respective morphisms) such that $D^0 = D^a$, $D^n = D^b$, and $D^{i}$ is jointly connected to $D^{i + 1}$ for all $i < n$.
\end{defin}

Note that if $\ck$ has the amalgamation property, then the category of cocones over a diagram in $\ck$ also has the amalgamation property, so in this case connectedness already implies joint connectedness. The replacement for pushouts we are looking for will therefore consist in a choice of connected component for each span. Eventually, we may want to look for a weakly initial object for this specific component class (called a \emph{prime} object by model theorists), but this seems to be too strong a requirement to start with: we want to prove it, not assume it.\footnote{For example, the proofs of existence and uniqueness of differential closure in differentially closed fields (Example \ref{indep-ex-0}(\ref{diff-field-ex})) rely on properties of the independence notion.} Thus we instead impose a transitivity conditions on the choice of connected component that make the resulting squares into the morphisms of an arrow category (and also holds of pushouts). Still by analogy with pushouts, and for reasons to be discussed later, we require this new category to be accessible and call the result a stable independence notion.

\begin{defin}[{\cite[3.24]{indep-categ-advances}}]\label{indep-def}
      A \emph{stable independence notion} in a category $\ck$ is a class of squares (called \emph{independent squares} and marked here with the anchor symbol $\nf$) such that:

    \begin{enumerate}
    \item Independent squares are closed under connectedness: if one amalgam of a given span is independent, then all the connected amalgams are also independent.
    \item Existence: any span can be amalgamated to an independent square.
    \item Uniqueness: any two independent amalgams of the same span are connected.
    \item Transitivity:

      $$  
      \xymatrix@=1pc{
        B \ar[r]\ar@{}[dr]|{\nf} & D\ar@{}[dr]|{\nf} \ar[r] & F \ar@{}[dr]|{\Rightarrow} & B \ar[r]\ar@{}[dr]|{\nf} & F \\
        A \ar [u] \ar [r] & C \ar[u] \ar[r] & E \ar[u] & A \ar [u] \ar [r] & E \ar[u]
      }
      $$
    \item Symmetry: ``swapping the ears'' $B$ and $C$ preserves independence.
    \item Accessibility: the arrow category whose objects are morphisms of $\ck$ and whose morphisms are independent squares is accessible. We let $\ck_{\smallnf}$ denote this category.
    \end{enumerate}
\end{defin}
\begin{remark}
  The existence property implies the amalgamation property. Hence any two connected amalgams are, in fact, jointly connected. Note also that any two amalgams of a span which contains an isomorphism will be connected, hence (by the existence property) independent squares \cite[3.12]{indep-categ-advances}. In particular, the identity morphism in the arrow category $\ck^2$ will indeed be an independent square, hence a morphism of $\ck_{\smallnf}$. The transitivity property ensures that $\ck_{\smallnf}$ is closed under composition, and hence is indeed a category.
\end{remark}

Most of the examples below are listed in \cite[3.31]{indep-categ-advances}, though of course many are trivial or date back to the beginning of stable first-order theories. See there for more references.

\begin{example}\label{indep-ex-0} \
  \begin{enumerate}
  \item In an accessible category with weak pushouts, there is a stable independence notion given by all commutative squares. This holds more generally in any accessible category with the amalgamation property where any two amalgams are always connected, for example in accessible categories with a terminal object.
  \item The category $\Set_{mono}$ has a stable independence notion: identifying the morphisms with inclusions, given $A \subseteq B \subseteq D$, $A \subseteq C \subseteq D$, we say that the resulting square $A, B, C, D$ is independent exactly when $B \cap C = A$. That is, the amalgam must be disjoint. In general, the ears of an independent squares ``do not interact'': independent squares are a notion of ``free'' amalgam.
  \item The category of vector spaces over a fixed field, with morphisms the injective linear transformations, has a stable independence notion, again given by disjoint amalgamation. This holds more generally for the category of modules over a fixed ring.
  \item The category of all algebraically closed fields of a fixed characteristic (with morphisms field homomorphisms) has a stable independence notion, essentially given by algebraic independence: again identifying the morphisms with inclusions, if we have a square of fields $F_0 \subseteq F_1 \subseteq F_3$, $F_0 \subseteq F_2 \subseteq F_3$, we say it is independent if for any subset $A \subseteq F_1$ and any $b \in F_1$, if $b$ is in the algebraic closure (computed in $F_3$) of $A \cup F_2$, then it is in the algebraic closure of $A \cup F_0$. Here, it does \emph{not} suffice to require that $F_1 \cap F_2 = F_0$ (because the pregeometry induced by algebraic closure is not modular).
  \item\label{diff-field-ex} A \emph{differential field} is a field together with an operator $D$ that preserves sums and satisfies Leibnitz' law: $D fg = g D f + f D g$. A \emph{differentially closed field} is, roughly, a differential field in which every system of linear differential equations that could possibly have a solution has a solution. Model-theoretic methods were used to obtain the first proof that every differential field of characteristic zero has, in some sense, a differential closure. Uniqueness of that differential closure was proven by Shelah, using the fact that the category of differentially closed fields of characteristic zero has a stable independence notion. See \cite{sacks-diff} for a short overview.
  \item\label{indep-stable-ex-0} Generalizing the last four examples, let $T$ be a stable first-order theory ($T$ is \emph{stable} if it does \emph{not} have the order property, and $T$ has the \emph{order property} if there exists a formula $\phi (\bx, \by)$, a model $M$ of $T$, and a sequence $\seq{a_i : i < \omega}$ of elements in $M$ such that $M \models \phi[\ba_i, \ba_j]$ if and only if $i < j$, see Definition \ref{op-def}). Consider the category $\ck = \Elem (T)$ of models of $T$ with elementary embeddings. Then $\ck$ has a stable independence notion. The proof (originally due to Shelah), is not trivial, see Appendix \ref{fo-sec} for a short exposition. The definition of an independent square mirrors that of fields (the reader should think of the formula $\psi$ in the next sentence as a polynomial): a square $M_0 \preceq M_1 \preceq M_3$, $M_0 \preceq M_2 \preceq M_3$ is \emph{independent} exactly when for any finite sequences $\ba$ from $M_1$ and $\bb$ from $M_2$, and any formula $\psi (\bx, \by)$, if $M_3 \models \psi[\ba, \bb]$, then there exists a sequence $\bb_0$ in $M_0$ so that $M_3 \models \psi[\ba, \bb_0]$. In fact, independent squares are the same as what model theorists call nonforking squares. Thus we have recovered the important model-theoretic notion of forking from simple category-theoretic considerations! In fact, the original motivation behind the definition of stable independence was to generalize forking.
  \item Conversely, if $T$ is not stable then $\Elem (T)$ does \emph{not} have a stable independence notion \cite[9.9]{indep-categ-advances}. 
  \item The previous item implies, in particular, that the category of graphs, with morphisms the full subgraph embeddings does \emph{not} have a stable independence notion. Another proof will be given in Remark \ref{canon-rmk}. The category $\Lin$ of linear orders with order-preserving maps similarly does not have a stable independence notion.
  \item The category of graphs with morphisms the (not necessarily full) subgraph embeddings \emph{does} have a stable independence notion: two graphs are independent over a base graph (inside an ambient graph) if all the cross-edges between the two are inside the base graph. See Example \ref{indep-ex}(\ref{indep-gr-2-ex}).
  \item \cite[4.8(5)]{indep-categ-advances} The category of Hilbert spaces with isometries has a stable independence notion, given by pullback squares. These correspond roughly to squares where everything is ``as orthogonal as possible''.
  \item We will give later general constructions of a stable independence notion, giving many other nontrivial examples. For example, for a fixed ring, the category of all flat modules with morphisms the flat monomorphisms has a stable independence notion (Example \ref{indep-ex}(\ref{indep-ex-mod})).
  \end{enumerate}
\end{example}

For model theorists, we note that in concrete cases (e.g.\ inside an $\infty$-AEC $\K$), the independence notion can be extended to a relation of the form $\nfs{M}{A}{B}{N}$, where $M \lea N$ and $A, B \subseteq U N$, satisfying generalizations of the properties of a stable independence notion. The accessibility axiom can then be shown to be equivalent to the conjunction of the two classical properties of forking: the \emph{witness property} (failure to be independent is witnessed by small subsets of the ears) and the \emph{local character property} (every type is independent over a small set). The proof of this fact is essentially a generalization of the proof of Theorem \ref{aec-acc}, see \cite[8.14]{indep-categ-advances}. We deduce, in particular, that an $\infty$-AEC has quite a bit of structure when it has a stable independence notion (e.g.\ it is tame and stable \cite[8.16]{indep-categ-advances}).

One of the most important fact about stable independence (and a strong justification for the definition) is that it leads to a canonical notion: if a category has a stable independence notion, then under very reasonable conditions it can have only one. For first-order forking, this is well known \cite{hh84}, but in fact it holds even in the very general categorical setup. The proof for AECs in \cite{bgkv-apal} was adapted to $\infty$-AECs with chain bounds in \cite[9.1]{indep-categ-advances}, and finally to any category with chain bounds in \cite[A.6]{more-indep-v2}. We say a category has \emph{chain bounds} if any ordinal-indexed chain has a cocone (this holds, of course, anytime the category has directed colimits).

\begin{thm}[The canonicity theorem, {\cite[A.6]{more-indep-v2}}]\label{canon-thm}
  A category with chain bounds has at most one stable independence notion. In fact, given a stable independence notion $\nf$, any other relation satisfying all the axioms of stable independence notion except perhaps accessibility will have to be $\nf$.
\end{thm}
\begin{proof}[Proof idea]
  First, existence of a stable independence notion implies that the category itself is accessible \cite[3.27]{indep-categ-advances}. Let $\nf^1$ be a stable independence notion and let $\nf^2$ satisfy all the axioms of stable independence, except perhaps accessibility. Given a span $B \leftarrow A \rightarrow C$, it suffices to show that it has an amalgam which is independent in the sense of both $\nf^1$ and $\nf^2$ (then uniqueness and invariance under connectedness show that $\nf^1$ and $\nf^2$ must coincide for any amalgam of this span). The idea of the construction is perhaps best described by the following property of a vector space: given any sequence $I$ of vectors and any other vector $a$, there exists a finite subset $I_0$ of $I$ such that $(I - I_0) \cup \{a\}$ is independent. Thus the idea is to first complete the span $B \leftarrow A \rightarrow C$ to a $\nf^2$-independent square, then build many $\nf^2$-independent ``copies'' of that square. Analogously to the fact mentioned for vector spaces, all except a few of these copies must also be $\nf^1$-independent. 
\end{proof}
\begin{remark}\label{canon-rmk}
  The canonicity theorem provides us with a tool to prove the \emph{non-existence} of stable independence notions. Indeed, it suffices for this purpose to find two different notions satisfying all the axioms of stable independence except perhaps for accessibility. For example, the category of graph with full subgraph embeddings does not have a stable independence notion. Indeed, the relations ``all cross-edges are contained inside the base'' and ``all cross-edges outside those in the base are present'' satisfy all the axioms of stable independence except for accessibility \cite[3.31(6)]{indep-categ-advances}.
\end{remark}

We now survey two different general constructions of a stable independence notion. In both cases, the hypotheses are provably optimal in the sense that the existence of stable independence implies them. The first construction works in any $\infty$-AEC with chain bounds, but uses large cardinals and only obtains an independence notion on a cofinal subclass (consisting of ``sufficiently homogeneous'' objects).

\begin{thm}[{\cite[10.3]{indep-categ-advances}}]\label{vopenka-constr}
  Assume Vopěnka's principle. Let $\K$ be an $\infty$-AEC. The following are equivalent:

  \begin{enumerate}
  \item $\K$ does not have the order property\footnote{This is defined by generalizing the definition given in Example \ref{indep-ex-0}(\ref{indep-stable-ex-0}) and in \ref{op-def}: there is no long sequence that can be ordered by a type in the sense to be given in Section \ref{aec-sec}. See \cite[9.7]{indep-categ-advances}.}.
  \item $\K$ has a cofinal subclass of ``sufficiently homogeneous'' objects which has a stable independence notion.
  \end{enumerate}
\end{thm}
\begin{proof}[Proof ideas]
  If $\K$ has the order property, as witnessed by a long sequence $\seq{\ba_i : i < \lambda}$ ($\lambda$ here is a big regular cardinal), we can use the accessibility and uniqueness properties of stable independence to get a subsequence $\seq{\ba_i : i \in S}$ with $|S| = \lambda$ that is independent and indiscernible (essentially, this means that the sequence is ``very homogeneous'' -- it looks like a sequence of mutually transcendental elements in an algebraically closed fields). The symmetry axiom can then be use to show that the elements of the subsequence can be permuted without impacting their properties. In particular, the sequence cannot serve as a witness for the order property, contradiction.

  Going from no order property to stable independence, the very rough idea is to imitate the standard construction given Appendix \ref{fo-sec}, but replace $\aleph_0$ by a sufficiently-big strongly compact, and look only at the locally $\kappa$-homogeneous models (see \cite[7.3]{indep-categ-advances}). Vopěnka's principle is used to prove that this subclass is an $\infty$-AEC, and also provides enough compactness to prove the existence property.
\end{proof}

For the second method, we start with a ``classical'' category $\ck$: a locally presentable category. We single out a certain class of morphisms $\Mm$ (usually some class of nice monomorphisms), and we want to build a stable independence notion on the category $\ck_{\Mm}$ obtained by restricting the morphisms in $\ck$ to be those of $\Mm$. To make this setup more precise, we will have to make some assumptions on $\Mm$. Of course, we want at minimum $\Mm$ to contain all isomorphisms and be closed under compositions. It turns out it is very convenient to assume that $\Mm$ satisfies a coherence property, similar to the coherence axiom of $\infty$-AECs: for morphisms $f, g$ of $\ck$, if $gf \in \Mm$ and $g \in \Mm$, then $f \in \Mm$. We will say that $\Mm$ is \emph{coherent}. For technical reasons, we will also want $\Mm$ to be closed under retracts (in the arrow category $\ck^2$). This holds under very mild conditions. For example, if $\Mm$ is closed under compositions, contains all split monos, and is coherent, then it is closed under retracts. Since we want to use pushouts to build our stable independence notion, we will require $\Mm$ to be \emph{closed under pushouts}: a pushout of a morphism in $\Mm$ (not necessarily along a morphism in $\Mm$) should be in $\Mm$. Finally, to perform iterative constructions, we will need $\Mm$ to be \emph{closed under transfinite compositions}.

Under all these conditions, we get that $\ck_{\Mm}$ has a stable independence notion \emph{if and only if} it is accessible and $\Mm$ is cofibrantly generated. Here, $\Mm$ is \emph{cofibrantly generated} if there is a subset $\Xx$ of $\Mm$ (i.e.\ not a proper class) such that closing $\Xx$ under retracts, pushouts, and transfinite compositions gives back $\Mm$. This notion originated in algebraic topology, where the cell complexes are precisely the objects that can be obtained from finitely-many pushouts and composition by starting from inclusions $\partial D^n \to D^n$ of the boundary of the $n$-ball. As we will see, similar notions have been used in homological algebra as well.

\begin{thm}[{\cite[4.9]{more-indep-v2}}]\label{cofib}
  Let $\ck$ be a locally presentable category and let $\Mm$ be a coherent class of morphisms containing all isomorphisms, closed under retracts, transfinite compositions, and pushouts. The following are equivalent:

  \begin{enumerate}
  \item\label{cofib-1} $\ck_{\Mm}$ has a stable independence notion.
  \item\label{cofib-2} $\ck_{\Mm}$ is accessible and $\Mm$ is cofibrantly generated.
  \end{enumerate}
\end{thm}
\begin{proof}[Proof ideas]
  Neither direction is easy. First, we need a candidate definition for a stable independence notion. Let us say that a commutative square with morphisms in $\Mm$ is \emph{effective} if the induced map from the pushout is in $\Mm$. That is, the outer square in the diagram below is effective if all its maps are in $\Mm$ and the induced map $f$ from the pushout $P$ is in $\Mm$.

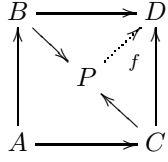
  $$  
  \xymatrix@=1pc{
    B \ar[dr] \ar[rr] & & D \\
    & P \ar@{.>}[ur]_{f} & \\
    A \ar[rr] \ar[uu] & & C \ar[uu] \ar[lu]
      }
  $$

  Without any additional hypotheses, it is not too difficult to show that effective squares will satisfy all the axioms of stable independence, except perhaps for accessibility. Let $\ck_{\Mm, \smallnf}$ be the category whose objects are morphisms in $\Mm$ and morphisms are effective squares. We can also show that $\ck_{\Mm, \smallnf}$ has directed colimits (and they are computed as in $\ck^2$). 

  \begin{itemize}
  \item \underline{(\ref{cofib-1}) implies (\ref{cofib-2})}: Starting from a stable independence notion on $\ck_{\Mm}$, it is not too difficult to show that $\ck_{\Mm}$ must be accessible \cite[3.27]{indep-categ-advances}. Moreover, the stable independence notion must be given by effective squares, by the canonicity theorem (\ref{canon-thm}). Let $\lambda$ be a big-enough regular uncountable cardinal such that all the relevant categories are $\lambda$-accessible. For $\mu \ge \lambda$ regular, let $\Mm_\mu$ denote the class of morphisms in $\Mm$ with $\mu$-presentable domain and codomain (in $\ck$). For a class $\Hh$ of morphisms, let $\cof (\Hh)$, the \emph{cofibrant closure of $\Hh$}, denote the closure of $\Hh$ under retracts, pushouts, and transfinite compositions. We will show that $\Mm = \cof (\Mm_\lambda)$. For this, we prove by induction that for all regular cardinals $\mu$, $\Mm_\mu \subseteq \cof (\Mm_\lambda)$. For $\mu \le \lambda$ this is trivial and if $\mu$ is limit, $\Mm_{\mu} = \bigcup_{\theta < \mu} \Mm_{\theta}$ by Corollary \ref{succ-cor}(\ref{succ-cor-1}). Thus we can assume $\mu = \mu_0^+$. Let $\delta := \cf{\mu_0}$, and fix a morphism $A \xrightarrow{f} B$ in $\Mm_{\mu}$. We write $f$ as the directed colimit of a chain $\seq{f_i : i < \delta}$ in $\ck_{\Mm, \smallnf}$, where each $f_i$ is in $\Mm_{\theta}$ for a regular $\theta \le \mu_0$. Finding such a chain is nontrivial (in accessible categories, objects can be written as directed colimits of directed diagrams of small objects, it is not clear you can find chains like this): assuming for simplicity all morphisms are monos, we use that it is possible in AECs, together with the fact every $\lambda$-accessible category $\cl$ with directed colimits is a reflexive full subcategory of a finitely accessible category (which is given by taking free directed colimits of the $\lambda$-presentable objects of $\cl$), and finitely accessible categories with all morphisms monos are AECs (Theorem \ref{acc-to-aec}).

    Once we have the $f_i$'s, say $A_i \xrightarrow{f_i} B_i$, they are each by the induction hypothesis part of $\cof (\Mm_\lambda)$. It suffices to use them to generate $f$. Let $(A_i \xrightarrow{g_i} A, B_i \xrightarrow{h_i} B)_{i < \delta}$ be the colimit maps. 

  $$  
  \xymatrix@=1pc{
    A   \ar@{}[dr]|{\nf}    \ar[r]^f & B \\
    A_i \ar[u]^{g_i} \ar[r]_{f_i}  & B_i \ar[u]_{h_i}
  }
  $$

  Take the pushout $P_0$ of $f_0$ along $g_0$. We get that $f = p_0 \bar{f}_0$, where $p_0$ is the induced map from the pushout, which is also in $\Mm$ by assumption. We have managed to generate $\bar{f}_0$, and we now repeat what we did for $f$ but for $p_0$ instead (write it as the colimit of an increasing chain of morphisms that are above $f_1$, take the pushout of the first morphism, etc.). In the end, we will have written $f$ as a $\delta$-length transfinite compositions of pushouts of members of $\cof (\Mm_\lambda)$, as desired.
  \item (\ref{cofib-2}) implies (\ref{cofib-1}): As before, fix $\lambda$ an uncountable regular cardinal so that all relevant categories are $\lambda$-accessible and $\Mm = \cof (\Mm_\lambda)$. By using what is called ``good colimits'' \cite[B.1]{fat-small-obj}, we can in fact show that $\Mm$ is generated from $\Mm_\lambda$ by just using pushouts and transfinite compositions (no retracts). Let $\Hh$ be the class of all morphisms that, in $\ck_{\Mm, \smallnf}$ are $\lambda$-directed colimits of morphisms of $\Mm_\lambda$. It suffices to show that $\Hh = \Mm$, and for this it suffices to see that $\Hh$ is closed under pushouts and transfinite compositions. This can readily be done, using the definition of an effective square.
  \end{itemize}
\end{proof}
\begin{remark}
  In Theorem \ref{cofib}, if $\ck$ is $\lambda$-accessible and all morphisms of $\Mm$ are monos, then $\ck_{\Mm}$ will be (equivalent to) a $\lambda$-AEC \cite[3.10]{more-indep-v2}. This is especially interesting when $\lambda = \aleph_0$ (i.e.\ $\ck$ is locally finitely presentable), in which case we get an AEC.
\end{remark}

Compared to Theorem \ref{vopenka-constr}, Theorem \ref{cofib} does not use large cardinal axioms and more importantly gives a stable independence notion on the whole category. The hypotheses seem to hold often-enough in practice (all the examples below and more are in \cite[\S6]{more-indep-v2}):

\begin{example}\label{indep-ex} \
  \begin{enumerate}
  \item Let $\ck$ be the category of graphs (reflexive and symmetric binary relations) with homomorphisms. This is locally finitely presentable. Let $\Mm$ be the full subgraph embeddings. We have seen (Remark \ref{canon-rmk}) that $\ck_{\Mm}$ does not have a stable independence notion. This automatically tells us that $\Mm$ is \emph{not} cofibrantly generated.
  \item\label{indep-gr-2-ex} Let $\ck$ again be the category of graphs with homomorphisms. This time, let $\Mm$ be the subgraph embeddings (corresponding to monomorphisms in $\ck$ -- the full subgraph embeddings correspond to the \emph{regular monomorphisms}: those that are equalizers of some pair of morphism). Then $\Mm$ is cofibrantly generated by $\emptyset \to 1$ and $1 + 1 \to 2$, where $1$ is a vertex, $2$ is an edge, and $1 + 1$ is an empty graph on two vertices. Therefore $\ck_{\Mm}$ has a stable independence notion.
  \item An \emph{orthogonal factorization system} in a category $\ck$ consist of two classes of morphisms $(\Mm, \Nn)$ such that both contain all the isomorphisms, both are closed under composition, and every map factors as $gf$ with $f \in \Mm$ and $g \in \Nn$ in a way that is unique up to unique isomorphism. For example, in the category of sets (epi, mono) is an orthogonal factorization system. A \emph{weak factorization system} is defined in a somewhat similar way, except the uniqueness condition is considerably relaxed. An example is (mono, epi) in the category of sets. It turns out that (in a cocomplete category) the left part of a weak factorization system is always closed under pushouts, retracts, and transfinite compositions. Thus we can call a weak factorization system $(\Mm, \Nn)$ \emph{cofibrantly generated} if $\Mm$ is cofibrantly generated. Thus if $\Mm$ is the left part of a weak factorization system in a locally presentable category $\ck$, $\Mm$ is coherent, and $\ck_{\Mm}$ is accessible, then $\ck_{\Mm}$ has a stable independence notion if and only if the weak factorization system is cofibrantly generated. In fact, the \emph{small object argument} tells us that if $\Xx$ is any set of morphisms in a locally presentable category, then $\cof (\Xx)$ will form the left part of a weak factorization system \cite[1.3]{beke-sheafifiable}. Thus in the setup of Theorem \ref{cofib}, existence of a stable independence notion on $\ck_{\Mm}$ automatically implies that $\Mm$ forms the left part of a weak factorization system.
  \item In algebraic topology, a \emph{model category} is a bicomplete category together with three classes of morphisms, $F$ (fibrations), $C$ (cofibrations), and $W$ (weak equivalences), such that the weak equivalences satisfy the two out of three property (if two of $g, f$ and $gf$ are weak equivalences, so is the third), and both $(C, F \cap W)$ and $(C \cap W, F)$ form weak factorization systems. A model category is called \emph{cofibrantly generated} when both weak factorization systems are cofibrantly generated. The typical example includes the usual fibrations, cofibrations, and weak homotopy equivalences in the category of topological spaces, or the monomorphisms, Kan fibrations, and weak homotopy equivalences in the category of simplicial sets. There are however model category of a more algebraic flavor, including a model category on chain complexes of modules \cite[Chapter 2]{hoveybook}. Thus the previous example also describes a two way connection between model categories and stable independence.
  \item\label{indep-ex-mod} Let $R$ be a (associative and unital) ring and let $\ck$ be the category of $R$-modules with homomorphisms. One can check that $\ck$ is locally finitely presentable. We want to study the modules that are \emph{flat} (i.e.\ directed colimits of free modules -- this is the easiest equivalent definition for the purpose of this discussion). We will do this through the class $\Mm$ of \emph{flat monomorphisms}: monomorphisms whose cokernel is flat. In particular, an inclusion $A \subseteq B$ is flat if and only if $B / A$ is flat. One can check that $\Mm$ contains all isomorphisms, is coherent, and closed under pushouts, retracts, and transfinite compositions. Now the class $\ck_{\Mm}$ is not quite the right category, we really want to study the category $\Ff_{\Mm}$, where $\Ff$ is the full subcategory of $\ck$ consisting of flat modules. Note that $\Ff$ can be described in terms of $\Mm$: an object $A$ is in $\Ff$ if and only if the initial morphism $0 \to A$ is in $\Mm$. We will say that $A$ is a cofibrant object (with respect to $\Mm$). Before even worrying about stable independence, is $\Ff_{\Mm}$ an accessible category? Using Theorem \ref{acc-to-aec}, one can see that this is equivalent to asking whether it is an AEC. Similar examples were studied by Baldwin-Eklof-Trlifaj \cite{bet}, where it is shown that $\Ff_{\Mm}$ is an AEC if and only if $\Ff$ has refinements: there is a regular cardinal $\lambda$ such that any flat module $A$ is the colimit of an increasing chain $\seq{A_i : i < \delta}$ with $A_0 = 0$ and $A_{i + 1} / A_i$ flat and $\lambda$-presentable. Earlier, Rosický had essentially shown \cite[4.5]{flat-covers-factorizations} that having refinements is equivalent to $\Mm$ being cofibrantly generated (even by a subset of $\Mm \cap \Ff$)! Thus having refinements is yet another disguise for being cofibrantly generated. Using a slight variation on Theorem \ref{cofib}, one can then close the loop to deduce that if $\Ff$ has refinements then $\Ff_{\Mm}$ has a stable independence notion. It turns out that having refinement is closely connected to the (now resolved) \emph{flat cover conjecture} \cite{enochs-flat-cover-orig} (a dualization of the fact that every module has an injective envelope). An argument of Bican, El Bashir, and Enochs \cite{flat-cover} (which also easily could have been deduced from existing results on accessible categories, see \cite[3.2]{flat-covers-factorizations} or \cite[6.21]{more-indep-v2}) establishes that indeed, $\Ff$ has refinements, and thus $\Ff_{\Mm}$ indeed is an AEC with a stable independence notion.
  \end{enumerate}
\end{example}

\subsection{Higher dimensional independence and categoricity}\label{higher-dim-sec}

Recall that stable independence was defined as a class of squares satisfying certain properties, the most important of which was that the arrow category $\ck_{\smallnf}$ (whose objects are the arrows of $\ck$ and whose morphisms are the independent squares) should be accessible.

Now, given any accessible category, it makes sense to ask whether it has a stable independence notion. Does $\ck_{\smallnf}$ have a stable independence notion? If it does (call this stable independence notion $\nf^\ast$), then what do the objects and arrow look like in the corresponding category $(\ck_\smallnf)_{\smallnf^\ast}$? Well, the objects are the morphisms of $\nf^\ast$, i.e.\ independent squares. The morphisms, in turn, are $\nf^\ast$-independent ``squares'' in $\ck_{\smallnf}$, but they really are morphisms between two independent squares of $\ck$, hence it makes sense to call them independent \emph{cubes}. Continuing in this way, one can define when a category has an $n$-dimensional stable independence notion, for any $n < \omega$ (the original definition of stable independence is the case $n = 2$). 

Very nice, but what are (possibly multidimensional) stable independence notions good for? Very roughly, they are useful to prove the existence and uniqueness of certain objects in a category. As a simple example, let $\K$ be an AEC with $\LS (\K) = \aleph_0$ (e.g.\ the AEC of abelian groups, ordered with subgroup). Let's assume that there is an object of cardinality $\aleph_0$, and suppose we want to establish that $\K$ has an object of cardinality $\aleph_1$. A simple way would be to first show that for every countable $A \in \K$, there exists a countable $B \in \K$ with $A \lta B$. Equivalently, there is a morphism with domain $A$ that is \emph{not} an isomorphism. If this last property holds, then we can build a strictly increasing chain $\seq{A_i : i < \omega_1}$ of countable objects (taking  unions at limits) and the union of this chain will be the desired object of cardinality $\aleph_1$.

One can think of this construction as establishing a $0$-dimensional property in $\aleph_1$ (existence of an object) by using a $1$-dimensional property in $\aleph_0$ (existence of extensions). There was nothing special about $\aleph_0$ and $\aleph_1$ in this example, we could have replaced them with $\lambda$ and $\lambda^+$ for an arbitrary infinite cardinal $\lambda$. In particular, existence in $\aleph_2$ is implied by extensions in $\aleph_1$. Now how do we get existence of extensions in $\aleph_1$? Well, it is a $1$-dimensional property, so it seems reasonable that we should look for a $2$-dimensional property in $\aleph_0$. Such a property is the \emph{disjoint amalgamation property}: any span can be completed to a pullback square. Note that in many cases, an independent square will be a pullback square \cite[10.6]{indep-categ-advances}. Still, existence is a purely combinatorial property. Stable independence notions become really useful to prove \emph{uniqueness} properties. The simplest uniqueness property is what model theorists call \emph{categoricity} (this is somewhat unfortunate terminology, dating back from before the invention of category theory \cite{veblen-geom-categ, los-conjecture}).

\begin{defin}\label{categ-def}
  For an infinite cardinal $\lambda$, a category is called \emph{$\lambda$-categorical} (or \emph{categorical in $\lambda$}) if it has exactly one object of size $\lambda$ (up to isomorphism).
\end{defin}

The following are classical examples of the occurrence of categoricity:

\begin{example} \
  \begin{enumerate}
  \item The category of abelian groups is not categorical in any infinite cardinal (for a cardinal $\lambda$, take $\lambda$ copies of $\mathbb{Z}$ and $\lambda$ copies of $\mathbb{Z} / 2 \mathbb{Z}$: the first is torsion-free the other is not -- they are not isomorphic).
  \item The category of all dense linear orders without endpoints is categorical in $\aleph_0$, but not in any uncountable cardinal.
  \item The category of sets is categorical in every infinite cardinal.
  \item The category of vector spaces over $\mathbb{Q}$ is categorical in every uncountable cardinal (but not in $\aleph_0$: consider the spaces of dimension 1 and 2).
  \item The category of all algebraically closed fields of a fixed characteristic is categorical in every uncountable cardinal (but again not in $\aleph_0$).
  \item \cite[6.3]{internal-sizes-jpaa} The category of all Hilbert spaces is categorical in every uncountable size (but not in every uncountable cardinal). 
  \end{enumerate}
\end{example}

A classical theorem of Morley \cite{morley-cip} says that for a countable first-order theory $T$, if $\Elem (T)$ is categorical in some uncountable cardinal, then it is categorical in all uncountable cardinals. A proof can be sketched as follows: derive the existence of a stable independence notion from categoricity, then use it to build enough of a dimension theory to transfer categoricity. While it seems that having a single object of a given size is a relatively rare occurrence, it seems to be a useful test case as tools developed in proofs of categoricity transfers are often much more general. For example, the category of abelian groups with monomorphisms still has a stable independence notion \cite[5.3]{indep-categ-advances}. For this reason, many generalizations of Morley's theorem have been conjectured. One variation is:

\begin{conjecture}[Shelah's eventual categoricity conjecture]\label{categ-conj}
  If an AEC is categorical in a proper class of cardinals, then it is categorical on a tail of cardinals.
\end{conjecture}

Note that we have strengthened the starting assumption to categoricity on a proper class. Often, assuming categoricity in a single ``high-enough'' cardinal, where ``high-enough'' has some effective meaning, is enough. Note also that the conjecture is open even for accessible categories (although I suspect it should be wrong in full generality there). The eventual categoricity conjecture for AECs implies the eventual categoricity conjecture for \emph{finitely} accessible categories \cite[5.8(2)]{beke-rosicky}. 

It should be noted that multidimensional stable independence notions were introduced by Shelah to make progress on categoricity questions in $\Ll_{\omega_1, \omega}$ (in a very different form and a much more special situation than described here), see \cite{sh87a, sh87b}. Very recently, multidimensional independence was used by Shelah and myself to prove the eventual categoricity conjecture in AECs, assuming the existence of a proper class of strongly compact cardinals \cite{multidim-v2}.

So how exactly is multidimensional stable independence used? It is a complicated story, that I do not have space to tell here. One idea is that if we have some version of multidimensional stable independence notion just for objects of size $\lambda$, then we can lift it up to a multidimensional stable independence notion on the whole category. Another key result is that once we have a multidimensional stable independence notion, we can get a much better approximation to existence of pushouts. In particular, among the independent amalgams of a given span, there will be a weakly initial one, and it will be unique up to (not necessarily unique) isomorphism. Such an object is called a \emph{prime object} (over the span) by model theorists. In this sense, every span has a ``very weak'' colimit. This is enough to obtain a notion of generation that helps build a dimension theory. 

Finally, let's note that categoricity is a $0$-dimensional uniqueness property, and we may want to know higher dimensional versions (even before proving existence of prime objects). For the two-dimensional case, instead of building amalgam that are, in the sense above, minimal, we will want to get amalgams that are maximal, in the sense of being very homogeneous (i.e.\ with a lot of injectivity). This is given by the notion of a \emph{limit object}, which we survey (in the one dimensional case) in the next section.

\section{Element by element constructions in concrete categories}\label{aec-sec}

The framework of abstract elementary classes is very useful to perform element by element constructions: we can use the concreteness to build certain objects ``point by point''. We really will just use that we work in an accessible category with concrete directed colimits and all morphisms embedding (so nothing about the vocabulary will be needed). The reader uncomfortable with logic can think about this more general case.

We first need a definition of a ``point'' and what it means for two points to be the same. This depends of course on a base set, e.g.\ in the category of fields, the elements $e^{1/2}$ and $e^{1/4}$ are the same over $\mathbb{Q}$ but not the same over $\mathbb{Q} (e)$. The equivalence class of a point over a given base will be called a \emph{type}\footnote{In the AEC literature, the terms ``Galois types'' or ``orbital types'' are used. This is to avoid confusion with the model-theoretic syntactic types. However, there is not really a Galois theory and types are not necessarily orbits. Moreover in the first-order case the syntactic types are the same as the orbital types, and we will never refer to syntactic type. Thus we prefer the simpler terminology.}. The definition for AECs is due to Shelah, but in the wider setup of concrete categories it was first explored by Lieberman and Rosický \cite[4.1]{ct-accessible-jsl}:

\begin{defin}\label{point-def}
  Let $\K = (\ck, U)$ be a concrete category and let $A$ be an object of $\ck$. The \emph{category of points over $A$}, $\K_A^\ast$, is defined as follows:

  \begin{itemize}
  \item Its objects are pairs $(f, a)$, where $f: A \to B$ and $a \in U B$.
  \item A morphism from $(A \xrightarrow{f} B, a)$ to $(A \xrightarrow{g} C, b)$ is a $\ck$-morphism $B \xrightarrow{h} C$ such that $h f = g$ and $h (a) = b$.
  \end{itemize}
\end{defin}
\begin{remark}
  This is simply a pointed version of the category of cocones over the diagram consisting of only $A$. Morphisms between two cocones must respect points. Note that we wrote $h (a)$ instead of the more pedantic $(U h) (a)$.
\end{remark}

\begin{defin}\label{type-def}
  Let $\K = (\ck, U)$ be a concrete category and let $A$ be an object of $\ck$. Two points $(f, a)$, $(g, b)$ over $A$ \emph{have the same type} if they are connected in the category of points over $A$ (Definition \ref{point-def}).  The \emph{type} of a point over $A$ is just its equivalence class under the relation of having the same type (in the category of points over $A$).
\end{defin}
\begin{remark}
  If $\K$ has amalgamation, then the category of points over $A$ also has amalgamation. Thus two points $(A \xrightarrow{f} B, a)$ and $(A \xrightarrow{g} C, b)$ have the same type exactly when they are jointly connected in the category of points over $A$ (see Lemma \ref{compat-lem}). Explicitly, this means there exists $B \xrightarrow{h_1} D$ and $C \xrightarrow{h_2} D$ such that the diagram below commutes and moreover $h_1 (a) = h_2 (b)$:

  $$  
  \xymatrix@=3pc{
    B   \ar@{.>}[r]^{h_1} & D \\
    A \ar[u]^{f} \ar[r]_{g}  & C \ar@{.>}[u]_{h_2}
  }
  $$  
\end{remark}

Note that if $(A \xrightarrow{f} B, a)$ is a point in an AEC, then after some ``renaming'', it has the same type as some point $(A \xrightarrow{i} C, b)$, where $i$ is an inclusion (i.e.\ $A \lea C$). It can sometimes be convenient (to avoid keeping track of morphisms) to just look at points induced by inclusions, and this gives a slightly simpler notation for types. Nevertheless, we will not follow this approach here. 

\begin{defin}
  Let $\K$ be an AEC and let\footnote{We shift the notation to use $M$ and $N$ instead of $A$ and $B$ -- this is to emphasize that the objects are structures (models).} $M \in \K$.

  \begin{enumerate}
  \item Let $\gS (M)$ be the collection of all types over $M$.
  \item For $M \lea N$, a type $p \in \gS (M)$ is \emph{realized in $N$} if there exists a point $(M \xrightarrow{i} N, b)$ whose type over $M$ is $p$ (here, $i$ is the inclusion).
  \end{enumerate}
\end{defin}
\begin{remark}\label{type-monot}
  The following are easy exercises from the definitions. Let $\K$ be an AEC, $M \lea N$.
  
  \begin{enumerate}
  \item If $p \in \gS (M)$ is realized in $N$ and $N \lea N'$, then $p$ is realized in $N$'.
  \item If $p \in \gS (M)$, then there exists an extension of $M$ in which $p$ is realized. Moreover, by the smallness axiom, we can ensure that extension has size at most $\LS (\K) + | U M|$. In particular, $\gS (M)$ is a set.
  \item If $q \in \gS (M_0)$, $M_0 \lea M$, and $\K$ has the amalgamation property, then there exists an extension of $M$ in which $q$ is realized (and we can similarly ensure the extension has size at most $\LS (\K) + | U M|$).
  \end{enumerate}
\end{remark}

How can types allow us to construct category-theoretically interesting objects? The following is an interesting definition:

\begin{defin}
  Let $\K$ be an AEC, let $\lambda$ be an infinite cardinal, and let $M \lea N$ both be in $\K$. We say that $N$ is \emph{$\lambda$-universal over $M$} if for any given $M \xrightarrow{f} M'$ such that $M'$ has size strictly less than $\lambda$, there exists $M' \xrightarrow{g} N$ such that the following diagram commutes (by convention, we will not label inclusions):

  $$  
  
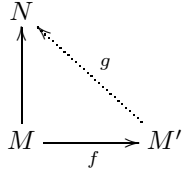
\xymatrix@=3pc{
    N  & \\
    M \ar[u] \ar[r]_{f}  & M' \ar@{.>}[ul]_{g}
  }
  $$

  When $\lambda = |U M|^+$, we omit it and say that $N$ is universal over $M$.
\end{defin}
\begin{remark}
  By renaming, we can assume without loss of generality that $f$ is also an inclusion. Recall also that ``size'' in AECs means the same as ``cardinality of the universe'' (see \ref{unif-ex}(\ref{unif-ex-aec})).
\end{remark}

In an AEC with amalgamation, we can build universal objects categorically by a simple exhaustion argument. However, in typical cases, $N$ above will be much bigger in size than $M$. For example, if $M$ has size $\lambda$ and $|\gS (M)| > \lambda$, then because a universal object over $M$ must realize all types over $M$, it must have at least $\lambda^+$-many elements. What if, on the other hand, $|\gS (M)| = \lambda$? This condition is given a name (recall that we use $\K_\lambda$ to denote the class of objects of size $\lambda$):

\begin{defin}\label{stable-def}
  An AEC is said to be \emph{$\lambda$-stable} (or \emph{stable in $\lambda$}) if $|\gS (M)| = \lambda$ for any $M \in \K_\lambda$.
\end{defin}

The name ``stable'' refers to the same kind of model-theoretic stability as ``stable independence notion''. In fact, in the first-order case, being stable on a proper class of cardinals is equivalent to the existence of a stable independence notion. The reader can think of stability as saying that ``most'' elements over a given base are transcendentals (i.e.\ all have the same type). In general, the existence of a stable independence notion in an AEC implies stability in certain cardinals \cite[8.16]{indep-categ-advances}. We will show that, assuming stability in $\lambda$, we can build universal extensions of the same cardinality. The result is due to Shelah \cite[II.1.16]{shelahaecbook} but the proof we give here is new and, as opposed to previous proofs (see \cite[6.2]{ct-accessible-jsl}) does \emph{not} use the coherence axiom.

\begin{thm}\label{univ-ext}
  Let $\K$ be an AEC and let $\lambda \ge \LS (\K)$. Assume that $\K$ is stable in $\lambda$ and $\K_\lambda$ has amalgamation. For any $M \in \K_\lambda$, there exists $N \in \K_\lambda$ such that $M \lea N$ and $N$ is universal over $M$.
\end{thm}
\begin{proof}
  We first build $\seq{M_i : i \le \lambda}$ increasing continuous in $\K_\lambda$ (i.e.\ $i < j$ implies $M_i \lea M_j$, all objects are in $\K_\lambda$, and at limit ordinals we take unions) such that $M_0 = M$ and $M_{i + 1}$ realizes all types over $M_i$ for all $i < \lambda$. This is possible: the only step to implement is the successor step, so assume that $M_i$ is given and we have to build $M_{i + 1}$. We know that $|\gS (M_i)| = \lambda$ by stability, so list the types as $\seq{p_j : j < \lambda}$. We build $\seq{M_{i, j} : j \le \lambda}$ increasing continuous in $\K_\lambda$ such that $M_{i, 0} = M_i$ and $M_{i, j + 1} \in \K_\lambda$ realizes $p_j$ for all $j < \lambda$. This is possible by Remark \ref{type-monot} and at the end we can set $M_{i + 1} = M_{i, \lambda}$.

  We now prove that $M_\lambda$ is universal over $M_0$. This is the hard part of the argument. Fix $M_0 \xrightarrow{f_0} M_0'$, with $M_0' \in \K_\lambda$. We want to find $g: M' \to M_\lambda$ with $gf_0$ fixing $M_0$. For each $i \le j \le \lambda$, we will build $M_i \xrightarrow{f_i} M_i'$ and $M_i' \xrightarrow{h_{ij}} M_j'$ such that for $i \le j \le k \le \lambda$, $M_i' \in \K_\lambda$, $h_{ik} = h_{jk}h_{ij}$, $h_{ii} = \id_{M_i'}$, and $h_{ij} f_i = f_j \rest M_i$. That is, the following commutes:

  $$  
  \xymatrix@=3pc{
    M_i' \ar[r]_{h_{ij}} & M_j' \\
    M_i \ar[u]_{f_i} \ar[r]  & M_j \ar[u]_{f_j}
  }
  $$

  We further require this diagram to be smooth (i.e.\ at limit ordinals $k$, $f_k$ is given by the colimit of $(f_i, h_{ij})_{i \le j < k}$). To go from $i$ to $i + 1$, we make sure that the type of a certain point $(M_i \xrightarrow{f_i} M_i', a_i)$ is realized in $M_{i + 1}$: there is $b_i \in U M_{i + 1}$ so that $(M_i \to M_{i + 1}, b_i)$ has the same type as $(f_i, a_i)$. This type equality is witnessed by $f_{i + 1}$ and $h_{i, i + 1}$. In particular, $h_{i, i + 1} (a_i) = f_{i + 1} (b_i)$, so $h_{i, i + 1} (a_i)$ is in the range of $f_{i + 1}$. We make sure to choose the $a_i$'s in such a way that, in the end, we catch our tail: $\{h_{j, \lambda} (a_j) : j < \lambda\} =  U M_\lambda'$. This is possible by some bookkeeping (essentially just using that $|\lambda \times \lambda| = \lambda$). We give the details (for a more general setup) in the appendix, see Theorem \ref{univ-ext-pf}.

  The construction we just described ensures that $f_\lambda$ is a surjection, hence an isomorphism! Thus $g := f_\lambda^{-1} h_{0 \lambda}$ is the desired embedding of $M_0'$ into $M_\lambda$.
\end{proof}
\begin{remark}\label{mh-sat-rmk}
  A similar argument (see Theorem \ref{mh-sat-pf}) establishes the \emph{model-homogeneous = saturated lemma} \cite[II.1.14]{shelahaecbook}: when $\lambda \ge \LS (\K)$, objects that realize all types over every substructure of size $\lambda$ will in fact be universal over every substructure of size $\lambda$.
\end{remark}

Once we know we can build universal extensions of the same cardinality, we can consider iterating them, and we get to the notion of a \emph{limit object}:

\begin{defin}
  Let $\K$ be an AEC, let $M \lea N$, let $\lambda \ge \LS (\K)$, and let $\delta < \lambda^+$ be a limit ordinal. We say that $N$ is \emph{$(\lambda, \delta)$-limit over $M$} if there exists an increasing continuous chain $\seq{M_i : i \le \delta}$ in $\K_\lambda$ such that $M_0 = M$, $M_\delta = N$, and $M_{i + 1}$ is universal over $M_i$ for all $i < \delta$.
\end{defin}

Theorem \ref{univ-ext} establishes that, if $\K$ is stable in $\lambda$, there exists limit models. It turns out that limit models have uniqueness properties as well. For example, a back and forth argument establishes that for $\delta_1, \delta_2 < \lambda^+$, if $M_\ell$ is $(\lambda, \delta_\ell)$-limit over $M$ for $\ell = 1,2$ and $\cf{\delta_1} = \cf{\delta_2}$, then there exists an isomorphism from $M_1$ to $M_2$ which fixes $M$.

What if $\cf{\delta_1} \neq \cf{\delta_2}$? The question of \emph{uniqueness of limit objects} asks whether the above is still true in this case (this can be thought of as a one-dimensional version of categoricity). For model theorists, we note that a positive answer is equivalent to superstability for first-order theories (thus limit objects provide a good way to define superstability categorically). More generally, in a stable theory $T$, limit models will be isomorphic exactly when their cofinality is at least $\kappa (T)$. In any case, proving uniqueness of limit models requires a good notion of stable independence for elements, allowing one to build very independent objects point by points. This is another (much harder than in Theorem \ref{univ-ext})  example of the use of ``points'' in the theory of AECs. See \cite{gvv-mlq} for an exposition of some result in the theory of limit objects in AECs, though stronger theorems have now been found, e.g.\ \cite{vandieren-symmetry-apal}. A state of the art result from categoricity is \cite[5.7(2)]{categ-saturated-afml} (an AEC $\K$ has \emph{no maximal objects} if for any $M \in \K$ there exists $N \in \K$ with $M \lea N$ and $M \neq N$):

\begin{thm}
  Let $\K$ be an AEC with amalgamation and no maximal objects. Let $\mu > \LS (\K)$. If $\K$ is categorical in $\mu$, then for any $\lambda \in [\LS (\K), \mu)$, any two limit objects of cardinality $\lambda$ are isomorphic.
\end{thm}

The reader may ask what limit models look like in specific categories. This is studied in recent work of Kucera and Mazari-Armida \cite{limit-abelian-apal, univ-modules-pp-v4} for certain categories of modules.

We end this section by noting that stable independence notions themselves can be built ``element by element''. For this, one starts with a good notion of stable independence for types (what Shelah calls a good frame), and tries to ``paste the points together''. See \cite{jrsh875} for an introduction to the theory of good frames and \cite{downward-categ-tame-apal, categ-primes-mlq, categ-amalg-selecta} for recent applications to the categoricity spectrum problem. We state in particular the following result \cite{categ-amalg-selecta}:

\begin{thm}
  Let $\K$ be an AEC with amalgamation. Assume the weak generalized continuum hypothesis: $2^\lambda < 2^{\lambda^+}$ for all infinite cardinals $\lambda$. If $\K$ is categorical in some cardinal $\mu \ge \ehanf{\LS (\K)}$, then $\K$ is categorical in all cardinals $\mu' \ge \ehanf{\LS (\K)}$, and moreover there is a stable independence notion on the category $\K_{\ge \ehanf{\LS (\K)}}$.
\end{thm}

Note that the statements of the last two theorems are purely category-theoretic, in the sense that their statement does not use concreteness, points, etc (if the reader is worried about whether the definition of an AEC is category-theoretic, they can look at the results for the special case of a finitely accessible category with all morphisms monos, see Theorem \ref{aec-acc}). I am not aware of ``purely category-theoretic'' proofs of any statements like this, so I suspect that the element by element methods used to study AECs can be useful even if one is interested only in categorical problems.

\section{Some known results and open questions}\label{problem-sec}

\subsection{Questions on categoricity}

Shelah's eventual categoricity conjecture for AECs (Conjecture \ref{categ-conj}) is still open, in ZFC, but is known to hold from many different types of (quite mild) assumptions. In many cases, we can say more about the ``high-enough'' bound and even (in (\ref{categ-3}) below) list exactly what the possibly categoricity spectrums are. For example:

\begin{enumerate} 
\item (Large cardinal axioms, \cite{multidim-v2}) Assuming there is a proper class of strongly compact cardinals, an AEC categorical in a proper class of cardinals is categorical on a tail of cardinals.
\item (Large cardinal axioms plus cardinal arithmetic, \cite{multidim-v2}) Assuming there is a proper class of \emph{measurable} cardinals and WGCH\footnote{Recall that this means that $2^\lambda < 2^{\lambda^+}$ for all infinite cardinals $\lambda$.}, an AEC categorical in a proper class of cardinals is categorical on a tail of cardinals.
\item\label{categ-3} (Amalgamation plus cardinal arithmetic \cite[9.7]{categ-amalg-selecta}) Assume WGCH. Let $\K$ be a large AEC with amalgamation (and $\K_{<\LS (\K)} = \emptyset$). Exactly one of the following holds:

  \begin{enumerate}
  \item $\K$ is not categorical in any cardinal above $\LS (\K)$.
  \item There exists $n \le m < \omega$ such that $\K$ is categorical in all cardinals in $[\LS (\K)^{+n}, \LS (\K)^{+m}]$ and no other cardinals.
  \item There exists $\chi < \ehanf{\LS (\K)}$ such that $\K$ is categorical in all $\mu \ge \chi$ (and no other cardinals).
  \end{enumerate}

  It is known that examples exist of all three types.
\item (No maximal objects plus strong cardinal arithmetic \cite[10.14]{categ-amalg-selecta}) Assume $\Diamond_S$ for every stationary set $S$ (this holds for example if V = L). An AEC with no maximal objects categorical in a proper class of cardinals is categorical on a tail of cardinals.
\item (Universal class \cite{ap-universal-apal, categ-universal-2-selecta}) If a universal class $\K$ is categorical in some $\mu \ge \beth_{\ehanf{\LS (\K)}}$, then it is categorical in all $\mu' \ge \beth_{\ehanf{\LS (\K)}}$. This holds more generally for multiuniversal classes \cite{abv-categ-multi-apal}
\end{enumerate}

Other approximations to categoricity (for example from tameness, a locality property of types that has not been discussed here) are in \cite{sh394, tamenesstwo, tamenessthree}. Note that given any finitely accessible category $\ck$, we have that $\ck_{mono}$ is an AEC and the embedding $\ck_{mono} \to \ck$ preserves and reflects presentability ranks, hence categoricity. Thus the results above are, in particular, valid for any finitely accessible category.

\begin{question} \
  \begin{enumerate}
  \item Is the eventual categoricity conjecture for AECs provable in, or at least consistent with, ZFC?
  \item Is there a counterexample to eventual categoricity for accessible categories? Is it true at least for accessible categories with directed colimits? Can we generalize the proof of eventual categoricity for continuous first-order logic \cite{shus837} to category-theoretic setups?
  \item (Diliberti) Can one give ``purely category-theoretic'' proofs of categoricity transfers, at least in simple cases (for example, for eventual categoricity in locally finitely presentable categories)?
  \item Is eventual categoricity true for locally presentable categories?
  \end{enumerate}
\end{question}

\subsection{Questions on set-theoretic aspects}

\begin{question} \
  \begin{enumerate}
  \item Is the presentability rank of every high-enough object in an accessible category always a successor?
  \item What is an example of a large accessible category that is not LS-accessible\footnote{Such an example would yield to a failure of eventual categoricity: take the coproduct with the category of sets \cite[6.3]{beke-rosicky}.}?
  \item If a category is $\lambda$-accessible, for $\lambda$ big-enough, does it have an object of size $\lambda$? (see Theorem \ref{ls-acc-thm})
  \item Is any large accessible category with directed colimits LS-accessible? (This is asked already in \cite{ct-accessible-jsl}).
  \item What other types of compactness can be expressed using accessible functors (Section \ref{set-func-sec})?
  \item Are there other local methods than Theorem \ref{ap-thm} that allow us to prove amalgamation in ZFC?    
  \end{enumerate}
\end{question}

Regarding the second question, we can give a version that does not mention category-theoretic sizes: does there exist a $\mu$-AEC $\K$ such that, for a proper class of cardinals $\lambda$ with $\lambda < \lambda^{<\mu}$, $\K$ is categorical in $\lambda^{<\mu}$, and $\K$ has no objects of cardinality in $[\lambda, \lambda^{<\mu})$? See \cite[4.16]{internal-sizes-jpaa} for why such a $\mu$-AEC is not LS-accessible.

\subsection{Questions on accessible categories vs AECs}

\begin{question} \
  \begin{enumerate}
  \item What is the role of the coherence axiom of AECs?
  \item Is there a short characterization of AECs that is completely category-theoretic, in the sense that it does not refer to concrete functors (as in \cite{ct-accessible-jsl}), or embeddings into (variations on) category of structures (as in \cite[5.7]{beke-rosicky})?
  \item Is there a natural logic axiomatizing AECs (see \cite[\S4]{logic-intersection-bpas})?
  \end{enumerate}
\end{question}

\subsection{Questions on stable independence}

\begin{question} \
  \begin{enumerate}
  \item What are more occurrences of stable independence in mainstream mathematics?
  \item How does stable independence interact with accessible functors?
  \item Can one characterize when the uniqueness of limit objects holds in terms of properties of stable independence (say in accessible categories with directed colimits)?
  \item Does stable independence tell us anything interesting about metric classes (Banach spaces, Hilbert spaces, etc.)? See \cite{byuscontlog} on first-order stability theory for metric classes.
  \item Is there a theory of independence in accessible categories mirroring that of independence in simple unstable first-order theories (see \cite{kp97})? What about other model-theoretic classes of unstable theories?    
  \end{enumerate}
  \end{question}

\subsection{Questions on the model theory of AECs}

The questions below are more technical, and cannot be understood just from the material of this paper. I chose to collect them here for the convenience of the expert reader:

\begin{question} \
  \begin{enumerate}
  \item If $\K$ is a $\lambda$-superstable AEC, does it have $\lambda$-symmetry (see \cite{vandieren-symmetry-apal})? More generally, what are the exact relationships between uniqueness of limit objects in $\lambda$, $\lambda$-symmetry, and $\lambda$-superstability?
  \item If $\K$ is a $\lambda$-stable AEC and $\K_\lambda$ has amalgamation and no maximal objects, can we prove uniqueness of long-enough limit objects in $\lambda$, as in \cite{limit-strictly-stable-v4}, but without using a continuity property for splitting?
  \item Is the following stability spectrum theorem true? In a large stable $\LS (\K)$-tame AEC with amalgamation, there exists $\chi < \ehanf{\LS (\K)}$ such that for all $\lambda \ge \ehanf{\LS (\K)}$, $\K$ is stable in $\lambda$ if and only if $\lambda = \lambda^{<\chi}$. Approximations are in \cite{stab-spec-jml}.
  \end{enumerate} 
\end{question}

\section{Further reading}\label{reading-sec}

In addition to all the references cited already, we mention some resources that may help newcomers become acquainted with the field. We repeat again that Makkai-Paré \cite{makkai-pare} and especially Adámek-Rosický \cite{adamek-rosicky} are the standard textbooks on accessible categories. The category-theoretic singular compactness theorem (Theorem \ref{sing-compact}) appears in \cite{cellular-singular-jpaa}, which has numerous examples and explanations. The relationship between accessible categories and abstract elementary classes is investigated in, for example, \cite{lieberman-categ,beke-rosicky, ct-accessible-jsl}. The Beke-Rosický paper, specifically, started the abstract study of category-theoretic sizes continued in \cite{internal-sizes-jpaa, internal-improved-v3-toappear}. The category-theoretic notion of stable independence is introduced in \cite{indep-categ-advances}, and a follow-up (establishing the connection with cofibrant generation) is \cite{more-indep-v2}.

To become acquainted with abstract elementary classes specifically, two introductions are Grossberg's survey \cite{grossberg2002} and Baldwin's book \cite{baldwinbook09}. Recently, classes about AECs were given at Harvard University by both Will Boney and myself, and both classes had lecture notes \cite{wb-aec-notes, sv-aec-notes} that give an updated take on the subject. The survey about tame AECs \cite{bv-survey-bfo} may also be helpful to get acquainted with the literature. When one starts studying independence for types, Shelah's good frames, a ``pointed'' and localized version of stable independence, become an unavoidable concept. Currently, the best introduction to good frames is the paper of Jarden and Shelah \cite{jrsh875}. Finally, it is impossible not to mention Shelah's two volume book \cite{shelahaecbook, shelahaecbook2} which has a very interesting and readable introduction, and is a gold mine of deep (but not always easily readable) results on good frames and AECs generally.

\appendix

\section{Forcing and construction categories}

I give here a general category-theoretic framework for point by point ``exhaustion arguments'' such as building algebraic closure of fields (or more generally saturated models), or proving a given extension realizing all types many times is universal as in Theorem \ref{univ-ext}. It is also a natural framework in which to understand set-theoretic forcing. To the best of my knowledge, this is new.

\begin{defin}
  A \emph{construction category} is a triple $\K = (\ck, U, U_0)$, where:

  \begin{enumerate}
  \item $\ck$ is a category.
  \item $U: \ck \to \Set$ is a faithful functor.
  \item $U_0: \ck \to \Set$ is a faithful\footnote{It seems the faithfulness of $U$ and $U_0$ is never used.} subfunctor of $U$: a faithful functor such that for all morphisms $A \xrightarrow{f} B$, $U_0 A \subseteq U A$ and $U_0 f = (U f) \rest U_0 A$.
  \end{enumerate}
\end{defin}

The idea is that, for an object $A$, $U A$ gives the elements that could ``conceivably'' be constructed at some point (e.g.\ $\ck$ could be the category of fields and $U A$ give the polynomials with coefficients from $A$, see Example \ref{forcing-ex}), while $U_0 A$ gives the elements that have been constructed already (e.g.\ in the example of the category of fields, $U_0 A$ could give the polynomials with coefficients from $A$ that have a root in $A$). We will be trying to find an object (called \emph{full} below) where everything that can be constructed in some extension has been constructed already. It may help the reader to think of the category $\ck$ as a partially ordered set.

\begin{example}[Set-theoretic forcing]
  Let $\Pp$ be a partially ordered set (we think of it also as a category). A notion of forcing for $\Pp$ (e.g.\ in the sense of \cite{forcing-omitting}) associates to each $p \in \Pp$ a set of sentences that it ``forces''  in such a way that if $p \le q$ then $q$ forces more sentences than $p$. Setting $U_0 p$ to be the formulas that $p$ forces, and $U p$ to be the set of all formulas, we obtain a construction category. 
\end{example}

\begin{defin}
  Let $\K$ be a construction category.

  \begin{enumerate}
  \item Given an object $A$ and an element $x \in U A$, we say that $x$ is \emph{constructed by stage $A$} if $x \in U_0 A$. We say that $x$ is \emph{constructible from $A$} if there exists a morphism $A \xrightarrow{f} B$ so that $f (x)$ is constructed by stage $B$.
  \item A directed diagram $D: I \to \K$ with maps $D_i \xrightarrow{d_{ij}} D_j$ is \emph{full} whenever the following is true: for any $i \in I$ and any $x \in U D_i$, \emph{if} for all $j \ge i$, $d_{ij} (x)$ is constructible from $D_j$, \emph{then} there exists $j \ge i$ such that $d_{ij} (x)$ is constructed by stage $D_j$.
  \item For an object $A$ and a set $X \subseteq U A$, we say that $A$ is \emph{full for $X$} if any $x \in X$ that is constructible from $A$ is constructed by stage $A$. We say that $A$ is \emph{full} if it is full for $U A$.
  \end{enumerate}
\end{defin}

The following are basic properties of the definitions:

\begin{remark}
  Let $\K$ be a construction category, $A \xrightarrow{f} B$ be a morphism, and $x \in U A$.
  
  \begin{enumerate}
  \item If $x$ is constructed by stage $A$, then $x$ is constructible from $A$.
  \item If $x$ is constructed by stage $A$ then $f (x)$ is constructed by stage $B$.
  \item If $f(x)$ is constructible from $B$, then $x$ is constructible from $A$.
  \item If $A$ is an amalgamation base (Definition \ref{ap-def}) and $x$ is constructible from $A$, then $f (x)$ is constructible from $B$.
  \item If $A$ is full for $X$, then $B$ is full for $f[X]$.
  \item An object $A$ is full if and only if the corresponding directed diagram with one object is full. 
  \end{enumerate}
\end{remark}

\begin{lem}\label{full-lem}
  Let $\K$ be a construction category and let $D: I \to \ck$ be a directed diagram with maps $D_i \xrightarrow{d_{ij}} D_j$.

  \begin{enumerate}
  \item If $D$ is full, then for any cocone $(D_i \xrightarrow{f_i} A)_{i \in I}$ for $D$ and any $i \in I$, $f_i^{-1}[U_0 A] \cap U D_i \subseteq \bigcup_{j \ge i} d_{ij}^{-1}[U_0 D_j]$.
  \item An object $A$ is full if and only if for any morphism $A \xrightarrow{f} B$, $U A \cap f^{-1}[U_0 B] = U_0 A$.
  \item\label{full-lem-3} If $D$ is full and $(D_i \xrightarrow{f_i} A)_{i \in I}$ is a cocone for $D$, then $A$ is full for $\bigcup_{i \in I} f_i[U D_i]$. In particular, if $U A = \bigcup_{i \in I} f_i[U D_i]$ then $A$ is full.
  \end{enumerate}
\end{lem}
\begin{proof} \
  \begin{enumerate}
  \item Let $(D_i \xrightarrow{f_i} A)_{i \in I}$ be a cocone for $D$ and let $i \in I$. Let $x \in f_i^{-1}[U_0 A] \cap U D_i$. Let $y := f_i (x)$ (so $y \in U_0 A$). For any $j \ge i$, $f_j$ (and the fact that $y \in U_0 A$) witnesses that $d_{ij} (x)$ is constructible from $D_j$. By definition of a full diagram, there exists $j \ge i$ such that $d_{ij} (x)$ is constructed by stage $D_j$, i.e.\ $d_{ij} (x) \in U_0 D_j$. Thus $x \in d_{ij}^{-1} (U_0 D_j)$.
  \item If $A$ is full, then the previous part gives that for any morphism $A \xrightarrow{f} B$, $U A \cap f^{-1}[U_0 B] \subseteq U_0 A$. The reverse inclusion is immediate because $U_0$ is a subfunctor of $U$. Conversely, assume that for any morphism $A \xrightarrow{f} B$, $U A \cap f^{-1}[U_0 B] =  U_0 A$. Let $x \in U A$ be constructible from $A$. This means there exists $A \xrightarrow{f} B$ such that $f (x) \in U_0 B$, i.e.\ $x \in f^{-1}[U_0 B]$. By assumption, $x \in U_0 A$, so $x$ is constructed by stage $A$. 
  \item Assume that $D$ is full. Let $y \in \bigcup_{i \in I} f_i[U D_i]$ be constructible from $A$. There exists $i \in I$ such that $y = f_i (x)$ for some $x \in U D_i$. Since $y$ is constructible from $A$, $x$ is constructible from $D_i$. In fact, for any $j \ge i$, $f_j$ witnesses that $d_{ij} (x)$ is constructible from $D_j$. Since $D$ is full, there exists $j \ge i$ so that $d_{ij} (x) \in U_0 D_j$. Thus $y = f_j d_{ij} (x) \in f_j[U_0 D_j] \subseteq U_0 A$, so $y$ is constructed by stage $A$.
  \end{enumerate}
\end{proof}

Although we will in the end mostly be interested in full objects, it is often helpful (and easier) to first verify that a full \emph{diagram} exists. Its colimit will then usually be the desired full object. In order for full diagrams to exist, objects should not be too big. We will measure size using cofinality of a certain ``order of construction''.

\begin{defin}
  Let $\K$ be a construction category and let $A$ be an object.

  \begin{enumerate}
  \item Define a relation $\le$ on $U A$ as follows: $x \le y$ if for all $A \xrightarrow{f} B$, if $f (y)$ is constructed by stage $B$, then $f (x)$ is constructed by stage $B$. 
  \item Let $\|A\|$ denote the cofinality of $(U A, \le)$. 
  \end{enumerate}
\end{defin}

Note that $\le$ is a preorder on $U A$, and $\|A\| \le |U A|$. The following will be our main tool to verify existence of full diagrams.

\begin{thm}[Existence of full diagrams]\label{full-existence}
  Let $\K$ be a construction category and let $\lambda$ be an infinite cardinal. If $\|A\| \le \lambda$ for all objects $A$, and (for $\alpha < \lambda$) any $\alpha$-indexed diagram has a cocone, then $\K$ has a $\lambda$-indexed full diagram.
\end{thm}

This will be a consequence of the following more general version (used in the proof of Theorem \ref{univ-ext-pf}), where we only require diagrams with ``small'' objects (according to a certain rank) to have an upper bound.

\begin{lem}\label{full-existence-lem}
  Let $\K$ be a construction category, let $\lambda$ be an infinite cardinal, and let $R : \K \to \lambda$ be given. If:

  \begin{enumerate}
  \item $\|A\| \le \lambda$ for all objects $A$.
  \item For any object $A$ and any $x \in U A$ that is constructible from $A$, there exists $A \xrightarrow{f} B$ such that $R B \le R A + 1$ and $f(x)$ is constructed by stage $B$.
  \item For every $\alpha < \lambda$ and every diagram $D: \alpha \to \K$, if $R D_i < \alpha$ for all $i < \alpha$, then $D$ has a cocone $(D_i \to A)_{i < \alpha}$ with $R A \le \alpha$.
  \end{enumerate}

  Then there is a $\lambda$-indexed full diagram $D: \lambda \to \K$ with $R D_i \le i$ for all $i < \lambda$.
\end{lem}
\begin{proof}[Proof of Theorem \ref{full-existence}]
  Apply Lemma \ref{full-existence-lem}, with $R A = 0$ for all objects $A$.
\end{proof}
\begin{proof}[Proof of Lemma \ref{full-existence-lem}]
  First, we fix a function $F: \lambda \to \lambda \times \lambda$ such that for each pair $(\alpha, \beta)$ in $\lambda \times \lambda$, there exists unboundedly-many $i < \lambda$ so that $F (i) = (\alpha, \beta)$. This can be done by first partitioning $\lambda$ into $\lambda$-many disjoint pieces of cardinality $\lambda$, then bijecting each of these pieces with $\lambda \times \lambda$. Write $(\alpha_i, \beta_i)$ for $F (i)$.

  We inductively build objects $\seq{D_i : i < \lambda}$, morphisms $\seq{D_i \xrightarrow{d_{ij}} D_j : i \le j < \lambda}$, and sequences $\seq{x_{i, j} : i, j < \lambda}$ such that:

  \begin{enumerate}
  \item $R D_i \le i$ for all $i < \lambda$.
  \item $d_{ii} = \id_{D_i}$, $d_{jk} d_{ij} = d_{ik}$ for all $i \le j \le k < \lambda$.
  \item For each $i < \lambda$, $\seq{x_{i, j} ; j < \lambda}$ enumerates the elements of a cofinal set of $U D_i$ (perhaps with repetitions).
  \item\label{forcing-3} For each $i < \lambda$, if $\alpha_i \le i$ and $d_{\alpha_i, i} (x_{\alpha_i, \beta_i})$ is constructible from $D_i$, then $d_{\alpha_i, i + 1} (x_{\alpha_i, \beta_i})$ is constructed by stage $D_{i + 1}$. Moreover, if there exists an object that is full over $D_i$, then $D_{i + 1}$ is full over $D_i$.
  \end{enumerate}

  \underline{This is enough}: Let $D: \lambda \to \ck$ be the diagram with maps $d_{ij}$. Then $D$ is the desired full diagram. Indeed, let $i < \lambda$, and let $x \in U D_i$ be such that $d_{ij} (x)$ is constructible from $D_j$ for all $j \ge i$. Let $\alpha := i$. Without loss of generality (using cofinality), $x = x_{\alpha, \beta}$ for some $\beta < \lambda$. By construction, there exists $j \ge i$ such that $(\alpha_j, \beta_j) = (\alpha, \beta)$. By assumption, $d_{i, j} (x)$ is constructible from $D_j$, so by (\ref{forcing-3}), $d_{i, j + 1} (x)$ is constructed by stage $D_{j + 1}$, as desired.

  \underline{This is possible}: We proceed by induction on $j < \lambda$. Assume inductively that $\seq{D_i : i < j}$, $\seq{d_{ii'} : i \le i' < j}$, and $\seq{x_{i, i'} : i < j, i' < \lambda}$ have been constructed. We will build $D_j$, $\seq{d_{i j} : i \le j}$, and $\seq{x_{j, j'} : j' < \lambda}$. First assume that $j$ is limit or zero. By assumption, the diagram $D: j \to \ck$ with maps $\seq{d_{ii'} : i \le i' < j}$ has a cocone $(D_i \xrightarrow{d_{ij}} D_j)_{i < j}$, with $R D_j \le j$. Set $d_{jj} = \id_{D_j}$, and let $\seq{x_{j, j'} : j' < \lambda}$ be any enumeration of $U D_j$. Now assume that $j$ is a successor: $j = i + 1$. If $\alpha_i > i$, or $\alpha_i \le i$ but $d_{\alpha_i, i} (x_{\alpha_i, \beta_i})$ is not constructible from $D_i$, then set $B = D_i$, $f = \id_{D_j}$. If $\alpha_i \le i$ and $d_{\alpha_i, i} (x_{\alpha_i, \beta_i})$ is constructible from $D_i$, then let $D_i \xrightarrow{f} B$ witness this, with $R B \le i + 1$. Set $D_i = B$, $d_{i, j} = f$, and let $d_{j, j} = \id_{D_j}$, $d_{i_0, j} = d_{i, j} d_{i_0, i}$ for all $i_0 < i$, and $\seq{x_{j, j'} : j' < \lambda}$ be any enumeration of $U D_j$.
\end{proof}

Even if we are unable to directly build full objects, we can also find a lot of them in continuous-enough chains:

\begin{thm}\label{refl-thm}
  Let $\lambda$ be a regular uncountable cardinal and let $\K = (\ck, U, U_0)$ be a construction category where $\ck$ is just the ordered set $\lambda$ and $U$-images of morphisms are inclusions. If $\|j\| < \lambda$ for all $j < \lambda$, then the set $\{j < \lambda \mid j \text{ is full for } \bigcup_{i < j} U i \}$ is closed unbounded.
\end{thm}
\begin{proof} Let $C := \{j < \lambda \mid j \text{ is full for } \bigcup_{i < j} U i \}$.
  
  \begin{itemize}
  \item \underline{$C$ is closed}: let $j < \lambda$ be a limit ordinal such that unboundedly-many $i < j$ are in $C$. Let $x \in \bigcup_{i < j} U i$ be constructible from $j$. Pick $i' \in C \cap j$ such that $x \in \bigcup_{i < i'} U i$. Then $x$ is constructible from $i'$ hence, by definition of $C$, $x$ is constructed by stage $i'$, hence by stage $j$.

  \item  \underline{$C$ is unbounded}: let $\alpha < \lambda$. We build $\seq{\alpha_n : n < \omega}$ an increasing sequence of ordinals below $\lambda$ such that for all $n < \omega$:
    
    \begin{enumerate}
    \item $\alpha_0 = \alpha$.
    \item Any $x \in U \alpha_n$ that is constructible in $\alpha_n$ is constructed by stage $\alpha_{n + 1}$.
    \end{enumerate}

    This is possible: given $\alpha_n$, for each $x \in U \alpha_n$ that is constructible in $\alpha_n$, there exists a least $i_x < \lambda$ such that $x$ is constructed by stage $i_x$. Since $\|\alpha_n\| < \lambda$ and $\lambda$ is regular, there exists $\alpha_{n + 1} > \alpha_n$ such that $\alpha_{n + 1} \ge i_x$ for all $x \in U \alpha_n$ constructible in $\alpha_n$.

    This is enough: let $\beta := \sup_{n < \omega} \alpha_n$. Since $\lambda$ is regular uncountable, $\beta < \lambda$. Moreover, $\beta$ is full for $\bigcup_{i < \beta} U i$. Indeed, if $x \in \bigcup_{i < \beta} U_i$ is constructible from $\beta$, then there exists $n < \omega$ such that $x \in U \alpha_n$ and $x$ is constructible from $\alpha_n$, hence constructed by stage $\alpha_{n + 1}$, hence by stage $\beta$.
  \end{itemize}
\end{proof}

Many well known constructions can easily be seen as special cases of the theorems just stated:

\begin{example}\label{forcing-ex} \
  \begin{enumerate}
  \item (Zorn's lemma) If $\Pp$ is a partially ordered set where each chain has an upper bound, then $\Pp$ has a maximal element. This can be obtained by applying Theorem \ref{full-existence} to $\lambda = |\Pp| + \aleph_0$, $\K = (\Pp, U, U_0)$, where $U p = \Pp$ for all $p \in \Pp$ and $U_0 p = \{q \in \Pp \mid q \le p\}$. Any upper bound to the full diagram gives the desired maximal element.
  \item (Existence of generics) Let $\Pp$ be the poset of all finite partial functions from $\omega$ to $\{0, 1\}$. Let $\K = (\Pp, U, U_0)$, where $U s = \omega$, $U_0 s = \dom (s)$. Then a full diagram $D: \omega \to \K$ corresponds to a (total) function $f: \omega \to \{0,1\}$. This generalizes to the existence of generics, in the sense of set-theoretic forcing \cite[14.4]{jechbook}. 
  \item (Existence of algebraic closure) Every field $F$ has an algebraic closure: take $\lambda = |F| + \aleph_0$, $\K = (\ck, U, U_0)$ where $\ck$ is the category of field extensions of $F$ of cardinality at most $\lambda$, $U A$ is the set of all polynomials with coefficients from $A$, and $U_0 A$ is the set of all such polynomials with a root in $A$. The colimit of the full diagram given by Theorem \ref{full-existence} is full (Lemma \ref{full-lem}(\ref{full-lem-3})), hence algebraically closed.
  \item (Existence of saturated models) Let $\K^\ast$ be an AEC with amalgamation and $\lambda > \LS (\K^\ast)$ be a regular cardinal such that $\K^\ast$ is stable in $\lambda$ and $\K_\lambda^\ast \neq \emptyset$. Let $\K = (\ck, U, U_0)$, where $\ck$ is the full subcategory of $\K^\ast$ with objects of cardinality $\lambda$, $U A$ is the set of all types over a substructure of $A$, and $U_0 A$ is the set of all such types that are realized in $A$. Note that we may well have $|U A| > \lambda$, but by stability we still have that $\|A\| \le \lambda$ (the types over $A$ form a cofinal set of size $\lambda$). The colimit of the full diagram given by Theorem \ref{full-existence} is full, hence is a saturated object of $\K_\lambda^\ast$ (i.e.\ all types over substructure of cardinality strictly less than $\lambda$ are realized).
  \item (Disjointness of filtrations on a club)  If $\lambda$ is a regular uncountable cardinal, $A \subseteq B$ are sets of cardinality $\lambda$, and $\seq{A_i : i < \lambda}$, $\seq{B_i : i < \lambda}$ are increasing continuous chains of subsets of cardinality strictly less than $\lambda$ such that $A = \bigcup_{i < \lambda} A_i$ and $B = \bigcup_{i < \lambda} B_i$, then the set of $i < \lambda$ such that $A \cap B_i = A_i$ is closed unbounded (in particular, if $A = B$ then $A_i = B_i$ on a closed unbounded set). Indeed, let $\K = (\lambda, U, U_0)$, where $U i = A_i \cup B_i$, $U_0 i = A_i \cap B_i$. By Theorem \ref{refl-thm}, the set of $i$ such that $i$ is full is closed unbounded. Now, if $i$ is full and $x \in A \cap B_i$, then $x \in A_j \cap B_j$ for some $j$, so $x$ is constructible from $i$, so is in $A_i \cap B_i \subseteq A_i$. Conversely, if $x \in A_i$ then it is constructible from $i$, hence in $A_i \cap B_i \subseteq A \cap B_i$. Thus $A \cap B_i = A_i$.
  \item Similarly to the previous example, density of reduced towers (in the study of uniqueness models, see e.g.\ \cite[5.5]{gvv-mlq}) can be seen as describing the existence of a full object in an appropriate construction category.
  \end{enumerate}
\end{example}

Let's now give more details on the proof of Theorem \ref{univ-ext}:

\begin{thm}\label{univ-ext-pf}
  Let $\K$ be an AEC, let $\lambda \ge \LS (\K)$ be such that $\K_\lambda$ has amalgamation, and let $\seq{M_i : i \le \lambda}$ be increasing continuous in $\K_\lambda$ such that $M_{i + 1}$ realizes all types over $M_i$. Then $M_\lambda$ is universal over $M_0$.
\end{thm}
\begin{proof}
  Let $\K^\ast = (\ck^\ast, V, V_0)$ be defined as follows:

  \begin{itemize}
  \item The objects of $\ck^\ast$ are morphisms $M_i \xrightarrow{f} M$ for $M \in \K_\lambda$, $i < \lambda$, such that $f \rest M_0 = \id_{M_0}$.
  \item A morphism from $M_i \xrightarrow{f} M$ to $M_j \xrightarrow{g} N$, with $i \le j < \lambda$, is a map $h: M \to N$ such that the following diagram commutes:

    $$
        \xymatrix@=3pc{
          M \ar[r]^h & N \\
          M_i \ar[u]_f \ar[r] & M_j \ar[u]_g \\
        }
        $$
  \item $V (M_i \to M) = U M$, and the $V$-image of a morphism $h$ from $M_i \to M$ to $M_j \to N$ is $U h$ (where $U$ is the universe functor from $\K$ to $\Set$)
  \item $V_0 (M_i \xrightarrow{f} M) = U f[M_i]$, and the $V_0$-image of a morphism $h$ from $M_i \to M$ to $M_j \to N$ is $U h \rest f[M_i]$.
  \end{itemize}

  This is easily checked to be a construction category. Now let $M_i \xrightarrow{f} N$ be given, with $i < \lambda$. We show that any $x \in U N$ is constructible from $f$. Indeed, $M_{i + 1}$ realizes all types over $M_i$, so in particular it realizes the type of $(x, M_i \xrightarrow{f} N)$. Thus there is $x' \in U M_{i + 1}$ and a commutative diagram:
  
      $$
        \xymatrix@=3pc{
          M \ar[r]^h & N \\
          M_i \ar[u]_f \ar[r] & M_{i + 1} \ar[u]_g \\
        }
        $$

        with $g (x') = h (x)$. In particular, $h (x) \in V_0 (g)$, so is constructed by stage $g$.

        Assume now that a diagram $D : \lambda \to \ck^\ast$ is full. Let $M_j \xrightarrow{f} N$ be a colimit (in $\K$) of $D$ (so $j \le \lambda$). From the previous discussion, it is easy to see that $f$ is surjective, hence an isomorphism. This then gives the result: for any $N_0 \in \K_\lambda$, $M_0 \xrightarrow{f_0} N_0$, the subcategory $\K^\ast_{f_0}$ of objects of $\K^\ast$ above $f_0$ satisfies the hypotheses of Lemma \ref{full-existence-lem} (with $R (M_i \to M) = i$), hence has a full diagram, whose colimit must therefore induce an embedding of $N_0$ into $M_\lambda$.
\end{proof}

We can similarly prove that model-homogeneous is equivalent to saturated (Remark \ref{mh-sat-rmk}):

\begin{thm}\label{mh-sat-pf}
  Let $\K$ be an AEC, let $\lambda \ge \LS (\K)$, and assume that $\K_{\lambda}$ has amalgamation. If $M \in \K_{\ge \lambda}$ realizes all types over every substructure of size $\lambda$, then $M$ is universal over every substructure of size $\lambda$.
\end{thm}
\begin{proof}
  Similar to the proof of Theorem \ref{univ-ext-pf}: this time the objects of the construction categories are maps $M_0 \to N_0$ with $M_0 \lea M$ and $M_0, N_0$ both of size $\lambda$, and the rest of the definition is analogous.
\end{proof}

\section{A very short introduction to first-order stability}\label{fo-sec}

For the convenience of the unacquainted reader, I give a quick and self-contained construction of stable independence in the first-order case, and derive two consequences: the equivalence of stability (in terms of counting types) with no order property, and the ability to extract indiscernibles from long-enough sequences. All the material in this appendix is well known but I am not aware of a place where it appears in such compressed form. I assume a very basic knowledge of model theory, but not previous knowledge of stability theory. 

Throughout, we fix a complete first-order theory $T$ with only infinite models in a vocabulary $\tau$. For notational simplicity, we assume that $|T| = |\tau| + \aleph_0$. We fix a proper class sized ``monster model'' $\sea$ for $T$. This means that $\sea$ is universal and homogeneous, so for convenience we work inside $\sea$. We use the letters $\ba, \bb, \bc$ for (possibly infinite) sequences of elements from $\sea$, $\bx, \by, \bz$ for (possibly infinite) sequences of variables, $A, B, C$ for subsets of $\sea$, and $M, N$ for elementary substructures of $\sea$. For a sequence $\ba$, $\ran (\ba)$ denotes its set of elements (i.e.\ its range when thought of as a function). We may write $A \cup B$ instead of $AB$, $A \bb$ instead of $A \cup \ran (\bb)$, etc. Formulas are denoted by $\phi (\bx) , \psi (\bx)$, where $\bx$ is a sequence of variables that contains all free variables from $\phi$ (but may contain more -- $\phi$ always has finitely-many free variables of course). We write $\models \phi[\ba]$ instead of $\sea \models \phi[\ba]$, which means that $\phi$ holds in $\sea$ when $\ba$ replaces $\bx$ (we are very casual with arities). As usual, we often do not distinguish between $M$ and its universe $U M$. We also abuse notation by writing $\ba \in A$ instead of the more proper $\ba \in \fct{<\infty}{A}$. We write $\tp (\bb / A)$ (the type of $\bb$ over $A$) for the set of formulas $\phi (\bx, \ba)$, where $\ba \in A$ and $\models \phi (\bb, \ba)$. We write $\bb_1 \equiv_A \bb_2$ to mean that $\tp (\bb_1 / A) = \tp (\bb_2 / A)$. Recall that, by the compactness theorem, this holds if and only if there exists an automorphism $f$ of $\sea$ that fixes $A$ pointwise and sends $\bb_1$ to $\bb_2$. Thus this corresponds to the notion defined in Definition \ref{type-def}. We let $\Ss (A) := \{\tp (b / A) \mid b \in \sea\}$, and more generally $\Ss^{\alpha} (A) := \{\tp (\bb / A) \mid \bb \in \fct{\alpha}{\sea}\}$, $\Ss^{<\infty} (A) = \bigcup_{\alpha} \Ss^\alpha (A)$. Note that $|\Ss (A)| \le 2^{|T| + |A|}$, and that if $A \subseteq B$ then there is a natural surjection of $\Ss (B)$ into $\Ss (A)$ (so $|\Ss (A)| \le |\Ss (B)|$). For $p \in \Ss (B)$ and $A \subseteq B$, we write $p \rest A$ for the restriction of $p$ to $\Ss (A)$: the set of formulas from $p$ with parameters in $A$.

It is time to define the classes ``well-behaved'' of theories we will work with. As in Definition \ref{stable-def}, we will say that $T$ is \emph{$\lambda$-stable} (or \emph{stable in $\lambda$}) if for any $A$ of cardinality $\lambda$, $|\Ss (A)| = \lambda$ (of course, this is exactly the same as stability in $\lambda$ in the sense of \ref{stable-def}, in the AEC of models of $T$ ordered by elementary substructure). We say that $T$ is \emph{stable} if it is stable in some cardinal $\lambda \ge |T|$. The following closely related property will play a key role: 

\begin{defin}\label{op-def}
  $T$ has the \emph{order property} if there exists a sequence $\seq{\ba_i : i < \omega}$ and a formula $\phi (\bx, \by)$ such that $\models \phi[\ba_i, \ba_j]$ if and only if $i < j$.
\end{defin}

It turns out that $T$ is stable if and only if it does \emph{not} have the order property. We prove one direction now. The other will be dealt with after we have constructed stable independence.

\begin{thm}\label{stab-nop}
  If $T$ has the order property, then $T$ is unstable.
\end{thm}
\begin{proof}
  Fix a cardinal $\lambda \ge |T|$, and fix a linear order $I$ of cardinality $\lambda$ with strictly more than $\lambda$ Dedekind cuts (if $\lambda = \aleph_0$, the rationals are such an order; in general take $\sigma$ minimal such that $\lambda < \lambda^{\sigma}$ and the set $\fct{<\sigma}{\lambda}$ ordered lexicographically will do the trick). Fix a formula $\phi (\bx, \by)$ witnessing the order property. Using the compactness theorem, there exists a $\seq{\ba_i : i \in I}$ such that for all $i, j \in I$, $\models \phi[\ba_i, \ba_j]$ if and only if $i < j$. Since $\phi$ has finitely-many free variables we can of course assume wihout loss of generality that the $\ba_i$'s are of finite length. Now each Dedekind cut of $I$ induces a distinct type over $\bigcup_{i \in I} \ran (\ba_i)$, a set of size $\lambda$. Thus $T$ is not stable in $\lambda$.
\end{proof}

We now define what it means for two sets to be ``as independent as possible'' over a base. For simplicity, the base is required to be a model. More advanced introductions investigate what happens when the base is an arbitrary set.

\begin{defin}
  We write $\ba \nf_M \bb$, and say that \emph{$\ba$ and $\bb$ are independent over $M$}, if whenever $\bc \in M$ and $\models \phi[\ba, \bb, \bc]$, there exists $\ba' \in M$ such that $\models \phi[\ba', \bb, \bc]$.
\end{defin}

Note that if $\ba'$, $\bb'$ have the same range as $\ba$, $\ba'$ respectively, then $\ba \nf_M \bb$ if and only if $\ba' \nf_M \bb'$. Thus we will also write for example $A \nf_M B$ to mean that $\ba \nf_M \bb$ for some (equivalently any) enumerations $\ba$, $\bb$ of $A$ and $B$ respectively.

Another way of saying the same thing: if $M \subseteq B$ and $p \in \Ss^{<\infty} (B)$, we say that $p$ is \emph{free over $M$} if it is finitely satisfiable over $M$: any formula $\phi (\bx, \bb)$ in $p$ is realized in $M$. Note that $\ba \nf_M B$ if and only if $\tp (\ba / MB)$ is free over $M$. We will freely go back and forth between these two point of views (types and independence notion $\nf$). Which one is easier to work with depends on the specific concepts we are studying.

\begin{thm}[Properties of independence]\label{indep-constr}
  Assume that $T$ does not have the order property.

  \begin{enumerate}
  \item (Invariance) If $A \nf_M B$ and $f$ is an automorphism of $\sea$, then $f[A] \nf_{f[M]} f[B]$.
  \item (Normality) If $A \nf_M B$, then $A M \nf_M BM$.
  \item (Left and right monotonicity) If $A \nf_M B$ and $A_0 \subseteq A$, $B_0 \subseteq B$, then $A_0 \nf_M B_0$.
  \item (Base monotonicity) If $A \nf_M B$ and $M \preceq N \subseteq B$, then $A \nf_N B$.
  \item (Finite character) $A \nf_M B$ if and only if $A_0 \nf_M B_0$ for all finite $A_0 \subseteq A$, $B_0 \subseteq B$.
  \item (Disjointness) If $A \nf_M B$, then $A \cap B \subseteq M$.
  \item (Symmetry) $A \nf_M B$ if and only if $B \nf_M A$.
  \item (Transitivity) If $M_0 \preceq M_1 \preceq M_2$, $A \nf_{M_0} M_1$, and $A \nf_{M_1} M_2$, then $A \nf_{M_0} M_2$.
  \item (Local character) For any $A$ and $N$, there exists $M \preceq N$ of cardinality at most $|A| + |T|$ such that $A \nf_M N$.
  \item (Uniqueness) If $M \subseteq B$, $p, q \in \Ss^{<\infty} (B)$ are both free over $M$ and $p \rest M = q \rest M$, then $p = q$.
  \item (Extension) If $p \in \Ss^{<\infty} (M)$ and $M \subseteq B$ is a set, there exists $q \in \Ss^{<\infty} (B)$ that extends $p$ and is free over $B$.
  \end{enumerate}
\end{thm}
\begin{proof}
  Invariance, normality, the monotonicity properties, and finite character are immediate from the definition. To see disjointness, it is enough to see that if $a \nf_M a$ then $a \in M$. This follows from the definition applied with the formula $x = y$. Let us prove the other properties:

  \begin{itemize}
  \item \underline{Symmetry}: Suppose not. Fix $\ba$, $\bb$, and $M$ so that $\ba \nf_M \bb$ but $\bb \nnf_M \ba$. Without loss of generality, $\ba$ and $\bb$ are finite and we can pick a formula $\phi (\bx, \by)$ witnessing $\bb \nnf_M \ba$ that has all parameters from $M$ already incorporated in $\ba$: $\models \phi[\ba, \bb]$ but for all $\bb' \in M$, $\models \neg \phi[\ba, \bb']$ (we have swapped the role of $\bx$ and $\by$ for convenience in the proof that follows).

    We inductively build two sequences  $\seq{\ba_i : i < \omega}, \seq{\bb_i : i < \omega}$ of tuples in $M$ such that $\models \phi[\ba_i, \bb_j]$ if and only if $i \le j$, and $\models \phi[\ba_i, \bb]$ for all $i < \omega$.

    This is enough: set $\psi (\bx_1, \by_1, \bx_2, \by_2)$ to be $\phi (\bx_1, \by_2) \land \bx_1\by_1 \neq \bx_2\by_2$. Then $\psi$ and the sequence $\seq{\ba_i\bb_i : i < \omega}$ witness the order property, contradiction.

    This is possible: Fix $j < \omega$ and assume we are given $\ba_i$ and $\bb_i$ for all $i < j$. By the induction hypothesis, we know that:

    $$
    \models \phi (\ba, \bb) \land \bigwedge_{i < j} \phi (\ba_i, \bb) \land \bigwedge_{i < j} \neg \phi (\ba, \bb_i)
    $$

    (the last part is from the hypothesis that $\models \neg \phi (\ba, \bb')$ for any $\bb' \in M$). Because $\ba \nf_M \bb$, there exists $\ba' \in M$ such that:

    $$
    \models \phi (\ba', \bb) \land \bigwedge_{i < j} \phi (\ba_i, \bb) \land \bigwedge_{i < j} \neg \phi (\ba', \bb_i)
    $$

    Since $M \preceq \sea$, there exists $\bb' \in M$ such that:

    $$
    \models \phi (\ba', \bb') \land \bigwedge_{i < j} \phi (\ba_i, \bb') \land \bigwedge_{i < j} \neg \phi (\ba', \bb_i)
    $$

    Set $\ba_j := \ba'$, $\bb_j := \bb'$.
  \item \underline{Transitivity}: Using the definition of independence, it is easy to check the ``left'' version of transitivity: if $M_0 \preceq M_1 \preceq M_2$, $M_2 \nf_{M_1} A$, and $M_1 \nf_{M_0} A$, then $M_2 \nf_{M_0} A$. The ``right'' version of transitivity then follows from symmetry.
  \item \underline{Local character}: This is a downward Löwenheim-Skolem closure argument, that we could do explicitly. Instead, fix $A$ and $N$, and pick a pair of models $M \preceq M'$ such that $A \subseteq M'$, $|U M'| \le |A| + |T|$, and $(M, M') \preceq (N, \sea)$ (in the vocabulary with an additional predicate for $M$). From the definition of independence, it follows that $N \nf_M M'$, so by monotonicity $N \nf_M A$, hence by symmetry $A \nf_M N$.
  \item \underline{Uniqueness}: This is similar to symmetry: let $\bb$ be an enumeration of $B - M$. Suppose $p = \tp (\ba / M \bb)$, $q = \tp (\ba' / M \bb)$ both are free over $M$ (so $\ba \nf_M \bb$, $\ba' \nf_M \bb$), and $p \rest M = q \rest M$. We have to see that $p = q$. Without loss of generality, $\bb$, $\ba$, and $\ba'$ are finite. Suppose $p \neq q$, and let $\phi (\bx, \by)$ be such that $\models \phi[\ba, \bb] \land \neg \phi[\ba', \bb]$.

    Define sequences $\seq{\ba_i : i < \omega}$, $\seq{\bb_i : i < \omega}$ in $M$ such that for all $i, j < \omega$:

  \begin{enumerate}
    \item $\models \phi[\ba_i, \bb]$.
    \item $\models \phi[\ba_i, \bb_j]$ if and only if $i \le j$.
    \item\label{cond3} $\models \neg \phi[\ba, \bb_i]$.
  \end{enumerate}

  This is enough: Then $\psi (\bx_1, \by_1, \bx_2, \by_2) := \phi (\bx_1, \by_2) \land \bx_1 \by_1 \neq \bx_2 \by_2$ together with the sequence $\seq{\ba_i\bb_i : i < \omega}$ witness the order property.

  This is possible: Suppose that $\ba_i, \bb_i$ have been defined for all $i < j$. By the induction hypothesis, we have:

  $$
  \models \bigwedge_{i < j} \phi [\ba_i, \bb] \land \bigwedge_{i < j} \neg \phi [\ba, \bb_i] \land \phi [\ba, \bb]
  $$

  Since $\ba \nf_M \bb$, there is $\ba'' \in M$ such that: 

  $$
  \models \bigwedge_{i < j} \phi [\ba_i, \bb] \land \bigwedge_{i < j} \neg \phi [\ba'', \bb_i] \land \phi [\ba'', \bb]
  $$

  We also know that $\models \neg \phi[\ba', \bb]$. Combining this with the above and the fact that $\ba' \nf_M \bb$, hence by symmetry $\bb \nf_M \ba'$, we obtain a $\bb'' \in M$ such that:

  $$
  \models \bigwedge_{i < j} \phi [\ba_i, \bb''] \land \bigwedge_{i < j} \neg \phi [\ba'', \bb_i] \land \phi [\ba'', \bb''] \land \neg \phi[\ba', \bb'']
  $$

  Let $\ba_j := \ba''$, $\bb_j := \bb''$. It is easy to check that this works (for condition (\ref{cond3}), we use that $p \rest M = q \rest M$ so $\models \neg \phi[\ba', \bb'']$ implies $\models \neg \phi[\ba, \bb'']$).
\item \underline{Extension}: This is the only place where we use the compactness theorem. Consider the set $q$ of formulas $\phi (\bx, \bb)$ where $\bb \in B$, $p \cup \{\phi (\bx, \bb)\}$ is consistent, and there exists $\ba' \in M$ such that $\models \phi [\ba', \bb]$. Note that $q$ is closed under conjunctions, hence by construction and compactness is consistent. It remains to check that $q$ is complete. Indeed, assume that $\neg \phi (\bx, \bb) \notin q$. There are two cases. If $p \cup \{\neg \phi (\bx, \bb)\}$ is inconsistent, then $p \models \phi (\bx, \bb)$, so there is $\psi (\bx, \bc) \in p$ such that $\psi (\bx, \bc) \models \phi (\bx, \bb)$, with $\bc \in M$. We know that $\sea \models \exists \bx \psi (\bx, \bc)$, so $M \models \exists \bx \psi (\bx, \bc)$, hence $\phi (\bx, \bb)$ is also realized in $M$. In the second case, $p \cup \{\neg \phi (\bx, \bb)\}$ is consistent but $\neg \phi (\bx, \bb)$ is not realized in $M$. In particular, $\models \phi[\ba', \bb]$ for all $\ba' \in M$. As before, for any $\psi (\bx, \bc) \in p$, there exists $\ba' \in M$ so that $M \models \psi [\ba', \bc]$, so $M \models \psi[\ba', \bc] \land \phi[\ba', \bb]$, hence $p \cup \{\phi (\bx, \bb)\}$ is finitely consistent, hence consistent. This shows that $\phi (\bx, \bb) \in q$, as desired.
  \end{itemize}
\end{proof}

It is straightforward to check that $\nf$ as defined here yields a stable independence notion (in the sense of Definition \ref{indep-def}). The existence of an independence notion as in the theorem implies stability, thus we get:

\begin{thm}\label{nop-stab}
  If $T$ does not have the order property (or just the conclusion of Theorem \ref{indep-constr} holds), then $T$ is stable in every infinite cardinal $\lambda$ with $\lambda = \lambda^{|T|}$. In particular, $T$ is stable if and only if $T$ does not have the order property.
\end{thm}
\begin{proof}
  The ``in particular'' part will follow from Theorem \ref{stab-nop}. Now assume that $T$ does not have the order property, or just that the conclusion of Theorem \ref{indep-constr} holds. Fix an infinite cardinal $\lambda$ such that $\lambda = \lambda^{|T|}$. Since any set of cardinality $\lambda$ is contained in a model of cardinality $\lambda$, it is enough to count types over models. Fix $M$ of cardinality $\lambda$ and a sequence $\seq{p_i : i < \lambda^+}$ of types over $M$. We will show that there exists $i < j$ so that $p_i = p_j$. First, for each $i < \lambda^+$, local character tells us there exists $M_i \preceq M$ of cardinality at most $|T|$ such that $p$ is free over $M_i$. Note that $|\{M_i : i < \lambda^+\}| \le \lambda^{|T|} = \lambda$, so by the pigeonhole principle there exists $S \subseteq \lambda^+$ of cardinality $\lambda^+$ and $i_0 < \lambda^+$ such that for any $i \in S$, $p_i$ is free over $M_{i_0}$. Now $|\Ss (M_{i_0})| \le 2^{|T|} \le \lambda^{|T|} = \lambda$, so by the pigeonhole principle again, there exists $i < j$ in $S$ such that $p_i \rest M_{i_0} = p_j \rest M_{i_0}$. Since both $p_i$ and $p_j$ are free over $M_{i_0}$, uniqueness implies that $p_i = p_j$.
\end{proof}

From now on, we will freely use the equivalence between stability and no order property, as well as the conclusion of Theorem \ref{indep-constr} (but not the exact definition of independence).

\begin{lem}\label{split-lem}
  Assume that $T$ is stable. If $\ba \nf_M B$ and $\bb_1, \bb_2 \in B$ are such that $\bb_1 \equiv_M \bb_2$, then $\bb_1 \ba \equiv_M \bb_2 \ba$.
\end{lem}
\begin{proof}
  Fix an automorphism $f$ of $\sea$ fixing $M$ and sending $\bb_1$ to $\bb_2$. Let $\ba' := f (\ba)$. Then $\bb_1 \ba \equiv_M \bb_2 \ba'$. On the other hand, $\ba \nf_M \bb_1$, so $\ba' \nf_M \bb_2$. By monotonicity, also $\ba \nf_M \bb_2$, so by uniqueness, $\bb_2 \ba' \equiv_{M} \bb_2 \ba$. Combining the two equality of types shows that $\bb_1 \ba \equiv_M \bb_2 \ba$.
\end{proof}

As a final application, we will show how to extract indiscernibles in stable theories.

\begin{thm}\label{indisc-extraction}
  Assume that $\lambda > |T|$ and $T$ is stable in $\lambda$. If $\seq{a_i : i < \lambda^+}$ is a sequence, there exists $I \subseteq \lambda^+$ of cardinality $\lambda^+$ such that $\seq{a_i : i \in I}$ is indiscernible.
\end{thm}
\begin{proof}
  First, build $\seq{M_i : i \le \lambda^+}$ an increasing continuous sequence of models of size $\lambda$ such that $a_i \in M_{i + 1}$ for all $i < \lambda^+$. Let $S := \{i < \lambda^+ \mid \cf{i} \ge |T|^+\}$. This is a stationary set. By local character, for each $i \in S$ there exists $j_i < i$ such that $a_i \nf_{M_{j_i}} M_i$. By Fodor's lemma, there exists $S_0 \subseteq S$ stationary and $j^\ast < \lambda^+$ such that $j_i = j^\ast$ for all $i \in S_0$. Let's prune a bit more: by stability $|\Ss (M_{j^\ast})| = \lambda$. Thus by the pigeonhole principle we can find an unbounded $I \subseteq S$ such that for $i < i'$ in $I$, $a_{i} \equiv_{M_{j^\ast}} a_{i'}$. We prove that $\seq{a_i : i \in I}$ is indiscernible over $M := M_{j^\ast}$.

  For this, we show by induction on $n < \omega$ that for any $i_0 < \ldots < i_{n}$, $i_0' < \ldots < i_{n}'$ in $I$, $a_{i_0} \ldots a_{i_n} \equiv_M a_{i_0'} \ldots a_{i_n'}$. The base case has just been observed. Assume now that $n = m + 1$, set $\bb_1 := a_{i_0} \ldots a_{i_m}$, $\bb_2 := a_{i_0'} \ldots a_{i_m'}$, and assume we know $\bb_1 \equiv_M \bb_2$. Without loss of generality, $i_n' \le i_n$. By monotonicity, $a_{i_n} \nf_M M_{i_n'}$, so by uniqueness $a_{i_n} \equiv_{M_{i_n'}} a_{i_n'}$. In particular, $\bb_2 a_{i_n} \equiv_M \bb_2 a_{i_n'}$. By Lemma \ref{split-lem}, we also have that $\bb_1 a_{i_n} \equiv_M \bb_2 a_{i_n}$. Combining these two type equalities gives that $\bb_1 a_{i_n} \equiv_M \bb_2 a_{i_n'}$, as desired.
\end{proof}

\bibliographystyle{amsalpha}
\bibliography{bfo-accessible}

\end{document}